\documentclass[reqno]{amsart}
\pdfoutput=1\relax\pdfpagewidth=8.26in\pdfpageheight=11.69in\pdfcompresslevel=9
\usepackage{mathptmx,amsmath,mathbbol,lineno}
\usepackage{hyperref}
\usepackage{amsfonts}
\usepackage{amssymb}
\usepackage{amscd}
\usepackage{latexsym}
\usepackage{graphics}
\usepackage{float}
\usepackage[all,knot,poly,arc,web]{xy}
\usepackage{ifthen}
\usepackage{epsfig}
\usepackage{tikz}
\usetikzlibrary{decorations.fractals}
\usetikzlibrary{lindenmayersystems}
\usetikzlibrary[shadings]

\newtheorem{theorem}{Theorem}
\newtheorem{Hypotheses (H)}{Hypotheses (H)}
\newtheorem{definition}{Definition}

\newtheorem{corollary}{Corollary}
\newtheorem{lemma}{Lemma}
\newtheorem{remark}{Remark}
\newtheorem{example}{Example}

\newtheorem{prop}{Proposition}

\newtheorem{assumption}{Assumption}

\def\neweq#1{\begin{equation}\label{#1}}
\def\endeq{\end{equation}}

\def\phi{\varphi}

\date{}

\def\pp{\overrightarrow{p}(\cdot) }
\def\qq{\overrightarrow{q}(\cdot) }

\def\fun(#1,#2,#3){\mathcal{E}_{_{#1}}(#2, #3)}
\def\sob(#1,#2){W^{1,#1}(#2)}
\def\forma(#1){\fun({\pp,\qq},\cdot,\cdot)}

\def\l2{L^2(\Omega)}

\def\funn(#1,#2,#3,#4,#5){\langle J_{#1} #2, #3 \rangle_{_{#4,#5}}}
\def\nn(#1,#2){\|#1\|_{_{#2}}}

\AtBeginDocument{{\noindent\small
{\em {\color{blue}Potential Analysis} {\color{orange}(to appear)}}}
\vspace{9mm}}

\begin{document}

\title{A priori estimates for general elliptic and parabolic boundary value problems over irregular domains}

\author{Maria R. Lancia,\,\,\,Alejandro V\'elez\,-\,Santiago}

\address{Maria R. Lancia\hfill\break
Dipartimento di Scienze di Base e Applicate per I'Ingegneria\\
``Sapienza" Universit\`{a} di Roma \\ Roma, Italy}
\email{mariarosaria.lancia@uniroma1.it}

\address{Alejandro V\'elez\,-\,Santiago\hfill\break
Department of Mathematics\\
University of Puerto Rico at R\'io Piedras\\
San Juan, Puerto Rico\, 00925}
\email{alejandro.velez2@upr.edu,\,\,\,dr.velez.santiago32@gmail.com}


\numberwithin{equation}{section}

\begin{abstract}
We investigate the realization of a myriad of general local and nonlocal inhomogeneous elliptic and parabolic boundary value problems over classes of irregular regions.
We present a unified approach in which either local or nonlocal Neumann, Robin, and Wentzell boundary value problems are treated simultaneously. We establish
solvability and global regularity results for both the stationary and time-dependent heat equations governed by general differential operators with unbounded
measurable coefficients and various boundary conditions at once, first on a general framework, and then by presenting concrete important examples of irregular domains,
Wentzell-type boundary conditions, and nonlocal maps. As a consequence, we develop a priori estimates for multiple differential equations under various situations, which
are tied to a large number of applications performed over real world regions, such heat transfer, electrical conductivity, stable-like processes (probability theory),
diffusion of medical sprays in the bronchial trees, and oceanography (among many others).
\end{abstract}
\maketitle

\section{Introduction}\label{sec01}

\indent The aim of this paper is to establish the solvability and regularity theory over for a general class of local and nonlocal elliptic and parabolic boundary value problems
with various boundary conditions on a large class of ``bad" domains, which include among others, many fractal-type domains, rough domains, and ramified domains.
By a {\it ``bad" domain} we will refer to a non-Lipschitz domain, in the sense that the boundary is ``worst" than the one of a bounded Lipschitz domain.
Although Lipschitz (or smoother) domains are clear examples of regions where all the arguments presented here will work, we will concentrate mainly
on regions with more irregular structure, as the smooth case (we refer to {\it ``smooth domains"} those domains that are at least Lipschitz) has been intensely investigated,
with many results achieved. Boundary value problems over more general regions are less investigated, with many open questions, especially when one examines
the regularity theory over irregular domains. In our case, we will also deal with generalized differential operators involving measure coefficients, many which are
not required to be bounded, and in addition we will allow the inclusion of classes of nonlocal-type interior and boundary operators. For some boundary value problems,
it will be important to keep track of the boundary terms, and given the fact that the regions under consideration are in general of irregular nature,
one needs to pair the domain with a suitable boundary measure for which the corresponding problem may be well-posed. Therefore, to deal with such
complexities in the different equations of interest, we seek to develop a unified approach in which one may encapsulate all the different structures on the operators,
all the boundary conditions, and all the different domains into a single model equation.\\
\indent To begin, before posing our problems of interest, we will need to introduce a generalized notion of normal derivative (in the weak sense) for Sobolev
functions on ``bad" domains.\\

\begin{definition}\label{Def-gen-normal}
Given $\alpha_{ij},\,\hat{a}_i$ measurable functions on $\overline{\Omega}$ (for $i,\,j\in
\{1,\ldots,N\}$), let $u\in W^{1,1}_{loc}(\Omega)$
be such that $$\displaystyle\sum^N_{i,j=1}(\alpha_{ij}\partial_{x_j}u+\hat{a}_ju)\partial_{x_i}\varphi\in L^1(\Omega),
\indent\textrm{for every}\,\,\,\varphi\in C^1(\overline{\Omega}).$$ Let $\mu$ be a Borel regular measure on the boundary
$\Gamma:=\partial\Omega$. If there exists a function $f\in L^1_{loc}(\Omega)$ such that $$\displaystyle\int_{\Omega}
\displaystyle\sum^N_{i,j=1}(\alpha_{ij}\partial_{x_j}u+\hat{a}_ju)\partial_{x_i}\varphi\,dx=\displaystyle\int_{\Omega}f\varphi\,dx+
\displaystyle\int_{\Gamma}\varphi\,d\mu,\indent\textrm{for all}\,\,\,\varphi\in C^1(\overline{\Omega}),$$
then we say that $\mu$ is the {\bf generalized normal derivative} of $u$, and we denote
$$\mathcal{N}^{\ast}(\bar{\alpha}Du+\hat{\mathfrak{a}}):=\displaystyle\sum^N_{i,j=1}\mathcal{N}^{\ast}_i
(\alpha_{ij}\partial_{x_j}u+a_ju):=\mu.$$
(where $\bar{\alpha}:=(\alpha_{11},\ldots,\alpha_{ij},\ldots,\alpha_{NN}),\,\,
\hat{\mathfrak{a}}:=(\hat{a}_1,\ldots,\hat{a}_N)$), and in this case, we say that $u$ is the {\bf weak solution} of the Eq.
$$\left\{
\begin{array}{lcl}
-\displaystyle\sum^N_{i,j=1}\partial_i(\alpha_{ij}\partial_{x_j}u+\hat{a}_ju)\,=\,f\,\,\,\textrm{in}\,\,\Omega\\[2ex]
\,\,\displaystyle\sum^N_{i,j=1}\mathcal{N}^{\ast}_i(\alpha_{ij}\partial_{x_j}u+\hat{a}_ju)\,=\,1\,\,\,\textrm{on}\,\,\,\Gamma\\
\end{array}
\right.$$
\end{definition}
\indent If $\Omega$ is sufficiently smooth and $u$ is smooth enough, then we see that
$$\displaystyle\sum^N_{i,j=1}\mathcal{N}^{\ast}_i(\alpha_{ij}\partial_{x_j}u+\hat{a}_ju)=\displaystyle\sum^N_{i,j=1}
(\alpha_{ij}\partial_{x_j}u+\hat{a}_ju)\nu_i$$
for $\nu:=(\nu_1,\ldots,\nu_N)$ the outer unit normal at $\Gamma$ (with respect to the classical Surface measure on $\Gamma$).
Therefore, for the remaining of the paper, we will denote
$$\displaystyle\sum^N_{i,j=1}\mathcal{N}^{\ast}_i(\alpha_{ij}\partial_{x_j}u+\hat{a}_ju):=\displaystyle\sum^N_{i,j=1}
(\alpha_{ij}\partial_{x_j}u+\hat{a}_ju)\nu_{\mu_i},$$
where for ``bad" domains, $\nu_{\mu_i}$ stands as the interpretation of the $i$-th ``outward unit vector" in $\Gamma$
with respect to a suitable measure $\mu$ supported in $\Gamma$ (in the sense of the above definition).\\
\indent We now present more concretely the unified general structure of the differential operators under consideration. To do so, we first recall that
throughout most of this paper we will consider the set $\Omega\subseteq\mathbb{R\!}^N$ as a bounded extension domain, for $N\geq2$.
and $\mu$ to be an upper $d$-Ahlfors measure supported on $\Gamma:=\partial\Omega$, for $d\in(N-2,N)$ (see Section \ref{sec02} for these definitions).
Then, we consider the differential operators $\mathcal{A}$ and $\mathcal{B}$, formally given by
\begin{equation}
\label{1.01}\mathcal{A}u:=-\displaystyle\sum^N_{^{i,j=1}}\partial_i(\alpha_{ij}(x)\partial_{x_j}u+\hat{a}_j(x)u)
+\displaystyle\sum^N_{i=1}\check{a}_i(x)\partial_{x_i}u+\lambda(x)u+\mathcal{J}_{_{\Omega}}u
\end{equation}
and
\begin{equation}
\label{1.02}\mathcal{B}u:=\displaystyle\sum^N_{^{i,j=1}}(\alpha_{ij}(x)\partial_{x_j}u+\hat{a}_j(x)u)\nu_{\mu_i}
+\mathcal{L}_{\mu}u+\beta(x)u+\Theta_{_{\Gamma}}u\,,
\end{equation}
where
\begin{equation}\label{1.03}
\langle\mathcal{L}_{\mu}u,\varphi\rangle=\Lambda_{_{\Gamma}}(u,\varphi),
\end{equation}
that is, $\mathcal{L}_{\mu}$ is the ``second-order" operator associated with a bilinear continuous weakly coercive (not necessarily symmetric)
``Wentzell-type" form $(\Lambda_{_{\Gamma}},D(\Lambda_{_{\Gamma}}))$
(for $D(\Lambda_{_{\Gamma}})$ a subspace of $H^{r_d/2}_{\mu}(\Gamma)$; see Section \ref{sec03} for more details),
and where $\mathcal{J}_{_{\Omega}}:H^{\mathfrak{s}/2}(\Omega)\rightarrow H^{-\mathfrak{s}/2}(\Omega)$ and
$\Theta_{_{\Gamma}}:H^{r_d/2}_{\mu}(\Gamma)\rightarrow H^{-r_d/2}_{\mu}(\Gamma)$ stand as linear bounded
operators, such that $\mathcal{J}_{_{\Omega}}1=\Theta_{_{\Gamma}}1=0$,\, $\langle\mathcal{J}_{_{\Omega}}u,1\rangle_{\mathfrak{s}/2}=0=
\langle\Theta_{_{\Gamma}}v,1\rangle_{r_d/2}$, \,$\langle\mathcal{J}_{_{\Omega}}u,u\rangle_{\mathfrak{s}/2}\geq0$,
and $\langle\Theta_{_{\Gamma}}v,v\rangle_{r_d/2}\geq0$,
for each $u\in H^{\mathfrak{s}/2}(\Omega)$ and $v\in H^{r_d/2}_{\mu}(\Gamma)$ (see next section for the definition of this trace space).
Here $\alpha_{ij}\in L^{\infty}(\Omega)$,\,$\hat{a}_i\in L^{r_1}(\Omega)$, \,$\check{a}_i \in L^{r_2}(\Omega)$ ($1\leq i,\,j\leq N$),\,$\lambda
\in L^{r_3}(\Omega)$, and $\beta\in L^{s}_{\mu}(\Gamma)$, where $r_1,\,r_2>N\chi_{_{\{N>2\}}}+\chi_{_{\{N=2\}}}$,\, $r_3>N/2$,
and $s>d(d+2-N)^{-1}$. Furthermore,
$\langle\cdot,\cdot\rangle_{r_d/2}$ denotes the duality between $H^{r_d/2}_{\mu}(\Gamma)$ and
its dual space $H^{-r_d/2}_{\mu}(\Gamma)$. We also assume that $\mathcal{A}$ is uniformly strongly elliptic,
that is, there exists a constant $c_0>0$ such that
\begin{equation}
\label{1.04}\displaystyle\sum^N_{^{i,j=1}}\alpha_{ij}(x)\xi_i\xi_j\,\geq\,c_0\,|\xi|^2,
\indent\,\forall\,x\in\overline{\Omega},\,\,\,\forall\,\xi=(\xi_1,\ldots,\xi_N)\in\mathbb{R\!}^N.
\end{equation}
Hence throughout the paper we will assume that the conditions (\ref{1.01}),\, (\ref{1.02}),\, (\ref{1.03}), and (\ref{1.04}) hold.\\
\indent On the first part of the paper, consider the generalized inhomogeneous linear elliptic boundary value problem with
general (mostly nonlocal) boundary conditions, formally given by\\
\begin{equation}
\label{1.05}\left\{
\begin{array}{lcl}
\mathcal{A}u\,=\,f(x)\indent\,\textrm{in}\,\,\Omega\\
\mathcal{B}u\,=\,g(x)\indent\,\textrm{on}\,\,\Gamma\\
\end{array}
\right.
\end{equation}\indent\\
where $f\in L^p(\Omega)$ and $g\in L^q_{\mu}(\Gamma)$ for given $p,\,q\in[1,\infty]$.
We will focus our work on the following conditions on the operators $\mathcal{B}$ and $L_{\mu}$:
\begin{enumerate}
\item[{\bf (N)}]\,\,\,{\it Neumann boundary conditions:}\,\,\,If $\mathcal{L}_{\mu}=0$ and $\beta=0$;
\item[{\bf (R)}]\,\,\,{\it Robin boundary conditions:}\,\,\,If $\mathcal{L}_{\mu}=0$ and $\beta\neq0$;
\item[{\bf (W)}]\,\,\,{\it Wentzell boundary conditions:}\,\,\,If $\mathcal{L}_{\mu}\neq0$.
\end{enumerate}
We stress that the case {\bf (W)} is a delicate one, and at times we may need to assume some additional structure at the boundary
of the domain, such that the notion of tangential gradient makes sense, and that the operator $\mathcal{L}_{\mu}$ given by (\ref{1.03})
may be well defined (refer to next section for more details). Then under some conditions on $p$ and $q$, we show that a notion of weak solution
to problem (\ref{1.05}) can be defined in a suitable way, and then develop a priori estimates for weak solutions of the equation (\ref{1.05}). In particular,
if $p>N/2$ and $q>d(2+d-N)^{-1}$, then a weak solution $u\in\mathcal{W}_2(\Omega;\Gamma)$ is globally bounded, with
\begin{equation}\label{H-norm}
\max\left\{\left\|u\right\|_{_{\infty,\Omega}},\,\left\|u\right\|_{_{\infty,\Gamma}}\right\}\,\leq\,C\left(\|f\|_{_{p,\Omega}}+\|g\|_{_{q,\Gamma}}+\|u\|_{_{2,\Omega}}\right)
\end{equation} for some constant $C>0$. Moreover, if either $\lambda\in L^{r_3}(\Omega)$ or $\beta\in L^{s}_{\mu}(\Gamma)$ is ``sufficiently positive"
(see Remark \ref{coercivity}), then
\begin{equation}\label{H-norm-coercive}
\max\left\{\left\|u\right\|_{_{\infty,\Omega}},\,\left\|u\right\|_{_{\infty,\Gamma}}\right\}\,\leq\,C\left(\|f\|_{_{p,\Omega}}+\|g\|_{_{q,\Gamma}}\right).
\end{equation}
The conditions $p>N/2$ and $q>d(2+d-N)^{-1}$ are optimal in order to obtain $L^{\infty}$-estimates for the weak solutions to problem (\ref{1.05}).
However, if either $p\leq N/2$ or $q\leq d(2+d-N)^{-1}$, then some norm estimates for weak solutions are also obtained, and in addition, some maximum principles
are established.\\
\indent The second part of the paper deals with the generalized heat equation
\begin{equation}
\label{1.06}\left\{
\begin{array}{lcl}
u_t-\mathcal{A}u\,=\,f(t,x)\indent\,\textrm{in}\,\,(0,\infty)\times\Omega\\
\,\,\,\,\mathcal{B}u\,\,\,\,\,\,\,\,=\,g(t,x)\indent\,\textrm{on}\,\,(0,\infty)\times\Gamma\\
\,u(0,x)=u_0(x)\,\,\,\,\,\,\,\,\,\,\,\,\,\,\,\,\,\,\,\,\,\,\,\,\,\,\,\,\textrm{in}\,\,\Omega
\end{array}
\right.
\end{equation}\indent\\
Then given $u_0\in L^2(\Omega)$, we show that problem (\ref{1.06}) is solvable, and unique within the class of mild solutions. Furthermore,
given $\kappa_1,\,\kappa_2,\,p,\,q\in[2,\infty)$ fulfilling
$$\displaystyle\frac{1}{\kappa_1}+\displaystyle\frac{N}{2p}<1\,\,\,\,\,\,\,\,\textrm{and}
\,\,\,\,\,\,\,\,\displaystyle\frac{1}{\kappa_2}+\displaystyle\frac{d}{2q(d+2-N)}<\frac{1}{2},$$
if $u_0\in L^{\infty}(\Omega)$,\, $f\in L^{\kappa_1}(0,T;L^p(\Omega))$, and $g\in L^{\kappa_1}(0,T;L^q_{\mu}(\Gamma))$, and
if $u\in C(0,T;L^2(\Omega))\cap L^2(0,T;\mathcal{W}_2(\Omega;\Gamma))$ is a weak solution of problem (\ref{1.06}), then
\begin{equation}\label{P-norm}
\|u\|_{_{L^{\infty}(0,T;L^{\infty}(\Omega))}}
\,\leq\,C'\left(\|u_0\|_{_{\infty,\Omega}}+\|f\|_{_{L^{\kappa_1}(0,T;L^{p}(\Omega))}}+
\|g\|_{_{L^{\kappa_2}(0,T;L^{q}_{\mu}(\Gamma))}}\right),
\end{equation}
for some constant $C'>0$. Moreover, if problem (\ref{1.06}) is local in time, then we show that the corresponding weak solution is globally bounded in space, with
\begin{equation}\label{P-norm-global}
\max\left\{\|u\|_{_{L^{\infty}(T_0,T;L^{\infty}(\Omega))}},\,\|u\|_{_{L^{\infty}(T_0,T;L^{\infty}_{\mu}(\Gamma))}}\right\}
\,\leq\,C'\left(\|u_0\|_{_{2,\Omega}}+\|f\|_{_{L^{\kappa_1}(0,T;L^{p}(\Omega))}}+
\|g\|_{_{L^{\kappa_2}(0,T;L^{q}_{\mu}(\Gamma))}}\right)
\end{equation}
whenever $0<T_0<T$ and $u_0\in L^2(\Omega)$.\\
\indent When a domain is ``bad enough,"
solvability results and a priori estimates for boundary value problems with either Neumann, Robin, or Wentzell boundary conditions, constitute
a much more delicate situation, as some problems of these types may fail to be well-posed. In \cite{AR-WAR03} it is established
a necessary and sufficient condition for a Robin problem to be well-posed. Such condition is closely related to the selection of an appropriate
measure on the boundary, called an {\it admissible measure} (refer to \cite{AR-WAR03} for the definition and more details).\\
\indent Under the above structure, for the case when the differential operator $\mathcal{A}$ given by (\ref{1.01}) is of Laplace-type (or symmetric with bounded coefficients),
the solvability and regularity theory for local and nonlocal elliptic boundary value problems of the kind considered in this paper
(over non-smooth domains) have been considered by many authors;
see for instance \cite{BASS-CHEN08,BIE-WAR,BOU-AVS18,FUK-TOM96,FUK-TOM95,LAD-URAL68,LAN-VELEZ-VER18-1,LAN-VELEZ-VER15-1,LAP-PANG95,LEM-MIL2011,LEM-MIL-SP2013,MAZ68,Str,VELEZ2013-1,ALEJO-PhD10,AVS-WAR,WAR14-L,WAR12,WAR02} (among others). In the case of local non-symmetric operators of type
(\ref{1.01}) and (\ref{1.02}) with bounded coefficients, Daners \cite{DAN} established a priori estimates for weak solution under Robin-type boundary conditions over
arbitrary bounded domains, and finer regularity results were obtained by Nittka \cite{NITTKA,NITTKA2010} over bounded Lipschitz domains, with a nonlocal counterpart obtained in \cite{VELEZ10.1,ALEJO-PhD10}. Now, the situation in which
the coefficients of the lower-order terms in (\ref{1.01}) and (\ref{1.02}) are of unbounded nature is less known, but is motivated in the pioneer works
by Stampacchia \cite{STAM65} and of Ladyzhenskaya and Ural'tseva \cite{LAD-URAL68} (for the Dirichlet problem over Lipschitz-like domains, and domains with the cone condition).
Under similar framework (bounded Lipschitz domains), Nittka \cite{NITTKA2013-1,NITTKA2010} developed fine regularity results for general classes of
Robin-type problems involving unbounded coefficients, and recently in \cite{APU-NAZ-PAL-SOF21}, the authors established generalized solvability results
for a class of Wentzell-type boundary value problem of type (\ref{1.01}).\\
\indent Concerning the realization of inhomogeneous parabolic problems with unbounded coefficients, there is much less literature, and almost non-existence
of literature in the case of non-smooth domains. One can
recall first the results in \cite{LAD-URAL68} over Lipschitz-type domains (for the Dirichlet problem), and their generalization to a broader class of domains recently achieved in \cite{KIM-RYU-WOO22} (for the Dirichlet problem), where they deal with inhomogeneous interior differential equation with homogeneous boundary conditions. For Neumann and Robin boundary value problems,
Nittka \cite{NITTKA2014} developed a very useful procedure whose motivation resides in \cite{LAD-URAL68}, from which he was able to establish the solvability
and global regularity for local boundary value problems of type (\ref{1.01}) and (\ref{1.02}) (for $\mathcal{L}_{\mu}=\mathcal{J}_{_{\Omega}}=\Theta_{_{\Gamma}}=0$)
with unbounded lower-order terms over bounded Lipschitz domains. In \cite{CREO-LAN-NAZ20-1,CREO-LAN-NAZ-VER19-1}, solvability results and some $W^{1,2}$-estimates for weak solutions to an inhomogeneous parabolic equation of Wentzell-type were obtained over
$2$-dimensional piecewise smooth domains, which included a nonlocal boundary Besov operator of type (\ref{Besov-Boundary-Map}). For smoother regions, solvability results for a class of parabolic Wentzell-type boundary value problem of type (\ref{1.06}) were obtained in \cite{APU-NAZ-PAL-SOF22}.\\
\indent In the following paper, we deal with the solvability and global regularity theory for the boundary value problems (\ref{1.05}) and (\ref{1.06})
under their full generality over a large class of irregular regions. The model equations (\ref{1.05}) and (\ref{1.06}) allow us to deal simultaneously
Neumann, Robin, and Wentzell-type boundary conditions, local or nonlocal problems in both the interior and the boundary. The class of nonlocal maps is given
in a general form, which include well-known important examples such as Besov-type maps, and the Dirichlet-to-Neumann map (see section \ref{sec06}).
Furthermore, in the structure of the operators $\mathcal{A}$ and $\mathcal{B}$ given by (\ref{1.01}) and (\ref{1.02}), respectively, we are not assuming symmetry
in the bounded coefficient $\alpha_{ij}$, and the lower-order coefficients are possibly unbounded. These generalities produce substantial complications throughout all
calculations and a priori estimates.\\
\indent To begin with the framework and foundation of our problem, we follow the standard approach employed by Daners \cite{DAN} and Nittka \cite{NITTKA},
which is based in the analysis of the corresponding bilinear forms which give a variational flavor of the problem. In our situation, we consider the (possibly
non-symmetric) form\\[2ex]
$\mathcal{E}_{\mu}(u,v):=$\\
$$\displaystyle\int_{\Omega}\left(\displaystyle\sum^N_{^{i,j=1}}(\alpha_{ij}(x)\partial_{x_j}u+
\hat{a}_i(x)u)\partial_{x_i}v+\displaystyle\sum^N_{^{i=1}}\check{a}_i(x)\partial_{x_i}uv+\lambda(x)uv\right)\,dx+
\Lambda_{_{\Gamma}}(u,v)+\displaystyle\int_{\Gamma}\beta(x)uv\,d\mu+
\left\langle\mathcal{J}_{_{\Omega}}u,v\right\rangle_{s/2}+\left\langle\Theta_{_{\Gamma}}u,v\right\rangle_{r_d/2},$$
which is directly related with the variational formulation of the elliptic problem (\ref{1.05}). The continuity and weak coercivity of this bilinear form provides a well-defined notion of a weak solution of problem (\ref{1.05}), and also highlights the possible environments in which a weak solution may be obtained. Then, by following an approach derived from a combination of
Moser's iteration-type schemes and variants of a modified refined of De Giorgi's technique, we derive general a priori estimates for weak solutions to the elliptic
problem (\ref{1.05}) under minimal assumptions on the interior and boundary data.\\
\indent Turning our attention into the inhomogeneous parabolic problem of type (\ref{1.06}), firstly, from the theory of forms, it follows that the operator associated with the form $\mathcal{E}_{\mu}(\cdot,\cdot)$ generates a compact holomorphic $C_0$-semigroup on $L^2(\Omega)$, which extrapolates to a family of holomorphic semigroups on $L^p(\Omega)$ for all $p\in[1,\infty]$, with the same angle of holomorphy. The trajectories of this semigroup are the unique mild solutions of the homogeneous parabolic problem (\ref{1.06}) (for $f=g\equiv0$; see subsection \ref{subsec04-01}).
From here we transition to the inhomogeneous problem (\ref{1.06}) in its full generality, by following the strategy employed by Nittka \cite{NITTKA2014},
with the necessary modifications and extensions in order to cover all the boundary conditions, generalities in the regions, and possible non-localities.
After deducing existence and uniqueness results for weak solutions to the inhomogeneous heat equation (\ref{1.06}), we develop a generalized method
based on a parabolic De Giorgi's technique (first discovered in \cite{LAD-URAL68} and extended to the Robin problem in \cite{NITTKA2014}) to
obtain local-global a priori $L^{\infty}$-estimates, leading to the fulfillment of inequalities (\ref{P-norm}) and (\ref{P-norm-global}).
Crucial novelties in problems (\ref{1.05}) and (\ref{1.06}) are the presence of a general ``Wentzell-type" boundary operator as in (\ref{1.03}),
the inclusion of nonlocal interior and boundary maps $\mathcal{J}_{_{\Omega}}$ and $\Theta_{_{\Gamma}}$, respectively (which are
defined generally), and the fact that the class of regions include many irregular domains.\\
\indent Concerning the Wentzell problem, we present two major examples. We first consider the case when the boundary is a compact Riemannian manifold, in which one can define a generalized Laplace-Beltrami-type operator $$\mathcal{L}_{\mu}u:=-\displaystyle\sum^{N-1}_{^{i,j=1}}\partial_{x_{\tau_i}}\left(\omega_{ij}(x)\partial_{x_{\tau_j}}u+\hat{b}_i(x)u\right)
+\displaystyle\sum^{N-1}_{i=1}\check{b}_i(x)\partial_{x_{\tau_i}}u$$ in (\ref{1.03}), for $\partial_{x_{\tau_j}}u$ denoting
the directional derivative of $u$ along the tangential directions $\tau_j$ at each point on $\Gamma$. Here again, the lower-order coefficients can be chosen
unbounded. For smooth domains and for $\hat{b}_j\equiv0$, the realization of problems involving a Wentzell boundary condition of the above structure was
recently considered in \cite{APU-NAZ-PAL-SOF22,APU-NAZ-PAL-SOF21}, but no global a priori estimates have been developed. We also consider the realization of an abstract Wentzell problem over the classical Koch snowflake domain (see Example \ref{Ex5.02}), in which the
boundary is no longer a Riemannian manifold. For the case of Laplace-type operators, this problem is known and substantial literature is known, with direct
applications to transmission problems; see for instance \cite{LAN-CEF-DELL12,LAN-VELEZ-VER18-1,LAN-VELEZ-VER15-1,LAN-VER14-2,LAN-VER12-3,LAN-VER10-1} (among others).
However, no literature is known for the realization of this problem involving non-symmetric operators and unbounded coefficients, and moreover, the
inhomogeneous parabolic problem with this boundary condition has never been considered before.\\
\indent The realization of problems (\ref{1.05}) and (\ref{1.06}) admit the presence of a general class of interior and boundary nonlocal maps,
so in this case we provide two main examples of nonlocal maps fulfilling the required properties. The first example corresponds to Besov-type interior
and boundary operators, which are well known and very much related with the transition to fractional differential equations (see \cite{WAR15-1}) and stable-like processes
(probability theory; see \cite{CHEN-KUM03}). We also include the Dirichlet-to-Neumann map, which is a classical nonlocal boundary map known in
smooth regions (see for instance \cite{JONN99,JONS97,SILV-UHLM} and the references therein). Arendt and Elst \cite{AR-ELST11} extended the notion of
a Dirichlet-to-Neumann map to general domains whose boundaries have positive finite $(N-1)$-dimensional Hausdorff measure. In our case, motivated by the
approach in \cite{AR-ELST11}, we develop a generalized notion of a Dirichlet-to-Neumann map involving more general measures at the boundary, and show that this map satisfies all the required properties. Consequently, a priori estimates for boundary value problems
involving the Dirichlet-to-Neumann boundary map can be obtained, which is something not done before. 
At the end, we also present a new realization of boundary value problems of type
(\ref{1.05}) and (\ref{1.06}) with nonlocal data of Besov structure. For the elliptic case
involving the Laplacian, this was introduced in \cite{VELEZ2013-1} (only in the boundary), but the solvability and regularity theory of
an inhomogeneous parabolic problem whose data is of nonlocal nature is a completely new valuable result, presented here for the first time.\\
\indent The paper is organized as follows. Section \ref{sec02} showcases the functional setting, with the
necessary definitions and notations, and also presents known embedding theorems and analytical results that
will be applied consistently in the subsequent sections. At the end
of this part, the first required assumption is stated. Section \ref{sec03} is devoted to the solvability and regularity theory of the elliptic problem (\ref{1.05}). After giving a sense to a proper definition for a weak solution to problem (\ref{1.05}), we proceed to establish a priori estimates for weak solutions to the equation (\ref{1.05}). The general estimates are refined in the case when the form $\mathcal{E}_{\mu}(\cdot,\cdot)$ is coercive. The section ends with some results related with inverse positivity and maximum principles, which showcase the continuous dependence of the weak solutions with the inhomogeneous data.\\
\indent In section \ref{sec04}, we turn our attention to the inhomogeneous parabolic problem (\ref{1.06}). Using the well-known semigroup approach
(coming from the bilinear form $\mathcal{E}_{\mu}(\cdot,\cdot)$), we briefly establish the well-posedness of the homogeneous parabolic problem, with some
basic properties that will be carried out throughout the whole section. From here, by following a similar approach as in \cite{NITTKA2014}, we
develop a way to define and establish the existence of weak solution to the inhomogeneous problem (\ref{1.06}). Several major modifications needed to
be taken, due to the fact that the model equations include nonlocal
maps and possibly Wentzell boundary conditions. We finish this section
by establishing the main results of the section, which consist on a priori $L^{\infty}$-estimates, and the fulfillment of inequalities (\ref{P-norm}) and
(\ref{P-norm-global}).\\
\indent Since the boundary value problems (\ref{1.05}) and (\ref{1.06}) involve several general operators and general domains, it is important to present concrete
examples of regions and the corresponding operators acting in these equations. Consequently, sections \ref{sec05}, \ref{sec06}, and \ref{sec07} will be devoted to
such examples. Section \ref{sec05} presents multiple examples of irregular regions considered by multiple authors, which are the ground for important applications.
It is important to address that all the results established in this paper are under the main assumption that the domains posses the extension property
(e.g. Definition \ref{ext-domain}), but in this section we will present several examples of domains in which the results of the paper may hold, but they are not
extension domains. In section \ref{sec06} we present two major examples of Wentzell-type boundary operators, and in section \ref{sec07} we give two main examples of
nonlocal maps, together with some final results regarding the approximation of weak solutions in terms of their corresponding nonlocal maps, as well
as the solvability and global regularity of boundary value problems (\ref{1.05}) and (\ref{1.06}) in the case when the data is of nonlocal nature.\\

\section{Preliminaries and intermediate results}\label{sec02}

\subsection{Functions spaces and basic definitions}\label{subsec02-01}

\indent In this section we review some fundamental properties of the Sobolev and Besov
spaces on our domains of interest, an we present some measure and analytical results that will be applied
on the subsequent sections.\\
\indent As usual, $L^p_{\mu}(E):=L^p(E;d\mu)$ denotes the $L^p$-space on $E$ with respect to the measure $\mu$; if $\mu$ is the $N$-dimensional Lebesgue measure,
then we write $L^p_{\mu}(E)=L^p(E)$. By $W^{1,p}(\Omega)$ we mean the well-known $L^p$-based Sobolev space. Also, denoting $\Gamma:=\partial\Omega$
the boundary of $\Omega$, and letting $\mu=\mu(\mathfrak{g})$ be a (suitable) finite Borel regular measure supported on $\Gamma$, assume that $\Gamma$ is a compact
Riemannian manifold, with Riemannian metric $\mathfrak{g}$. Then we define the space $W^{1,p}(\Gamma)$
as the completion of the space $C^1(\Gamma)$ with respect to the norm
$$\|u\|^p_{_{W^{1,p}(\Gamma)}}:=\|u\|^p_{_{p,\Gamma}}+\|\nabla_{^{\Gamma}}u\|^p_{_{p,\Gamma}},$$
where $\nabla_{^{\Gamma}}$ denotes the tangential gradient (also called Riemannian gradient) on $\Gamma$. We recall that
in this case we are assuming more regularity in the boundary, such that the above definition makes sense, and such that if $p=2$ and
$\varphi,\,\psi$ are smooth functions, one has
\begin{equation}\label{tangential}
-\displaystyle\int_{\Gamma}(\nabla_{^{\Gamma}}\cdot\nabla_{^{\Gamma}}\varphi)\psi\,d\mu=
\displaystyle\int_{\Gamma}\nabla_{^{\Gamma}}\varphi\nabla_{^{\Gamma}}\psi\,d\mu.
\end{equation}
Here $\mu=\mu(\mathfrak{b})$ stands as the measure induced by the Riemannian metric $\Bbbk$ on $\Gamma$. When $p=2$, we will denote the Sobolev spaces
previously discussed by $H^1(\Omega)$ and $H^1_{\mu}(\Gamma)$, respectively. In addition, in the case when $\Gamma$ is not a Riemannian manifold,
we follow an abstract approach, by considering $(\Lambda_{_{\Gamma}},D(\Lambda_{_{\Gamma}}))$
a bilinear continuous weakly coercive form such that $\left(D(\Lambda_{_{\Gamma}}),\|\cdot\|_{_{D(\Lambda_{_{\Gamma}})}}\right)$
is a Hilbert space, we define the Hilbert space
$$H^1_{\mu}(\Omega;\Gamma):=\left\{u\in H^{1}(\Omega)\mid u|_{_{\Gamma}}\in D(\Lambda_{_{\Gamma}})\right\},$$
endowed with the norm $$\|\cdot\|^2_{_{H^1_{\mu}(\Omega;\Gamma)}}:=\|\cdot\|^2_{_{H^{1}(\Omega)}}+\|\cdot\|^2_{_{D(\Lambda_{_{\Gamma}})}}.$$
In some situations, we will work with the space
$$W_0(\Omega):=\left\{u\in W^{1,p}(\Omega)\mid\int_{\Omega}u\,dx=0\right\}.$$
It is well-known that $W_0(\Omega)$ is a closed subspace of $W^{1,p}(\Omega)$, endowed with the norm
$$\|\cdot\|_{_{W_0(\Omega)}}:=\|\cdot\|_{_{W^{1,p}(\Omega)}}\,.$$
We will also consider the quotient space $W^{1,p}(\Omega)/\mathbb{R\!}\,$, endowed with the norm
$$\|u\|_{_{W^{1,p}(\Omega)/\mathbb{R\!}}}:=\displaystyle\min_{^{v\in W^{1,p}(\Omega),\, u-v=k}}\|v\|_{_{W^{1,p}(\Omega)}}\,,$$
for $u\in W^{1,p}(\Omega)/\mathbb{R\!}$\,, and where $k\in\mathbb{R\!}\,$ denotes a constant. It is known that
$W^{1,p}(\Omega)/\mathbb{R\!}$\, is a Banach space, because it is isomorphic to $W_0(\Omega)$. We recall that two functions
$u,\,v\in W^{1,p}(\Omega)/\mathbb{R\!}\,$ are identical if and only if $u-v=k$, for some constant $k$. Again, when
$p=2$, we will denote $H_0(\Omega):=W_0(\Omega)$, and moreover, in this case, we will also be interested in the space
$$H_{0,\mu}(\Omega;\Gamma):=\left\{u\in H_0(\Omega)\mid u|_{_{\Gamma}}\in D(\Lambda_{_{\Gamma}})\right\}.$$
Also, for $r,\,s\in[1,\infty)$, or $r=s=\infty$, we will sometimes refer to the Banach Space defined in \cite{AVS-WAR}, i.e.,
$$\mathbb{X\!}^{\,r,s}(\Omega;\Gamma):=\left\{(f,g)\,:\,f\in L^r(\Omega),\,\,g\in L^s_{\mu}(\Gamma)\right\},$$
endowed with norm
$$|\|(f,g)\||_{_{r,s}}:=\|f\|_{_{r,\Omega}}+\|g\|_{_{s,\Gamma}}\,,
\indent\,\textrm{if}\,\,\,r,\,s\in[1,\infty),$$
and
$$|\|(f,g)\||_{_{\infty,\infty}}:=
\max\left\{\|f\|_{_{\infty,\Omega}},\,\|g\|_{_{\infty,\Gamma}}\right\}.$$
\indent\\
If $r=s$, then we will write $\mathbb{X\!}^{\,r}(\Omega,\Gamma):=\mathbb{X\!}^{\,r,r}(\Omega,\Gamma)$, and
$|\|(f,g)\||_{_{r}}:=|\|(f,g)\||_{_{r,r}}$. Also if $u\in L^p(\Omega)$ is such
that $u|_{_{\Gamma}}\in L^p_{\mu}(\Gamma)$, we write $\mathbf{u}:=(u,u|_{_{\Gamma}})$.
Finally, given $p\geq1,\,s\in(0,1),\,d\in(0,N)$, let $F\subseteq\mathbb{R\!}^N$ be a compact set, and
$\mu$ a Borel regular measure supported on $F$. Then we define the {\it Besov spaces} $\mathbb{B\!}^{\,p}_{\,s}(\Omega)$
and $\mathbb{B\!}^{\,p}_{_{\,1-\frac{N-d}{p}}}(F,\mu)$ as follows:
$$\mathbb{B\!}^{\,p}_{\,\mathfrak{s}}(\Omega):=\left\{u\in L^p(\Omega)\mid\mathcal{M}^p_{\mathfrak{s}}(u,\Omega)<\infty\right\},$$
and
$$\mathbb{B\!}^{\,p}_{_{\,1-\frac{N-d}{p}}}(F,\mu):=\left\{u\in L^p_{\mu}(F)\mid\mathcal{N}^p_{_{1-\frac{N-d}{p}}}(u,F,\mu)<\infty\right\},$$
where
\begin{equation}\label{nonlocal-interior}
\mathcal{M}^p_{\mathfrak{s}}(u,\Omega):=\left(\displaystyle\int_{\Omega}\int_{\Omega}\left[\frac{|u(x)-u(y)|}
{|x-y|^{\mathfrak{s}+\frac{N}{p}}}\right]^p\,dxdy\right)^{\frac{1}{p}}
\end{equation}
and
\begin{equation}\label{nonlocal-boundary}
\mathcal{N}^p_{_{1-\frac{N-d}{p}}}(u,F,\mu):=\left(\displaystyle\int_{F}\int_{F}\left[\frac{|u(x)-u(y)|}
{|x-y|^{1+\frac{2d-N}{p}}}\right]^p\,d\mu_xd\mu_y\right)^{\frac{1}{p}}.
\end{equation}
These Besov spaces are Banach spaces with respect to their respective norms
$$\|u\|_{_{\mathbb{B\!}^{\,p}_{\,\mathfrak{s}}(\Omega)}}:=\|u\|_{_{p,\Omega}}+\mathcal{M}^p_{\mathfrak{s}}(u,\Omega)\,\,\,\,\,\textrm{and}\,\,\,\,\,
\|u\|_{_{\mathbb{B\!}^{\,p}_{_{\,1-\frac{N-d}{p}}}(F,\mu)}}:=\|u\|_{_{p,F}}+\mathcal{N}^p_{_{1-\frac{N-d}{p}}}(u,F,\mu).$$
If $p=2$, we will write $H^{\mathfrak{s}/2}(\Omega):=\mathbb{B\!}^{\,p}_{\,\mathfrak{s}}(\Omega)$ and
$H^{r_d/2}_{\mu}(F):=\mathbb{B\!}^{\,p}_{\,r_d/2}(F,\mu)$, for $r_d:=2-N+d$, and
we will write $H^{-\mathfrak{s}/2}(\Omega)$ and $H^{-r_d/2}_{\mu}(F)$ for the dual spaces (related to the corresponding Besov spaces).
For more information regarding these spaces, refer to \cite{ADA,JO-WAL,HEB2000,HEB96,JOST,MAZ,MAZ-POB97,NECAS,TAR07,TRIE-97,ZIE89} (among many others).\\
\indent Next, we give the following two standards definitions (e.g. \cite{BIE09,CHEN-KUM03,DLL06,JON}).\\

\begin{definition}\label{ext-domain}
Let $p\in[1,\infty]$. A domain
$\Omega\subseteq\mathbb{R\!}^N$ is called a $W^{1,p}$-{\bf extension domain}, if there exists a bounded linear operator
$S:W^{1,p}(\Omega)\rightarrow W^{1,p}(\mathbb{R\!}^N)$ such that $Su=u$\, a.e. on $\Omega$. If $p=2$ and the latter property holds for
$\Omega$, we will call it simply an {\it extension domain}.
\end{definition}

\begin{definition}\label{Ahlfors}
Let $d\in(0,N)$ and $\mu$ a Borel measure supported on a bounded set $F\subseteq\mathbb{R\!}^N$.
Then $\mu$ is said to be an {\bf upper $d$-Ahlfors (regular) measure}, if there exist constants $M_0,\,R_0>0$ such that
\begin{equation}\label{2.01}
\mu(F\cap B(x,r))\,\leq\,M_0r^d,\indent\,\,\,\textrm{for all}\,\,\,0<r<R_0\,\,\,\textrm{and}\,\,\,x\in F,
\end{equation}
where $B(x,r)$ denotes the ball of radius $r$ centered at $x\in F$. On the other hand, the measure $\mu$ as called a
{\bf lower $d$-Ahlfors measure}, if the reverse inequality in (\ref{2.01}) is fulfilled. Also, if the inequality (\ref{2.01})
(or its reverse) hold, then we call $F$ an {\bf upper (lower) $d$-set} (with respect to the measure $\mu$).
Finally, we call $\mu$ a $d$-{\bf Ahlfors measure},
if $\mu$ is both an upper and lower $d$-Ahlfors measure, and in this case the set $F\subseteq\mathbb{R\!}^N$ will be called a
{\bf $d$-set} (with respect to the measure $\mu$; see for instance \cite{CHEN-KUM03,JO-WAL}).
\end{definition}

\begin{remark}
When $\Omega\subseteq\mathbb{R\!}^N$ is bounded, it is obvious that the inequality
$|B(x,r)\cap\Omega|\leq cr^N$ is valid for some constant $c>0$ and for every $x\in\Omega$ and $r>0$.
If in addition $\Omega$ is a bounded $W^{1,p}$-extension domain, then it is well-known that $\Omega$ satisfies
the so called {\bf measure density condition}, that is, the inequality
$$|B(x,r)\cap\Omega|\,\geq\,c' \,r^N$$
holds some constant $c' >0$, and for all $x\in\Omega$ and $r\in(0,1]$ (e.g. \cite{SHV07}). Thus, one sees that in this
case the set $\Omega$ is a $N$-set with respect tho the Lebesgue measure $|\cdot|$. Moreover, if $\Omega$ is a
$W^{1,p}$-extension domain, then $|\partial\Omega|=0$.
\end{remark}

\subsection{Key embedding results}\label{subsec02-02}

\indent To begin, we collect some key important known embedding results for ``bad" domains, which are taken from
\cite{BIE09,HAJLASZ-KOS-TUO08,HEB2000,HEB96,WAR15-1,ZIE89} (and the references therein). We put them all together in the following result.\\

\begin{theorem}\label{embeddings}
\,(see \cite{BIE09,HAJLASZ-KOS-TUO08,HEB2000,HEB96,WAR15-1,ZIE89}) The following hold:
\begin{enumerate}
\item[$\bullet$]\,\,\,If $\Omega\subseteq\mathbb{R\!}^N$ is a bounded $W^{1,p}$-extension domain,
and $p\in[1,N)$, then there exists a continuous embedding
$W^{1,p}(\Omega)\hookrightarrow L^{^{\frac{Np}{N-p}}}(\Omega)$ and a constant $c_1>0$ such that
\begin{equation}\label{2.03}
\|u\|_{_{\frac{Np}{N-p},\Omega}}\,\leq\,c_1\|u\|_{_{W^{1,p}(\Omega)}},
\indent\,\textrm{for every}\,\,\,u\in W^{1,p}(\Omega).
\end{equation}
Furthermore, for $p\in[1,N)$ and $q\in[1,Np(N-p)^{-1})$, the embedding $W^{1,p}(\Omega)\hookrightarrow L^{q}(\Omega)$ is compact, and
if $p\geq N$, then the embedding $W^{1,p}(\Omega)\hookrightarrow L^{q}(\Omega)$ is compact for all $q\in[1,\infty)$.
\item[$\bullet$]\,\,\,Let $\Omega\subseteq\mathbb{R\!}^N$ be a bounded $W^{1,p}$-extension domain,
and let $\mu$ be an upper $d$-Ahlfors measure supported on $\Gamma$, for $d\in(N-p,N)\cap(0,N)$.
If $p\in(1,N)$, then there exists a linear mapping $W^{1,p}(\Omega)\hookrightarrow
L^{^{\frac{dp}{N-p}}}_{\mu}(\Gamma)$ and a constant $c_2>0$ such that
\begin{equation}\label{2.04}
\|u\|_{_{\frac{dp}{N-p},\Gamma}}\,\leq\,c_2\|u\|_{_{W^{1,p}(\Omega)}},
\indent\,\textrm{for all}\,\,\,u\in W^{1,p}(\Omega).
\end{equation}
Moreover, for $p\in(1,N)$ and $q\in[1,dp(N-p)^{-1})$, the trace map $W^{1,p}(\Omega)\hookrightarrow L^{q}_{\mu}(\Gamma)$ is compact, and
if $p\geq N$, then the trace embedding $W^{1,p}(\Omega)\hookrightarrow L^{q}_{\mu}(\Gamma)$ is compact for all $q\in[1,\infty)$.
\item[$\bullet$]\,\,\,Let $\Omega\subseteq\mathbb{R\!}^N$ be a bounded $W^{1,p}$-extension domain whose boundary is
a compact $(N-1)$-dimensional Riemannian maniforld ($N\geq3$),
and let $\mu=\mu(\mathfrak{g})$ be the measure on $\Gamma$ induced by the Riemannian metric $\mathfrak{g}$.
If $p\in(1,N-1)$, then there exists an embedding $W^{1,p}(\Gamma)\hookrightarrow
L^{^{\frac{(N-1)p}{N-1-p}}}_{\mu(\mathfrak{g})}(\Gamma)$ and a constant $c_3>0$ such that
\begin{equation}\label{2.04a}
\|u\|_{_{\frac{(N-1)p}{N-1-p},\Gamma}}\,\leq\,c_3\|u\|_{_{W^{1,p}(\Gamma)}},
\indent\,\textrm{for all}\,\,\,u\in W^{1,p}(\Gamma).
\end{equation}
Furthermore, for $p\in(1,N-1)$ and $q\in[1,(N-1)(N-1-p)^{-1})$, the embedding $W^{1,p}(\Gamma)\hookrightarrow L^{q}_{\mu(\mathfrak{g})}(\Gamma)$ is compact, and
if $p\geq N-1$, then the embedding $W^{1,p}(\Gamma)\hookrightarrow L^{q}_{\mu(\mathfrak{g})}(\Gamma)$ is compact for all $q\in[1,\infty)$.
\item[$\bullet$]\,\,\,Let $\Omega\subseteq\mathbb{R\!}^N$ be a bounded $W^{1,p}$-extension domain.
If $p\in(1,N)$ and $\mathfrak{s}\in(0,1)$, then there exists a linear mapping $W^{1,p}(\Omega)\hookrightarrow
\mathbb{B\!}^{\,p}_{\,\mathfrak{s}}(\Omega)$ and a constant $c_{4}>0$ such that
\begin{equation}\label{2.04b}
\|u\|_{_{\mathbb{B\!}^{\,p}_{\,\mathfrak{s}}(\Omega)}}\,\leq\,c_{4}\|u\|_{_{W^{1,p}(\Omega)}},
\indent\,\textrm{for all}\,\,\,u\in W^{1,p}(\Omega).
\end{equation}
\item[$\bullet$]\,\,\,Let $\Omega\subseteq\mathbb{R\!}^N$ be a bounded $W^{1,p}$-extension domain,
and let $\mu$ be an upper-$d$-Ahlfors measure supported on $\Gamma$.
If $p\in(1,N)$ and $d\in(N-p,N)$, then there exists a linear mapping $W^{1,p}(\Omega)\hookrightarrow
\mathbb{B\!}^{\,p}_{_{\,1-\frac{N-d}{p}}}(\Gamma,\mu)$ and a constant $c_{5}>0$ such that
\begin{equation}\label{2.05}
\|u\|_{_{\mathbb{B\!}^{\,p}_{_{\,1-\frac{N-d}{p}}}(\Gamma,\mu)}}\,\leq\,c_{5}\|u\|_{_{W^{1,p}(\Omega)}},
\indent\,\textit{for all}\,\,\,u\in W^{1,p}(\Omega).
\end{equation}
\item[$\bullet$]\,\,\,If $\Omega\subseteq\mathbb{R\!}^N$ is a bounded $W^{1,p}$-extension domain,
and $p\in(1,\infty)$, then there exists a constant $c_6>0$ such that
\begin{equation}\label{2.06}
\|u\|_{_{W^{1,p}(\Omega)}}\,\leq\,c_6\|\nabla u\|_{_{p,\Omega}},
\end{equation}
for every $u\in W^{1,p}(\Omega)$ that integrates to zero over a measurable set $E\subseteq\Omega$ with $|E|>0$. In particular,
(\ref{2.06}) holds for all $u\in W_0(\Omega)$.\\
\end{enumerate}
\end{theorem}

\begin{remark}\label{R-embedding}
The following consequences are useful for the subsequent sections:
\begin{enumerate}
\item[(a)]\,\,\,If $p=2$, then it is obvious that the inequalities (\ref{2.03}), (\ref{2.04}), (\ref{2.04b}), (\ref{2.05}), and (\ref{2.06}),
are valid if one replace the space $H^{1}(\Omega)$ by the space $H^1_{\mu}(\Omega;\Gamma)$.
\item[(b)]\,\,\,If $\Omega\subseteq\mathbb{R\!}^N$ is a bounded $W^{1,p}$-extension domain
and $\mu$ is an upper-$d$-Ahlfors measure supported on $\Gamma$, then from (\ref{2.04}) together with an argument in
\cite[Theorem 4.24]{BIE10.1}, it follows that the norm
$$|\|u\||^p_{_{\mathcal{W}^{1,p}(\Omega)}}:=\|\nabla u\|^p_{_{p,\Omega}}+\left\|u|_{_{\Gamma}}\right\|^p_{_{p,\Gamma}}$$
is an equivalent norm for $W^{1,p}(\Omega)$.
\item[(c)]\,\,\,From (\ref{2.06}), it follows immediately that there exists a constant $c_{6}>0$ such that
\begin{equation}\label{2.07}
\|u\|_{_{W^{1,p}(\Omega)/\mathbb{R\!}}}\,\leq\,\|u\|_{_{W^{1,p}(\Omega)}}\,\leq\,c_{6}\|u\|_{_{W^{1,p}(\Omega)/\mathbb{R\!}}},
\end{equation}
for every $u\in W_0(\Omega)$.
\item[(d)]\,\,\,Let $p\in(1,N)$. Then, in views of \cite[Lemma 4.3]{GES-MITR09} (see also \cite[Lemma 2.4.7]{NITTKA2010}),
given $\epsilon>0$, there exists constants $C_{\epsilon},\,C'_{\epsilon}>$ such that
\begin{equation}\label{epsilon-interior}
\|u\|_{_{r,\Omega}}\,\leq\,\epsilon\|\nabla u\|_{_{p,\Omega}}+C_{\epsilon}\|u\|_{_{p,\Omega}}
\end{equation}
and
\begin{equation}\label{epsilon-trace}
\|u\|_{_{s,\Gamma}}\,\leq\,\epsilon\|\nabla u\|_{_{p,\Omega}}+C'_{\epsilon}\|u\|_{_{p,\Omega}},
\end{equation}
for each $u\in W^{1,p}(\Omega)$, whenever $1\leq r<Np(N-p)^{-1}$ and $1\leq s<dp(N-p)^{-1}$. Furthermore, by easily extending the results in \cite[Proposition 2.1]{Winkert-2014} and \cite[Theorem 2.3]{Winkert-Zacher-2016}, one sees that
$C_{\epsilon}=c\epsilon^{^{-\frac{Nr-Np}{rp-Nr+Np}}}$ and $C'_{\epsilon}=c'\epsilon^{^{-\frac{Ns-dp}{sp-Ns+dp}}}$ for some positive constants $c,\,c'$.
\end{enumerate}
\end{remark}

\subsection{Intermediate analytical results}\label{subsec02-03}

\indent The following result, which is of analytical nature, will be applied in some key subsequent results.

\begin{lemma}\label{lemma2}\,(see \cite{LAD-URAL68})\,
Let $\{y_n\}_{n\geq0},\,\{z_n\}_{n\geq0}$ be sequences of nonnegative real numbers, and let $c$,\, $b$,\, $\varepsilon$, and $\delta$ be positive constants
with $b\geq1$, such that
$$y_{n+1}\,\leq\,cb^n\left(y^{1+\delta}_n+z^{1+\varepsilon}_ny^{\delta}_n\right)\,\,\,\,\,\,\,\,\textrm{and}
\,\,\,\,\,\,\,\,z_{n+1}\,\leq\,cb^n\left(y_n+z^{1+\varepsilon}_n\right),\,\,\,\,\,\,\,\,\textrm{for all}\,\,n\in\mathbb{N\!}\,.$$
Define $$d:=\min\left\{\delta,\,\displaystyle\frac{\varepsilon}{1+\varepsilon}\right\}\,\,\,\,\,\,\,\,\textrm{and}
\,\,\,\,\,\,\,\,\eta:=\min\left\{(2c)^{^{-\frac{1}{\delta}}}b^{^{-\frac{1}{\delta d}}},\,
(2c)^{^{-\frac{1+\varepsilon}{\varepsilon}}}b^{^{-\frac{1}{\varepsilon d}}}\right\},$$ and assume that
$$y_0\,\leq\,\eta\,\,\,\,\,\,\,\,\textrm{and}\,\,\,\,\,\,\,\,z_0\,\leq\,\eta^{1/(1+\varepsilon)}.$$
Then $$y_n\,\leq\,\eta b^{^{-\frac{n}{b}}}\,\,\,\,\,\,\,\,\textrm{and}\,\,\,\,\,\,\,\,z_n\,\leq\,\left(\eta b^{^{-\frac{n}{b}}}\right)^{^{\frac{1}{1+\varepsilon}}},
\,\,\,\,\,\,\,\textrm{for every integer}\,\,n\geq0.$$
\end{lemma}

\subsection{Main assumption}

\indent Here we present the first and main assumptions that will be used throughout the rest of the paper (unless explicitly specified).\\

\begin{assumption}\label{As1}
Given $N\geq2,\,\,\mathfrak{s}\in(0,1),\,\,d\in (N-2,N)\cap(0,N)$, and $r_d:=2-N+d$,
let $\Omega\subseteq\mathbb{R\!}^N$ be a bounded extension domain,
and let $\mu$ be an upper $d$-Ahlfors measure supported on the boundary
$\Gamma:=\partial\Omega$.
\end{assumption}

\section{The elliptic problem}\label{sec03}

\indent The aim of this section is to establish the solvability and global regularity theory for the elliptic boundary value problem
(\ref{1.05}), under all the conditions in Assumption \ref{As1}.

\subsection{The variational setting}\label{subsec03-01}

\indent In this subsection we will state precisely the notion of weak solutions of the elliptic boundary
value problem (\ref{1.05}), and all the fundamental settings for our problem of interest.\\
\indent To begin, consider the bilinear closed form $(\mathcal{E}_{\mu},D(\mathcal{E}_{\mu}))$
formally defined by\\
\begin{equation}\label{3.1.01}
\mathcal{E}_{\mu}(u,v):=\mathcal{E}_{_{\Omega}}(u,v)+\mathcal{E}_{_{\Gamma}}(u,v)+K(u,v),
\indent\textrm{for all}\,\,\,u,\,v\in\mathcal{W}_2(\Omega;\Gamma),
\end{equation}
for
\begin{equation}\label{3.1.02}
\mathcal{W}_2(\Omega;\Gamma):=\left\{
\begin{array}{lcl}
H^1(\Omega),\indent\indent\,\,\,\textrm{if}\,\,\,\Lambda_{_{\Gamma}}\equiv0,\\[1ex]
H^1_{\mu}(\Omega;\Gamma),\indent\,\,\,\,\textrm{otherwise},\\[1ex]
\end{array}
\right.
\end{equation}
where
\begin{equation}\label{3.1.03}
\mathcal{E}_{_{\Omega}}(u,v):=\displaystyle\int_{\Omega}
\left(\displaystyle\sum^N_{^{i,j=1}}(\alpha_{ij}(x)\partial_{x_j}u+
\hat{a}_j(x)u)\partial_{x_i}v+\displaystyle\sum^N_{^{i=1}}\check{a}_i(x)\partial_{x_i}uv+\lambda(x)uv\right)\,dx
\end{equation}
and
\begin{equation}\label{3.1.04}
\mathcal{E}_{_{\Gamma}}(u,v):=\Lambda_{_{\Gamma}}(u,v)+\displaystyle\int_{\Gamma}\beta(x)uv\,d\mu,
\end{equation}
for $\Lambda_{_{\Gamma}}(\cdot,\cdot)$ the bilinear continuous weakly coercive form as in (\ref{1.03}), and
\begin{equation}\label{nonlocal-term}
K(u,v):=\left\langle\mathcal{J}_{_{\Omega}}u,v\right\rangle_{s/2}+\left\langle\Theta_{_{\Gamma}}u,v\right\rangle_{r_d/2}.
\end{equation}\indent\\
Here we recall that $\mathcal{J}_{_{\Omega}}:H^{s/2}(\Omega)\rightarrow H^{-s/2}(\Omega)$ and
$\Theta_{_{\Gamma}}:H^{r_d/2}_{\mu}(\Gamma)\rightarrow H^{-r_d/2}_{\mu}(\Gamma)$ are linear continuous
operators, with $\mathcal{J}_{_{\Omega}}1=\Theta_{_{\Gamma}}1=0$,\, $\langle\mathcal{J}_{_{\Omega}}u,1\rangle_{s/2}=0=
\langle\Theta_{_{\Gamma}}v,1\rangle_{r_d/2}$, \,$\langle\mathcal{J}_{_{\Omega}}u,u\rangle_{s/2}\geq0$,
and $\langle\Theta_{_{\Gamma}}v,v\rangle_{r_d/2}\geq0$, for each $u\in H^{s/2}(\Omega)$ and
$v\in H^{r_d/2}_{\mu}(\Gamma)$ (and thus for $u,\,v\in\mathcal{W}_2(\Omega;\Gamma)$). We also recall that
$\alpha_{ij}\in L^{\infty}(\Omega)$,\,$\hat{a}_i\in L^{r_1}(\Omega)$, \,$\check{a}_i \in L^{r_2}(\Omega)$ ($1\leq i,\,j\leq N$),\,$\lambda
\in L^{r_3}(\Omega)$, and $\beta\in L^{s}_{\mu}(\Gamma)$, where $r_1,\,r_2>N\chi_{_{\{N>2\}}}+\chi_{_{\{N=2\}}}$,\, $r_3>N/2$, and $s>d(d+2-N)^{-1}$.
Also, $\left(D(\Lambda_{_{\Gamma}}),\|\cdot\|_{_{D(\Lambda_{_{\Gamma}})}}\right)$ is assumed to be a Hilbert space,
and we will assume that $\Lambda_{_{\Gamma}}(\gamma_1,\gamma_2)=0$ for $\gamma_1,\,\gamma_2$ constants. In addition, set $\bar{\alpha}:=(\alpha_{11},\ldots,\alpha_{ij},\ldots,\alpha_{NN})$,\,
$\hat{\mathfrak{a}}:=(\hat{a}_1,\ldots,\hat{a}_N)$, and $\check{\mathfrak{a}}:=(\check{a}_1,\ldots,\check{a}_N)$.\\
\indent Now, taking into account the definition of the space $\mathcal{W}_2(\Omega;\Gamma)$ given by (\ref{3.1.02}),
we write the conditions on such definition, as follows:
\begin{enumerate}
\item[{\bf (C1)}]\,\,\,$\Lambda_{_{\Gamma}}\equiv0$;
\item[{\bf (C2)}]\,\,\,$\Lambda_{_{\Gamma}}\neq0$.
\end{enumerate}
Then it is clear that $\mathcal{W}_2(\Omega;\Gamma)$ is a Hilbert space, endowed with the norm
\begin{equation}\label{3.1.05}
\|\cdot\|_{_{\mathcal{W}_2(\Omega;\Gamma)}}:=\left\{\begin{array}{lcl}
\|\cdot\|_{_{H^1(\Omega)}},\indent\indent\indent\,\,\,\,\textrm{if {\bf{(C1)}} holds},\\[1ex]
\|\cdot\|_{_{H^1_{\mu}(\Omega;\Gamma)}},\indent\indent\indent\textrm{if {\bf{(C2)}} holds}.\\[1ex]
\end{array}
\right.
\end{equation}
Then we have the following simple result.

\begin{prop}\label{cont-coercive-form}
The form $(\mathcal{E}_{\mu},\mathcal{W}_2(\Omega;\Gamma))$ is continuous and weakly coercive. Consequently, if
$(f,g)\in\mathbb{X\!}^{\,p,q}(\Omega;\Gamma)$ for $p\geq(2^{\ast}_{_N})'$ and $q\geq(2^{\ast}_{_d})'$, where
\begin{equation}\label{linear-critical-exp}
2^{\ast}_{_N}:=\left\{\begin{array}{lcl}
\frac{2N}{N-2},\indent\indent\,\,\,\textrm{if}\,\,\,N>2,\\[1ex]
r\in[2,\infty),\,\,\,\,\,\,\,\,\textrm{if}\,\,\,N=2,\\[1ex]
\end{array}
\right.\,\,\,\textrm{and}\,\,\,\,\, 2^{\ast}_{_d}:=\left\{\begin{array}{lcl}
\frac{2d}{N-2},\indent\indent\,\,\,\textrm{if}\,\,\,N>2,\\[1ex]
r\in[2,\infty),\,\,\,\,\,\,\,\,\textrm{if}\,\,\,N=2,\\[1ex]
\end{array}
\right.
\end{equation}
then there exists an unique function $u\in\mathcal{W}_2(\Omega;\Gamma))$, and there exists a constant $\delta^{\star}_0\geq0$, such that
\begin{equation}\label{cc00}
\mathcal{E}_{\mu}(u,\varphi)+\delta^{\star}_0\displaystyle\int_{\Omega}u\varphi\,dx=\displaystyle\int_{\Omega}f\varphi\,dx+
\displaystyle\int_{\Gamma}g\varphi\,d\mu\,,\indent\,\,\textrm{for all}\,\,\varphi\in\mathcal{W}_2(\Omega;\Gamma).
\end{equation}
\end{prop}
\indent\\
\begin{proof}
We assume that condition {\bf (C2)} holds (the other cases are simpler), and take $u,\,v\in\mathcal{W}_2(\Omega;\Gamma))$.
Also for simplicity, we prove the result for $N>2$.
Then in views of (\ref{2.03}) together with H\"older's inequality (when necessary),
we get the following calculations over each of the integral terms of the form $\mathcal{E}_{\mu}(\cdot,\cdot)$:\\[2ex]
$\left|\displaystyle\int_{\Omega}\displaystyle\sum^N_{^{i=1}}\hat{a}_i(x)u\partial_{x_i}v\,dx\right|\,\leq\,\|\hat{\mathfrak{a}}\|_{_{N,\Omega}}\|u\nabla v\|_{_{\frac{N}{N-1},\Omega}}$\\
\begin{equation}\label{cc01}
\leq\,|\Omega|^{^{\frac{1}{N}-\frac{1}{r_1}}}\|\hat{\mathfrak{a}}\|_{_{r_1,\Omega}}\|u\|_{_{\frac{2N}{N-2},\Omega}}\|\nabla v\|_{_{2,\Omega}}
\,\leq\,C_1|\Omega|^{^{\frac{1}{N}-\frac{1}{r_1}}}\|\hat{\mathfrak{a}}\|_{_{r_1,\Omega}}\|u\|_{_{\mathcal{W}_2(\Omega;\Gamma)}}\|v\|_{_{\mathcal{W}_2(\Omega;\Gamma)}},
\end{equation}
for some constant $C_1>0$. Similarly, we have
\begin{equation}\label{cc02}
\left|\displaystyle\int_{\Omega}\displaystyle\sum^N_{^{i=1}}\check{a}_i(x)\partial_{x_i}uv\,dx\right|
\,\leq\,C_2|\Omega|^{^{\frac{1}{N}-\frac{1}{r_2}}}\|\check{\mathfrak{a}}\|_{_{r_2,\Omega}}\|u\|_{_{\mathcal{W}_2(\Omega;\Gamma)}}\|v\|_{_{\mathcal{W}_2(\Omega;\Gamma)}},
\end{equation}
for some constant $C_2>0$. In the same way, using (\ref{2.03}) and (\ref{2.04}), we get the existence of constants $C_3,\,C_4>0$, such that\\[2ex]
$\left|\displaystyle\int_{\Omega}\lambda uv\,dx\right|\,\leq\,|\Omega|^{^{1-\frac{1}{r_3}-\frac{N-2}{N}}}\|\lambda\|_{_{r_3,\Omega}}\|uv\|_{_{\frac{N}{N-2},\Omega}}$\\
\begin{equation}\label{cc03}
\leq\,|\Omega|^{^{1-\frac{1}{r_3}-\frac{N-2}{N}}}\|\lambda\|_{_{r_3,\Omega}}\|u\|_{_{\frac{2N}{N-2},\Omega}}\|v\|_{_{\frac{2N}{N-2},\Omega}}\,
\leq\,C_3|\Omega|^{^{1-\frac{1}{r_3}-\frac{N-2}{N}}}\|\lambda\|_{_{r_3,\Omega}}\|u\|_{_{\mathcal{W}_2(\Omega;\Gamma)}}\|v\|_{_{\mathcal{W}_2(\Omega;\Gamma)}},
\end{equation}
and\\[2ex]
$\left|\displaystyle\int_{\Gamma}\beta uv\,d\mu\right|\,\leq\,\mu(\Gamma)^{^{1-\frac{1}{s}-\frac{N-2}{d}}}\|\beta\|_{_{s,\Gamma}}\|uv\|_{_{\frac{d}{N-2},\Gamma}}$\\
\begin{equation}\label{cc04}
\leq\,\mu(\Gamma)^{^{1-\frac{1}{s}-\frac{N-2}{d}}}\|\beta\|_{_{s,\Gamma}}\|u\|_{_{\frac{2d}{N-2},\Omega}}\|v\|_{_{\frac{2d}{N-2},\Omega}}\,
\leq\,C_4\mu(\Gamma)^{^{1-\frac{1}{s}-\frac{N-2}{d}}}\|\beta\|_{_{s,\Gamma}}\|u\|_{_{\mathcal{W}_2(\Omega;\Gamma)}}\|v\|_{_{\mathcal{W}_2(\Omega;\Gamma)}}.
\end{equation}
Now, applying (\ref{2.04b} and (\ref{2.05}), we see that
\begin{equation}\label{cc05}
\left|\left\langle\mathcal{J}_{_{\Omega}}u,v\right\rangle_{s/2}\right|\,\leq\,c\|u\|_{_{H^{s/2}(\Omega)}}\|v\|_{_{H^{s/2}(\Omega)}}\,\leq\,
C_5\|u\|_{_{\mathcal{W}_2(\Omega;\Gamma)}}\|v\|_{_{\mathcal{W}_2(\Omega;\Gamma)}},
\end{equation}
and
\begin{equation}\label{cc06}
\left|\left\langle\Theta_{_{\Gamma}}u,v\right\rangle_{r_d/2}\right|\,\leq\,c'\|u\|_{_{H^{r_d/2}_{\mu}(\Gamma)}}\|v\|_{_{H^{r_d/2}_{\mu}(\Gamma)}}\,\leq\,
C_6\|u\|_{_{\mathcal{W}_2(\Omega;\Gamma)}}\|v\|_{_{\mathcal{W}_2(\Omega;\Gamma)}},
\end{equation}
for some constants $c,\,c',\,C_5,\,C_6>0$. Furthermore, since the ``Wentzell-type" form $\Lambda_{_{\Gamma}}(\cdot,\cdot)$ is assumed to be continuous,
this means that there exists constants $c'',\,C_7>0$, such that
\begin{equation}\label{cc07}
\left|\Lambda_{_{\Gamma}}(u,v)\right|\,\leq\,c''\|u\|_{_{D(\Lambda_{_{\Gamma}})}}\|v\|_{_{D(\Lambda_{_{\Gamma}})}}\,\leq\,
C_7\|u\|_{_{\mathcal{W}_2(\Omega;\Gamma)}}\|v\|_{_{\mathcal{W}_2(\Omega;\Gamma)}}.
\end{equation}
Obviously
\begin{equation}\label{cc08}
\left|\displaystyle\int_{\Omega}\displaystyle\sum^N_{^{i,j=1}}\alpha_{ij}(x)\partial_{x_j}u\partial_{x_i}v\,dx\right|\,\leq\,C_8\|\bar{\alpha}\|_{_{\infty,\Omega}}
\|u\|_{_{\mathcal{W}_2(\Omega;\Gamma)}}\|v\|_{_{\mathcal{W}_2(\Omega;\Gamma)}},
\end{equation}
and therefore (\ref{cc01}), (\ref{cc02}), (\ref{cc03}), (\ref{cc04}), (\ref{cc05}), (\ref{cc06}), (\ref{cc07}),(\ref{cc08}), show that
the form $(\mathcal{E}_{\mu},D(\mathcal{E}_{\mu}))$ is continuous (possibly nonlocal, depending on the selection of the operators $\mathcal{J}_{_{\Omega}}$ and
$\Theta_{_{\Gamma}}$). To show the weak coercivity of $\mathcal{E}_{\mu}(\cdot,\cdot)$, we first recall that $\langle\mathcal{J}_{_{\Omega}}u,u\rangle_{s/2}\geq0$,\,
$\langle\Theta_{_{\Gamma}}v,v\rangle_{r_d/2}\geq0$. Also we recall that $\Lambda_{_{\Gamma}}(\cdot,\cdot)$ is weakly coercive, which
means that there exist constants $c^{\star}_0>0$ and $\gamma_0\in\mathbb{R\!\,}$ such that
\begin{equation}\label{Wentzell-weak-coercive}
c^{\star}_0\|u\|^2_{_{D(\Lambda_{_{\Gamma}})}}\,\leq\,\Lambda_{_{\Gamma}}(u,u)+\gamma_0\|u\|^2_{_{2,\Gamma}}.
\end{equation}
Consequently, taking into account (\ref{1.04}), we get that\\[2ex]
\indent$\min\{c_0,c^{\star}_0\}\|u\|^2_{_{\mathcal{W}_2(\Omega;\Gamma)}}\,\leq\,\mathcal{E}_{\mu}(u,u)+\gamma_0\displaystyle\int_{\Gamma}u^2\,d\mu+
\left|\displaystyle\int_{\Omega}\displaystyle\sum^N_{^{i=1}}\hat{a}_i(x)u\partial_{x_i}u\,dx\right|+$
\begin{equation}\label{cc09}
\indent\indent\indent+\left|\displaystyle\int_{\Omega}\displaystyle\sum^N_{^{i=1}}\check{a}_i(x)\partial_{x_i}u\,u\,dx\right|+
\left|\displaystyle\int_{\Omega}(\lambda+c_0) u^2\,dx\right|+\left|\displaystyle\int_{\Gamma}\beta u^2\,d\mu\right|.
\end{equation}
Now, since $\min\{r_1,r_2\}>N$, we have that the embedding maps $\mathcal{W}_2(\Omega;\Gamma)\hookrightarrow L^{^{\frac{2r_1}{r_1-2}}}(\Omega)$ and
$\mathcal{W}_2(\Omega;\Gamma)\hookrightarrow L^{^{\frac{2r_2}{r_2-2}}}(\Omega)$ are compact, so using this together with
Remark \ref{R-embedding}(d), we apply H\"older's inequality and Young's inequality to produce the following calculations:\\[2ex]
$\left|\displaystyle\int_{\Omega}\displaystyle\sum^N_{^{i=1}}\hat{a}_i(x)u\partial_{x_i}u\,dx\right|\,\leq\,
\|\hat{\mathfrak{a}}\|_{_{r_1,\Omega}}\|u\|_{_{\frac{2r_1}{r_1-2},\Omega}}\|\nabla u\|_{_{2,\Omega}}$\\
$$\leq\,\|\hat{\mathfrak{a}}\|_{_{r_1,\Omega}}\left(\epsilon_1\|\nabla u\|^2_{_{2,\Omega}}+
\epsilon^{-1}_1\|u\|^2_{_{\frac{2r_1}{r_1-2},\Omega}}\right)\,\leq\,\|\hat{\mathfrak{a}}\|_{_{r_1,\Omega}}\left(2\epsilon_1\|\nabla u\|^2_{_{2,\Omega}}+
\frac{C_{\epsilon^2_1}}{\epsilon_1}\|u\|^2_{_{2,\Omega}}\right)$$
\begin{equation}\label{cc10}
\indent\indent\indent\indent\,\,\,\,\,\,\indent\leq\,\|\hat{\mathfrak{a}}\|_{_{r_1,\Omega}}\left(2\epsilon_1
\|u\|^2_{_{\mathcal{W}_2(\Omega;\Gamma)}}+\frac{C_{\epsilon^2_1}}{\epsilon_1}\|u\|^2_{_{2,\Omega}}\right),
\end{equation}
for all $\epsilon_1>0$, and for some constant $C_{\epsilon^2_1}>0$. In the exact way, we get that
\begin{equation}\label{cc11}
\left|\displaystyle\int_{\Omega}\displaystyle\sum^N_{^{i=1}}\check{a}_i(x)\partial_{x_i}u\,u\,dx\right|
\leq\,\|\check{\mathfrak{a}}\|_{_{r_2,\Omega}}\left(2\epsilon_2
\|u\|^2_{_{\mathcal{W}_2(\Omega;\Gamma)}}+\frac{C_{\epsilon^2_2}}{\epsilon_2}\|u\|^2_{_{2,\Omega}}\right),
\end{equation}
for every $\epsilon_2>0$, and for some constant $C_{\epsilon^2_2}>0$. Now, if $r_3>N/2$ and $s>d(d+2-N)^{-1}$, it follows that the maps
$\mathcal{W}_2(\Omega;\Gamma)\hookrightarrow L^{^{\frac{2r_3}{r_3-1}}}(\Omega)$ and $\mathcal{W}_2(\Omega;\Gamma)\hookrightarrow L^{^{\frac{2s}{s-1}}}_{\mu}(\Gamma)$
are both compact. Consequently, continuing as in the above procedures, we find that
\begin{equation}\label{cc12}
\left|\displaystyle\int_{\Omega}\lambda u^2\,dx\right|\,\leq\,\|\lambda\|_{_{r_3,\Omega}}\|u\|^2_{_{\frac{2r_3}{r_3-1},\Omega}}\,
\leq\,\|\lambda\|_{_{r_3,\Omega}}\left(\epsilon_3\|\nabla u\|^2_{_{2,\Omega}}+C_{\epsilon_3}\|u\|^2_{_{2,\Omega}}\right)\,
\leq\,\|\lambda\|_{_{r_3,\Omega}}\left(\epsilon_3\|u\|^2_{_{\mathcal{W}_2(\Omega;\Gamma)}}+C_{\epsilon_3}\|u\|^2_{_{2,\Omega}}\right),
\end{equation}
and
\begin{equation}\label{cc13}
\left|\displaystyle\int_{\Gamma}\beta u^2\,d\mu\right|\,\leq\,\|\beta\|_{_{s,\Gamma}}\|u\|^2_{_{\frac{2s}{s-1},\Gamma}}\,
\leq\,\|\beta\|_{_{s,\Gamma}}\left(\epsilon_4\|\nabla u\|^2_{_{2,\Omega}}+C_{\epsilon_4}\|u\|^2_{_{2,\Omega}}\right)\,
\leq\,\|\beta\|_{_{s,\Gamma}}\left(\epsilon_4\|u\|^2_{_{\mathcal{W}_2(\Omega;\Gamma)}}+C_{\epsilon_4}\|u\|^2_{_{2,\Omega}}\right),
\end{equation}
and
\begin{equation}\label{cc14}
\gamma_0\displaystyle\int_{\Gamma}u^2\,d\mu\,\leq\,|\gamma_0|\left(\epsilon_5\|u\|^2_{_{\mathcal{W}_2(\Omega;\Gamma)}}+C_{\epsilon_5}\|u\|^2_{_{2,\Omega}}\right),
\end{equation}
for all $\epsilon_3,\,\epsilon_4,\epsilon_5>0$, and for some positive constants $C_{\epsilon_3},\,C_{\epsilon_4},\,C_{\epsilon_5}$. Then, if $\gamma_0\neq0$, we select
$$\epsilon_1:=\frac{\min\{c_0,c^{\star}_0\}}{20\|\hat{\mathfrak{a}}\|_{_{r_1,\Omega}}},\,\,\,\,\,\,\,\,\,\,\,\,\,\,\,\,\,
\epsilon_2:=\frac{\min\{c_0,c^{\star}_0\}}{20\|\check{\mathfrak{a}}\|_{_{r_2,\Omega}}},$$
$$\epsilon_3:=\frac{\min\{c_0,c^{\star}_0\}}{10\|\lambda\|_{_{r_3,\Omega}}},\,\,\,\,\,\,\,\,\epsilon_4:=\frac{\min\{c_0,c^{\star}_0\}}{10\|\beta\|_{_{s,\Gamma}}},
\,\,\,\,\,\,\,\,\epsilon_5:=\frac{\min\{c_0,c^{\star}_0\}}{10|\gamma_0|},$$
and put
\begin{equation}\label{delta-star}
\delta^{\star}_0:=c_0+\frac{20\left(C_{\epsilon^2_1}\|\hat{\mathfrak{a}}\|^2_{_{r_1,\Omega}}+
C_{\epsilon^2_2}\|\check{\mathfrak{a}}\|^2_{_{r_2,\Omega}}\right)}{\min\{c_0,c^{\star}_0\}}+C_{\epsilon_3}\|\lambda\|_{_{r_3,\Omega}}+
C_{\epsilon_4}\|\beta\|_{_{s,\Gamma}}+C_{\epsilon_5}|\gamma_0|.
\end{equation}
Then, by combining the inequalities
(\ref{cc09}), (\ref{cc10}), (\ref{cc11}), (\ref{cc12}), (\ref{cc13}), (\ref{cc14}) together with the selections of $\epsilon_i>0$ ($1\leq i\leq5$),
we arrive at the key inequality\\
\begin{equation}
\label{3.1.06}
\frac{\min\{c_0,c^{\star}_0\}}{2}\|u\|^2_{_{\mathcal{W}_2(\Omega;\Gamma)}}\,\leq\,\mathcal{E}_{\mu}(u,u)+\delta^{\star}_0\|u\|^2_{_{2,\Omega}},
\,\,\,\,\,\,\,\textrm{for all}\,\,u\in\mathcal{W}_2(\Omega;\Gamma).
\end{equation}
\indent\\
Inequality (\ref{3.1.06}) shows that the form $\mathcal{E}_{\mu}(\cdot,\cdot)$ is weakly coercive, as asserted.
To complete the proof, let $(f,g)\in\mathbb{X\!}^{\,p,q}(\Omega;\Gamma)$, for $p\geq(2^{\ast}_{_N})'$ and $q\geq(2^{\ast}_{_d})'$, where
$2^{\ast}_{_N},\,2^{\ast}_{_d}$ are given as in (\ref{linear-critical-exp}) (and $r'$ denotes the conjugate of $r\in(1,\infty)$).
Then, using (\ref{2.03}) and (\ref{2.04}),
it follows that the functional $T_{\mu}:\mathcal{W}_2(\Omega;\Gamma)
\rightarrow\mathbb{R\!}$\,, given by $$T_{\mu}v:=\displaystyle\int_{\Omega}fv\,dx+\displaystyle\int_{\Gamma}gv\,d\mu$$
is linear and continuous. Therefore, an application of Lax-Milgram theorem theorem gives the existence of an unique function
$u\in\mathcal{W}_2(\Omega;\Gamma)$ fulfilling (\ref{cc00}), as desired.
\end{proof}

In views of Proposition \ref{cont-coercive-form}, the following definition makes sense.\\

\begin{definition}\label{weak-sol}
A function $u\in\mathcal{W}_2(\Omega;\Gamma)$ is said to be a \textbf{weak solution of (\ref{1.05})}, if
\begin{equation}
\label{3.1.07}\mathcal{E}_{\mu}(u,\varphi)=\displaystyle\int_{\Omega}f\varphi\,dx+
\displaystyle\int_{\Gamma}g\varphi\,d\mu\,,\indent\,\,\textrm{for all}\,\,\varphi\in\mathcal{W}_2(\Omega;\Gamma).
\end{equation}
Also, $u\in\mathcal{W}_2(\Omega;\Gamma)$ is called a \textbf{weak subsolution of (\ref{1.05})}, if
\begin{equation}
\label{3.1.08}\mathcal{E}_{\mu}(u,\varphi)\,\leq\,\displaystyle\int_{\Omega}f\varphi\,dx+
\displaystyle\int_{\Gamma}g\varphi\,d\mu\,\,\,\textrm{for all}\,\,\varphi\in\mathcal{W}_2(\Omega;\Gamma)^+
:=\{w\in\mathcal{W}_2(\Omega;\Gamma)\mid w\geq0\},
\end{equation}
and if the reverse inequality of (\ref{3.1.08}) holds, then
$u\in\mathcal{W}_2(\Omega;\Gamma)$ is called a \textbf{weak supersolution of (\ref{1.05})}.\\
\end{definition}

\subsection{A priori estimates for weak solutions}\label{subsec03-02}

\indent In this part we establish some a priori estimates for weak solutions to problem (\ref{1.05}), as well as other
useful norm estimates for such solutions. For simplicity, all the results will be established for the case $N>2$.\\
\indent To begin, define the functions\, $\psi_{_{t,m}}:\mathbb{R\!}\rightarrow\mathbb{R\!}$\, ,
for each $t\geq 1,\, m\geq 1$,\, by:
\begin{equation}
\label{3.2.01}\psi_{_{t,m}}(x):=\left\{
\begin{array}{lcl}
\,\,\,\,\,0\, ,\,\,\,\,\,\,\,\,\,\,\, x\leq 0\\
\,\,\,\,x^t\, ,\,\,\,\, 0<x<m\\
m^{t-1}x\, ,\,\,\,x\geq m\\
\end{array}
\right.
\end{equation}
Clearly \,$\psi_{_{t,m}}\in C(\mathbb{R\!}\,)$, and is piecewise smooth, with bounded derivative. Thus, given $u\in\mathcal{W}_2(\Omega;\Gamma)$,
using the assumption {\bf (B1)} below, one sees that $\psi_{_{t,m}}\circ\,u\in H^1(\Omega)$. Moreover, as the substitution operator
induced by $\psi_{_{t,m}}$ is continuous on $H^1(\Omega)$, so letting
\begin{equation}\label{3.2.02}
w_{_m}:=\psi_{_{\frac{k}{2},m}}(u),\indent\textrm{and}\indent
v_{_m}:=\psi_{_{k-1,m}}(u),\indent\textrm{for all}\,\,\,k\in[2,\infty)\,\,\,\textrm{and}\,\,\,m\geq 1,
\end{equation}
one has $w_{_m},\,v_{_m}\in H^1(\Omega)^+$.
Also, the following calculations are valid:
\begin{enumerate}
\item[$\bullet$]\,\,\,$\partial_{x_j}w_{_m}\partial_{x_i}w_{_m}=\frac{k^2}{4(k-1)}\partial_{x_j}u\partial_{x_i}v_{_m},\,\,
w_{_m}\partial_{x_i}w_{_m}=\frac{k}{2}v_{_m}\partial_{x_i}u=\frac{k}{2(k-1)}u\partial_{x_i}v_{_m}$ \,and\, $w^2_{_m}=uv_{_m}$,\, if \,$0\leq u\leq m$.\\
\item[$\bullet$]\,\,\,$\partial_{x_j}w_{_m}\partial_{x_i}w_{_m}=\partial_{x_j}u\partial_{x_i}v_{_m},\,\,
w_{_m}\partial_{x_i}w_{_m}=v_{_m}\partial_{x_i}u=u\partial_{x_i}v_{_m},\,\,
w^2_{_m}=uv_{_m}$,\,\,\,\, if \,$u\geq m$,\, for $1\leq i,\,j\leq N$.
\end{enumerate}
Also, for $D$ denoting either $\Omega$, or $\Gamma$, given $k_0\geq0$ a real number, and $k\geq k_0$, we define the following:
\begin{equation}\label{u_k def}
u_k:=(u-k)^+\,\,\,\,\,\,\,\,\textrm{and}\,\,\,\,\,\,\,\,D_k:=\{x\in D\mid u(x)>k\}.
\end{equation}
Clearly $u_k\in H^1(\Omega)$ with $\partial_{x_i}u_k=\partial_{x_i}u\chi_{_{\Omega_k}}$. We also define
an equivalent norm for $\mathcal{W}_2(\Omega;\Gamma)$ by letting
$\|\cdot\|_{_{\mathcal{W}_2(\overline{\Omega})}}$ be the same norm as $\|\cdot\|_{_{\mathcal{W}_2(\Omega;\Gamma)}}$, without
the $L^2$-norm at the boundary (when applicable).\\
\indent Next, under the above notations and assumptions, in order to achieve the validity of the next two results, we need to
assume the following technical conditions on the boundary form $\mathcal{E}_{_{\Gamma}}(\cdot,\cdot)$, as well
as the nonlocal-type maps. Indeed, for each
$u\in\mathcal{W}_2(\Omega;\Gamma)$, we assume the fulfillment of the following conditions:
\begin{enumerate}
\item[{\bf (A1)}]\,\,\,$K(u,v_{_m})+\eta_0\,k^2|\|\mathbf{w}_{_m}\||^2_{2}\geq0$ for some constant $\eta_0\in\mathbb{R\!}\,$, where
$\mathbf{w}_{_m}:=(w_{_m},w_{_m}|_{_{\Gamma}})$;
\item[{\bf (A2)}]\,\,\,$K(u^+,u^-)\leq0$;
\item[{\bf (A3)}]\,\,\,$K(u,u_k)\geq0$ for each $u\in H^1(\Omega)$;
\item[{\bf (B1)}]\,\,\,$D(\Lambda_{_{\Gamma}})$ is a Banach lattice, and $\psi_{_{t,m}}\circ\,u\in D(\Lambda_{_{\Gamma}})$
if $u\in D(\Lambda_{_{\Gamma}})$;
\item[{\bf (B2)}]\,\,\,$\|w_{_m}\|^2_{_{D(\Lambda_{_{\Gamma}})}}\leq c_0\,k^2\left(\Lambda_{_{\Gamma}}(u,v_{_m})+\gamma_0\|w_{_m}\|^2_{_{2,\Gamma}}\right)$
for some constants $c_0>0$ and $\gamma_0\in\mathbb{R\!}\,$;
\item[{\bf (B3)}]\,\,\,$\Lambda_{_{\Gamma}}(u^+,u^-)\leq0$.
\item[{\bf (B4)}]\,\,\,For each $u\in D(\Lambda_{_{\Gamma}})$ one has $u_k\in D(\Lambda_{_{\Gamma}})$ and
$$\Lambda_{_{\Gamma}}(u,u_k)\,\geq\,c^{\star}_1\Lambda_{_{\Gamma}}(u_k,u_k)+
c^{\star}_2\|u_k\|^2_{_{D(\Lambda_{_{\Gamma}})}}-c^{\star}_3\displaystyle\int_{\Gamma_k}\rho(x)(|u_k|^2+k^2)\,d\mu,$$
for some constants $c^{\star}_1,\,c^{\star}_2,\,c^{\star}_3\geq0$, and for some function $\rho\in L^{\theta}_{\mu}(\Gamma)$,
where $\theta\geq1$ is such that the map
$\mathcal{W}_2(\Omega;\Gamma)\hookrightarrow L^{^{\frac{2\theta}{\theta-1}}}_{\mu}(\Gamma)$ is compact.
\end{enumerate}
Here we recall that $K(\cdot,\cdot)$ is defined as in (\ref{nonlocal-term}). Then, under these assumptions, we
clearly see that $w_{_m},\,v_{_m}\in\mathcal{W}_2(\Omega;\Gamma)^+$. In views of {\bf (B4)}, for each $u\in\mathcal{W}_2(\Omega;\Gamma)$, we have $u_k\in\mathcal{W}_2(\Omega;\Gamma)$.
Although these conditions seem to be non-standard and quite technical,
they turn out to be necessary for the achievement of the next results. Examples of bilinear Wentzell-type
forms $\left(\Lambda_{_{\Gamma}},D(\Lambda_{_{\Gamma}})\right)$ fulfilling {\bf (B1)}, {\bf (B2)}, {\bf (B3)}, and {\bf (B4)},
as well as of linear bounded (nonlocal) boundary operators satisfying
{\bf (A1)}, {\bf (A2)}, and  {\bf (A3)}, will be given in the sections \ref{sec06} and \ref{sec07}, respectively.\\
\indent Before establishing the main result of this section, we first prove the following lemma, whose parabolic version is given in section \ref{sec04}. Since we will be giving a
complete proof for the parabolic version (e.g. proof of Lemma \ref{Lem1}), in the result below we will only spell out the main steps of its proof.

\begin{lemma}\label{Lem0}
Let $u\in\mathcal{W}_2(\Omega;\Gamma)$, and assume that $u_k|_{_{\Gamma}}\in D(\Lambda_{_{\Gamma}})$. Given $m_1,\,m_2\geq1$ integers and
$\xi\in\{1,\ldots,m_1\}$, and $\zeta\in\{1,\ldots,m_2\}$, let $s_{1,\xi}\in[2,2N(N-2)^{-1}]$ and $s_{2,\zeta}\in[2,2d(N-2)^{-1}]$.
Assume that there exists constants $\hat{k}_0,\,\gamma_0\geq0$, and positive constants $\theta_{1,\xi},\,\theta_{2,\zeta}$, such that
for all  $k\geq\hat{k}_0$, we have
\begin{equation}\label{Lem0-02}
\|u_k\|^2_{_{\mathcal{W}_2(\overline{\Omega})}}\,\leq\,\gamma_0\left(\|u_k\|^2_{_{2,\Omega_k}}
+k^2\displaystyle\sum^{m_1}_{\xi=1}|\Omega_k|
^{^{\frac{2(1+\theta_{1,\xi})}{s_{1,\xi}}}}+k^2\displaystyle\sum^{m_2}_{\zeta=1}\mu(\Gamma_k)^{^{\frac{2(1+\theta_{2,\zeta})}{s_{2,\zeta}}}}\right).
\end{equation}
Then there exists a constant $C^{\ast}_0>0$ (independent of $u$ and $\hat{k}_0$) such that
\begin{equation}\label{Lem0-03}
\textrm{ess}\displaystyle\sup_{\overline{\Omega}}[u]
\,\leq\,C^{\ast}_0\left(\|u\|^2_{_{2,\Omega}}+\hat{k}^2_0\right)^{1/2}.
\end{equation}
\end{lemma}

\begin{proof}
Given $\hat{k}\geq\hat{k}_0$ fixed, without loss of generality, we prove the result for
$$\theta_{1,\xi}=\theta_{2,\zeta}=\theta:=\min\left\{\min\{\theta_{1,\xi}\mid1\leq\xi\leq m_1\}\cup\min\{\theta_{2,\zeta}\mid1\leq\zeta\leq m_2\}\right\}$$
for each $\xi\in\{1,\ldots,m_1\}$ and $\zeta\in\{1,\ldots,m_2\}$. Then,
define the sequences $\{k_n\}_{n\geq0}$,\, $\{y_n\}_{n\geq0}$,\, $\{z_n\}_{n\geq0}$ of positive real numbers, by:
\begin{equation}\label{Lem0-06}
k_n:=(2-2^{-n})\hat{k},\,\,\,\,\,\,\,\,\,\,\,\,\,y_n:=\hat{k}^{-2}\|u_{k_n}\|^2_{_{2,\Omega_k}},
\end{equation}
and
\begin{equation}\label{Lem0-07}
z_n:=\displaystyle\sum^{m_1}_{\xi=1}|\Omega_{k_n}|^{^{\frac{2}{s_{1,\xi}}}}
+\displaystyle\sum^{m_2}_{\zeta=1}\mu(\Gamma_{k_n})^{^{\frac{2}{s_{2,\zeta}}}}.
\end{equation}
Clearly $k_n\geq\hat{k}\geq\hat{k}_0$ for each nonnegative integer $n$. Taking into account
the definitions of the sequences defined above together with H\"older's inequality and (\ref{2.03}), we calculate and get that\\[2ex]
$\hat{k}^2y_{n+1}\,\leq\,|\Omega_{k_{n+1}}|^{2/N}\left\|u_{k_{n+1}}\right\|^2_{_{2^{\ast}_{_N},\Omega}}$\\
\begin{equation}\label{Lem0-08}
\indent\indent\leq\,
c\left\{(k_{n+1}-k_n)^{-2}\hat{k}^2y_n\right\}^{2/N}\|u_{k_{n+1}}\|^2_{_{\mathcal{W}_2(\overline{\Omega})}}
\,\leq\,c\,4^{n+1}y_n^{2/N}\|u_{k_{n+1}}\|^2_{_{\mathcal{W}_2(\overline{\Omega})}},\,\,\,\,\,\,\,\,\,\,\,\,
\end{equation}
where $c>0$ is a constant that varies from line to line (this will be assumed for the remaining of the proof). Similarly, one has
\begin{equation}\label{Lem0-09}
4^{-(n+1)}\hat{k}^2z_{n+1}\,\leq\,(k_{n+1}-k_n)^2z_{n+1}\,\leq\,c\|u_{k_{n}}\|^2_{_{\mathcal{W}_2(\overline{\Omega})}}.
\end{equation}
Applying assumption (\ref{Lem1-02}) yields that
\begin{equation}\label{Lem0-10}
\|u_{k_{n+1}}\|^2_{_{\mathcal{W}_2(\overline{\Omega})}}\,\leq\,\gamma_02^{n+4}\hat{k}^2(y_n+z^{1+\theta}_n),
\end{equation}
and then combining (\ref{Lem0-08}), (\ref{Lem0-09}), and (\ref{Lem0-10}), we arrive at
\begin{equation}\label{Lem0-11}
y_{n+1}\,\leq\,c'\,8^n(y^{1+\delta}_n+z^{1+\theta}_ny^{\delta}_n)\,\,\textrm{and}\,\,
z_{n+1}\,\leq\,c'\,8^n(y_n+z^{1+\theta}_n),\,\,\textrm{for each integer}\,\,n\geq0,
\end{equation}
for some constant $c'>0$, where $\delta:=2/N$. Now, it is clear that
\begin{equation}\label{Lem0-12}
y_0\,\leq\,\hat{k}^{-2}\|u\|^2_{_{2,\Omega_k}}.
\end{equation}
On the other hand, using (\ref{Lem0-02}) and proceeding as above, we find that
$$(\hat{k}-\hat{k}_0)^2z_0\,\leq\,c\gamma_0\left(\|u_{k_0}\|^2_{_{2,\Omega_k}}+k^2_0\displaystyle\sum^{m_1}_{\xi=1}|\Omega|^{^{\frac{2(1+\theta)}{s_{1,\xi}}}}
+k^2_0\displaystyle\sum^{m_2}_{\zeta=1}\mu(\Gamma)^{^{\frac{2(1+\theta)}{s_{2,\zeta}}}}\right).$$
Thus setting
$$c'':=c\gamma_0\max\left\{1,\,\displaystyle\sum^{m_1}_{\xi=1}|\Omega|^{^{\frac{2(1+\theta)}{s_{1,\xi}}}}
+\displaystyle\sum^{m_2}_{\zeta=1}\mu(\Gamma)^{^{\frac{2(1+\theta)}{s_{2,\zeta}}}}\right\},$$
we obtain that
\begin{equation}\label{Lem0-13}
z_0\,\leq\,\displaystyle\frac{c''}{(\hat{k}-\hat{k}_0)^2}\left(\|u_{k_0}\|^2_{_{2,\Omega_{k_0}}}+k^2_0\right),
\,\,\,\,\,\,\,\,\,\,\,\,\textrm{for all}\,\,\,\hat{k}\geq k_0.
\end{equation}
Selecting $$d:=\min\left\{\delta,\,\frac{\theta}{1+\theta}\right\}\,\,\,\,\,\textrm{and}\,\,\,\,\,
\eta:=\min\left\{\frac{1}{(2c')^{^{\frac{1}{\delta}}}8^{^{\frac{1}{\delta d}}}},\,\frac{1}{(2c')^{^{\frac{1+\theta}{\theta}}}\,8^{^{\frac{1}{\theta d}}}}\right\},$$
and choosing
\begin{equation}\label{Lem0-14}
\hat{k}:=\max\left\{\displaystyle\frac{1}{\sqrt{\eta}}\|u\|_{_{2,\Omega}},\,
\displaystyle\frac{\sqrt{c''}}{\eta^{^{\frac{1}{2(1+\theta)}}}}
\left(\|u\|^2_{_{2,\Omega}}+\hat{k}^2_0\right)^{1/2}+\hat{k}_0\right\},
\end{equation}
we see that
\begin{equation}\label{Lem0-15}
\hat{k}\,\leq\,c^{\star}\left(\|u\|^2_{_{2,\Omega}}+\hat{k}^2_0\right)^{1/2}
\end{equation}
for some constant $c^{\star}>0$, and moreover the selection implies that
\begin{equation}\label{Lem0-16}
y_0\,\leq\,\eta\,\,\,\,\,\,\,\,\textrm{and}\,\,\,\,\,\,\,\,z_0\,\leq\,\eta^{^{\frac{1}{1+\theta}}}.
\end{equation}
Therefore, the sequences $\{y_n\}$,\, $\{z_n\}$ satisfy the conditions of Lemma \ref{lemma2}.
Thus, applying Lemma \ref{lemma2} for $\hat{k}$ given by (\ref{Lem0-14})
shows that $\displaystyle\lim_{n\rightarrow\infty}z_n=0$, and consequently
\begin{equation}\label{Lem0-17}
u(t)\,\leq\,\displaystyle\lim_{n\rightarrow\infty}k_n=2\hat{k}\,\,\,\,\,\,\,\,\textrm{a.e. in}\,\,\overline{\Omega}.
\end{equation}
Combining (\ref{Lem0-17}) with (\ref{Lem0-15}), we obtain (\ref{Lem0-03}), as desired.
\end{proof}

\indent Now we are ready to state and prove the first result of this subsection. For simplicity, we will consider the
hardest case, namely, when {\bf (C2)} is valid.\\

\begin{theorem}\label{elliptic-bounded}
Assume all the conditions in Assumption \ref{As1}, let $(f,g)\in\mathbb{X\!}^{\,p,q}(\Omega;\Gamma)$, where $p\geq(2^{\ast}_{_N})'$ and $q\geq(2^{\ast}_{_d})'$, and assume that {\bf (A1)}, {\bf (A3)},
{\bf (B1)}, {\bf (B2)}, and {\bf (B4)} hold. If $u\in\mathcal{W}_2(\Omega;\Gamma)$ is a weak solution of (\ref{1.05}), then there exists a constant
$c^{\ast}=c(p,q,N,d,|\Omega|,\mu(\Gamma))>0$ such that
\begin{equation}\label{3.2.06}
|\|\mathbf{u}\||_{_{m_1(p,q),m_2(q,q)}}\,\leq\,c^{\ast}\left(|\|(f,g)\||_{_{p,q}}+\|u\|_{_{2,\Omega}}\right),
\end{equation}
where
\begin{equation}\label{m1}
m_1(p,q):=\left\{\begin{array}{lcl}
\frac{N}{N-2}\min\left\{\frac{p(N-2)}{N-2p},\,\frac{q(N-2)}{Nq-2q-dq+d}\right\},\,\,\textrm{if}\,\,p<N/2\,\,\textrm{and}\,\,q<\frac{d}{d+2-N},\\[1ex]
\frac{Np}{Np-2p-dp+d},\,\,\textrm{if}\,\,p<N/2<\frac{d}{d+2-N}\geq q,\,\,
\textrm{or}\,\,q<\frac{d}{d+2-N}\,\,\textrm{and}\,\,N/p\leq p<q;\\[1ex]
\frac{Nq}{Nq-2q-dq+d},\,\,\textrm{if}\,\,q<\frac{d}{d+2-N}\,\,\textrm{and}\,\,p\geq\max\{N/2,q\},\\[1ex]
r\in(2^{\ast}_{_N},\infty),\,\,\textrm{if}\,\,p=N/2\,\,\textrm{and}\,\,q\geq\frac{d}{d+2-N},\,\,\textrm{or}
\,\,q=\frac{d}{d+2-N}\,\,\textrm{and}\,\,p\geq N/2,\\[1ex]
\,\,\,\,\,\infty,\,\,\textrm{if}\,\,p>N/2\,\,\,\textrm{and}\,\,\,q>\frac{d}{d+2-N},
\end{array}
\right.
\end{equation}
and
\begin{equation}\label{m2}
m_2(p,q):=\left\{\begin{array}{lcl}
\frac{d}{N-2}\min\left\{\frac{p(N-2)}{N-2p},\,\frac{q(N-2)}{Nq-2q-dq+d}\right\},\,\,\textrm{if}\,\,p<N/2\,\,\textrm{and}\,\,q<\frac{d}{d+2-N},\\[1ex]
\frac{dp}{Np-2p-dp+d},\,\,\textrm{if}\,\,p<N/2<\frac{d}{d+2-N}\geq q,\,\,
\textrm{or}\,\,q<\frac{d}{d+2-N}\,\,\textrm{and}\,\,N/p\leq p<q;\\[1ex]
\frac{dq}{Nq-2q-dq+d},\,\,\textrm{if}\,\,q<\frac{d}{d+2-N}\,\,\textrm{and}\,\,p\geq\max\{N/2,q\},\\[1ex]
s\in(2^{\ast}_{_d},\infty),\,\,\textrm{if}\,\,p=N/2\,\,\textrm{and}\,\,q\geq\frac{d}{d+2-N},\,\,\textrm{or}
\,\,q=\frac{d}{d+2-N}\,\,\textrm{and}\,\,p\geq N/2,\\[1ex]
\,\,\,\,\,\infty,\,\,\textrm{if}\,\,p>N/2\,\,\,\textrm{and}\,\,\,q>\frac{d}{d+2-N},
\end{array}
\right.
\end{equation}
\end{theorem}
\indent\\
\begin{proof}
Without loss of generality, we will only give the proof for the case when condition {\bf (C2)} holds. Also, for simplicity,
we assume that $N>2$ (the case $N=2$ possesses sharper embedding results; e.g. Theorem \ref{embeddings}).
First of all, it is easy to see that it is enough to show the estimate (\ref{3.2.06}) for $u$ subsolution of (\ref{1.05})
(cf. Definition \ref{weak-sol}). For a supersolution, the assertion will follow as well, since $-u$ will be a subsolution,
and hence it will be valid for solutions by combining the two inequalities. Let $u\in\mathcal{W}_2(\Omega;\Gamma)$ solve the inequality
\begin{equation}
\label{3.2.09}\mathcal{E}_{\mu}(u,\varphi)\,\leq\,\displaystyle\int_{\Omega}f\varphi\,dx+
\displaystyle\int_{\Gamma}g\varphi\,d\mu\,,\,\,\,\,\,\,\,\forall\,\varphi\in\mathcal{W}_2(\Omega;\Gamma)^+.
\end{equation}
Given $k\geq2$ and $m\geq1$, let $v_{_m},\,w_{_m}\in\mathcal{W}_2(\Omega;\Gamma)^+$ be the functions given by (\ref{3.2.02}).
Then, testing (\ref{3.2.09}) with the function $v_{_m}$, using (\ref{1.04}), and the fact that
$\gamma_k:=\frac{k^2}{4(k-1)}\geq 1$, proceeding as in the proof of Proposition \ref{cont-coercive-form},
we take the sum over all $i,\,j$, and then integrate both sides to obtain that\\[2ex]
$c_0\|\nabla w_{_m}\|^2_{_{2,\Omega}}\,$\\
\begin{equation}\label{k-1st-est}
\leq\,\gamma_k\,\mathcal{E}_{_{\Omega}}(u,v_{_m})+\frac{k}{2}\left(|\Omega|^{^{\frac{1}{N}-\frac{1}{r_1}}}\|\hat{\mathfrak{a}}
\|_{_{r_1,\Omega}}+|\Omega|^{^{\frac{1}{N}-\frac{1}{r_2}}}\|\check{\mathfrak{a}}\|_{_{r_2,\Omega}}\right)
\left(\|w_{_m}\|_{_{\frac{2r_1}{r_1-2},\Omega}}+\|w_{_m}\|_{_{\frac{2r_2}{r_2-2},\Omega}}\right)\|\nabla w_{_m}\|_{_{2,\Omega}}+\|\lambda\|_{_{r_3,\Omega}}\|w_{_m}\|^2_{_{\frac{2r_3}{r_3-1},\Omega}},
\end{equation}
where the form $\mathcal{E}_{_{\Omega}}(\cdot,\cdot)$ is given by (\ref{3.1.03}). By virtue of {\bf (B1)} and {\bf (B2)}, we also have
$$\|w_{_m}\|^2_{_{D(\Lambda_{_{\Gamma}})}}\leq \eta'_0\,k^2\left(\mathcal{E}_{_{\Gamma}}(u,v_{_m})+\gamma_0\|w_{_m}\|^2_{_{2,\Gamma}}\right)+
\|\beta\|_{_{s,\Gamma}}\|w_{_m}\|^2_{_{\frac{2s}{s-1},\Gamma}}$$
for some positive constant $\eta'_0$, and for a constant $\gamma_0\in\mathbb{R\!}\,$, where $\mathcal{E}_{_{\Gamma}}(\cdot,\cdot)$
is defined as in (\ref{3.1.04}). Now, taking into account Young's inequality $2\xi\zeta\leq\epsilon^{-1}\xi^2+\epsilon\zeta^2$
(valid for all $\xi,\,\zeta\geq0$ and $\epsilon>0$) together with Remark \ref{R-embedding}(d), we have that\\[2ex]
$\displaystyle\frac{k}{2}\left(|\Omega|^{^{\frac{1}{N}-\frac{1}{r_1}}}\|\hat{\mathfrak{a}}\|_{_{r_1,\Omega}}+|\Omega|^{^{\frac{1}{N}-\frac{1}{r_2}}}
\|\check{\mathfrak{a}}\|_{_{r_2,\Omega}}\right)\left(\|w_{_m}\|_{_{\frac{2r_1}{r_1-2},\Omega}}+\|w_{_m}\|_{_{\frac{2r_2}{r_2-2},\Omega}}\right)\|\nabla w_{_m}\|_{_{2,\Omega}}$\\
$$\leq\,\frac{\epsilon_1}{2}\|\nabla w_{_m}\|^2_{_{2,\Omega}}+\frac{k^2}{2\epsilon_1}\left(|\Omega|^{^{\frac{1}{N}-\frac{1}{r_1}}}\|\hat{\mathfrak{a}}\|_{_{r_1,\Omega}}+
|\Omega|^{^{\frac{1}{N}-\frac{1}{r_2}}}\|\check{\mathfrak{a}}\|_{_{r_2,\Omega}}\right)^2\left(\|w_{_m}\|_{_{\frac{2r_1}{r_1-2},\Omega}}+\|w_{_m}\|_{_{\frac{2r_2}{r_2-2},\Omega}}\right)^2
\indent\indent\indent\indent\indent$$
$$\leq\,\displaystyle\frac{\epsilon_1}{2}\|\nabla w_{_m}\|^2_{_{2,\Omega}}+
\frac{k^2}{2\epsilon_1}\left(|\Omega|^{^{\frac{1}{N}-\frac{1}{r_1}}}\|\hat{\mathfrak{a}}\|_{_{r_1,\Omega}}+
|\Omega|^{^{\frac{1}{N}-\frac{1}{r_2}}}\|\check{\mathfrak{a}}\|_{_{r_2,\Omega}}\right)^2\left(\epsilon_k\|\nabla w_{_m}\|^2_{_{2,\Omega}}+
\frac{C}{\epsilon_k}\|w_{_m}\|^2_{_{2,\Omega}}\right),$$ for all $\epsilon_1,\,\epsilon_k>0$, and for some constant $C>0$. Choosing
$$\epsilon_k:=\frac{\epsilon^2_1}{k^2\left(|\Omega|^{^{\frac{1}{N}-\frac{1}{r_1}}}\|\hat{\mathfrak{a}}\|_{_{r_1,\Omega}}+
|\Omega|^{^{\frac{1}{N}-\frac{1}{r_2}}}\|\check{\mathfrak{a}}\|_{_{r_2,\Omega}}\right)^2}$$ in the above calculation,
we obtain that\\[2ex]
$\displaystyle\frac{k}{2}\left(|\Omega|^{^{\frac{1}{N}-\frac{1}{r_1}}}\|\hat{\mathfrak{a}}\|_{_{r_1,\Omega}}+|\Omega|^{^{\frac{1}{N}-\frac{1}{r_2}}}
\|\check{\mathfrak{a}}\|_{_{r_2,\Omega}}\right)\left(\|w_{_m}\|_{_{\frac{2r_1}{r_1-2},\Omega}}+\|w_{_m}\|_{_{\frac{2r_2}{r_2-2},\Omega}}\right)\|\nabla w_{_m}\|_{_{2,\Omega}}$\\
\begin{equation}\label{k-2nd-est}
\leq\,\epsilon_1\|\nabla w_{_m}\|^2_{_{2,\Omega}}+\frac{k^4C}{2\epsilon^3_1}\left(|\Omega|^{^{\frac{1}{N}-\frac{1}{r_1}}}\|\hat{\mathfrak{a}}\|_{_{r_1,\Omega}}+
|\Omega|^{^{\frac{1}{N}-\frac{1}{r_2}}}\|\check{\mathfrak{a}}\|_{_{r_2,\Omega}}\right)^4\|w_{_m}\|^2_{_{2,\Omega}},
\end{equation}
for all $\epsilon_1>0$. In a similar way, recalling the compactness of the maps
$\mathcal{W}_2(\Omega;\Gamma)\hookrightarrow L^{^{\frac{2r_3}{r_3-1}}}(\Omega)$ and $\mathcal{W}_2(\Omega;\Gamma)\hookrightarrow L^{^{\frac{2s}{s-1}}}_{\mu}(\Gamma)$
together with Young's inequality and Remark \ref{R-embedding}(d), we get that
\begin{equation}\label{k-3rd-est}
\|\lambda\|_{_{r_3,\Omega}}\|w_{_m}\|^2_{_{\frac{2r_3}{r_3-1},\Omega}}\,\leq\,\epsilon_2\|\nabla w_{_m}\|^2_{_{2,\Omega}}+
C_{_{\frac{\epsilon_2}{\|\lambda\|_{_{r_3,\Omega}}}}}\|w_{_m}\|^2_{_{2,\Omega}},
\end{equation}
and
\begin{equation}\label{k-4th-est}
\|\beta\|_{_{s,\Omega}}\|w_{_m}\|^2_{_{\frac{2s}{s-1},\Gamma}}\,\leq\,\epsilon_3\|\nabla w_{_m}\|^2_{_{2,\Omega}}+
C'_{_{\frac{\epsilon_3}{\|\beta\|_{_{s,\Gamma}}}}}\|w_{_m}\|^2_{_{2,\Omega}},
\end{equation}
and
\begin{equation}\label{k-5th-est}
\gamma_0\|w_{_m}\|^2_{_{2,\Gamma}}\,\leq\,\epsilon_4\|\nabla w_{_m}\|^2_{_{2,\Omega}}+
C''_{_{\frac{\epsilon_4}{|\gamma_0|}}}\|w_{_m}\|^2_{_{2,\Omega}}\,\,\,\,\,\,\,\,\,\,\,\,\,\textrm{(if}\,\,\gamma_0\neq0\,\textrm{)},
\end{equation}
for every $\epsilon_2,\,\epsilon_3,\,\epsilon_4>0$, and for some positive constants
$C_{_{\frac{\epsilon_2}{\|\lambda\|_{_{r_3,\Omega}}}}}$, \,$C'_{_{\frac{\epsilon_3}{\|\beta\|_{_{s,\Gamma}}}}}$,\,
$C''_{_{\frac{\epsilon_4}{|\gamma_0|}}}$. Letting $\epsilon_j:=c_0/8$ for $j\in\{1,2,3,4\}$, we put
\begin{equation}\label{3.2.10}
\delta^{\ast}:=\frac{C}{c^3_0}\left(|\Omega|^{^{\frac{1}{N}-\frac{1}{r_1}}}\|\hat{\mathfrak{a}}\|_{_{r_1,\Omega}}+
|\Omega|^{^{\frac{1}{N}-\frac{1}{r_2}}}\|\check{\mathfrak{a}}\|_{_{r_2,\Omega}}\right)^4
+C_{_{\frac{c_0}{8\|\lambda\|_{_{r_3,\Omega}}}}}+C'_{_{\frac{c_0}{8\|\beta\|_{_{s,\Gamma}}}}}+C''_{_{\frac{c_0}{8|\gamma_0|}}},
\end{equation}
We are lead to the following calculation:\\[2ex]
$\min\{c_0,\,1\}\left(\|\nabla w_{_m}\|^2_{_{2,\Omega}}+\|w_{_m}\|^2_{_{D(\Lambda_{_{\Gamma}})}}\right)\,
\leq\,8\left(\gamma_k\,\mathcal{E}_{_{\Omega}}(u,v_{_m})+\eta'_0\,k^2\mathcal{E}_{_{\Gamma}}(u,v_{_m})
+k^4\delta^{\ast}\|w_{_m}\|^2_{_{2,\Omega}}\right)$\\
$$\indent\indent\indent\indent\indent\,\leq\,8k\left(\mathcal{E}_{_{\Omega}}(u,v_{_m})+2\eta'_0\,k\mathcal{E}_{_{\Gamma}}(u,v_{_m})
+2k^3\delta^{\ast}\|w_{_m}\|^2_{_{2,\Omega}}\right).$$\\
Taking the definition of $v_{_m}$ together with the assumption {\bf (A1)}, we combine the previous estimate to arrive at
\begin{equation}\label{3.2.11}
\|w_{_m}\|^2_{_{\mathcal{V}_2(\Omega;\Gamma)}}\,
\leq\,\frac{8\max\{\eta'_0,\,1\}}{\min\{c_0,\,1\}}k^4\left\{\mathcal{E}_{\mu}(u,v_{_m})+
\left(\delta^{\ast}+|\eta_0|+\min\{c_0,\,1\}(1+|\eta_0|)\right)\|w_{_m}\|^2_{_{2,\Omega}}\right\}.
\end{equation}
Letting $$\gamma^{\ast}:=\frac{8(c_1+c_2)\max\{\eta'_0,\,1\}\max\left\{1,\,\delta^{\ast}+|\eta_0|+\min\{c_0,\,1\}(1+|\eta_0|)\right\}}{\min\{c_0,\,1\}},$$\indent\\
and using the fact that $u\in\mathcal{W}_2(\Omega;\Gamma)$ fulfill (\ref{3.2.09}),
we get by means of (\ref{3.2.11}) together with (\ref{2.03}) and (\ref{2.04}), that
\begin{equation}\label{3.2.12}
\|w_{_m}\|^2_{_{\frac{2N}{N-2},\Omega}}+\|w_{_m}\|^2_{_{\frac{2d}{N-2},\Gamma}}\,\leq\,(c_1+c_2)\|w_{_m}\|^2_{_{\mathcal{V}_2(\Omega;\Gamma)}}
\,\leq\,k^4\gamma^{\ast}\left(\displaystyle\int_{\Omega}fv_{_m}\,dx+\displaystyle\int_{\Gamma}gv_{_m}\,d\mu+\|w_{_m}\|^2_{_{2,\Omega}}\right).
\end{equation}
As $v_{_m}\,\leq\,w^{^{2(k-1)/k}}_{_m}$ for each $k\geq2$, we see that
\begin{equation}
\label{3.2.13}\displaystyle\int_{\Omega}fv_{_m}\,dx\,\leq\,
\|f\|_{_{p,\Omega}}\|v_{_m}\|_{_{\frac{p}{p-1},\Omega}}\,\leq\,
\|f\|_{_{p,\Omega}}\|w_{_m}\|^{^{\frac{2(k-1)}{k}}}_{_{\left(\frac{p}{p-1}\right)\frac{2(k-1)}{k},\Omega}}
\end{equation}
and
\begin{equation}
\label{3.2.14}\displaystyle\int_{\Gamma}gv_{_m}\,d\mu\,\leq\,
\|g\|_{_{q,\Gamma}}\|v_{_m}\|_{_{\frac{q}{q-1},\Gamma}}\,\leq\,
\|g\|_{_{q,\Gamma}}\|w_{_m}\|^{^{\frac{2(k-1)}{k}}}_{_{\left(\frac{q}{q-1}\right)\frac{2(k-1)}{k},\Gamma}}.
\end{equation}\\
Furthermore, by a well known interpolation inequality (e.g. \cite{TAR07}, Lemma 8.2), we observe that
\begin{equation}\label{3.2.15}
\|w_{_m}\|^2_{_{2,\Omega}}\,\leq\,\|w_{_m}\|^{^{2(1-\theta)}}_{_{\frac{2N}{N-2},\Omega}}
\|w_{_m}\|^{^{2\theta}}_{_{\zeta,\Omega}}\,,
\end{equation}
where $\theta\in(0,1)$ is chosen such that $2^{-1}=\theta\zeta^{-1}+(1-\theta)(N-2)(2N)^{-1}$, and where
$\zeta:=4p(3p-2)^{-1}\,\leq\,2$. To estimate $\|w_{_m}\|^{^{2\theta}}_{_{\zeta,\Omega}}$, we use the fact that $w^{2/k}_{_m}\leq u^+$ together with
H\"older's inequality to deduce that
$$\|w_{_m}\|^2_{_{\zeta,\Omega}}=\left\|w^{^{\frac{2}{k}+\frac{2}{k}(k-1)}}_{_m}\right\|_{_{\frac{\zeta}{2},\Omega}}\,\leq\,
\|u^+\|_{_{2,\Omega}}\|w_{_m}\|^{^{\frac{2(k-1)}{k}}}_{_{\left(\frac{p}{p-1}\right)\frac{2(k-1)}{k},\Omega}}.$$
Inserting this estimate in (\ref{3.2.15}), we apply Young's inequality in (\ref{3.2.15} to achieve
\begin{equation}
\label{3.2.16}\|w_{_m}\|^2_{_{2,\Omega}}\,\leq\,c^{\ast}\,\|u^+\|_{_{2,\Omega}}
\|w_{_m}\|^{^{\frac{2(k-1)}{k}}}_{_{\left(\frac{p}{p-1}\right)\frac{2(k-1)}{k},\Omega}}+\epsilon\,\|w_{_m}\|^2_{_{\mathcal{W}_2(\Omega;\Gamma)}},
\end{equation}\\
for some constant $c^{\ast}>0$ and for all $\epsilon>0$. We now deal each of the cases (five in total), in accordance to the values of the
constants (\ref{m1}) and (\ref{m2}).\\[2ex]
$\bullet$\,\,\,{\it \underline{\textsc{Case 1:}} $p<N/2$ \,and\, $q<\frac{d}{d+2-N}$}.\\[2ex]
For this case, choose
\begin{equation}\label{Case1-01}
k=\vartheta_{p,q}:=\min\left\{\frac{p(N-2)}{N-2p},\,\frac{q(N-2)}{Nq-2q-dq+d}\right\}.
\end{equation}
Using this selection together with (\ref{2.03}) and (\ref{2.04}), from (\ref{3.2.13}), (\ref{3.2.14}), and (\ref{3.2.16}), we get that
\begin{equation}
\label{Case1-02}\displaystyle\int_{\Omega}fv_{_m}\,dx\,\leq\,
C\|f\|_{_{p,\Omega}}\|w_{_m}\|^{^{\frac{2(\vartheta_{p,q}-1)}{\vartheta_{p,q}}}}_{_{\frac{2N}{N-2},\Omega}}
\leq\,c^{^{\frac{2(\vartheta_{p,q}-1)}{\vartheta_{p,q}}}}_1C\|f\|_{_{p,\Omega}}
\|w_{_m}\|^{^{\frac{2(\vartheta_{p,q}-1)}{\vartheta_{p,q}}}}_{_{\mathcal{W}_2(\Omega;\Gamma)}},
\end{equation}
and
\begin{equation}
\label{Case1-03}\displaystyle\int_{\Gamma}gv_{_m}\,d\mu\,\leq\,
C'\|g\|_{_{q,\Gamma}}\|w_{_m}\|^{^{\frac{2(\vartheta_{p,q}-1)}{\vartheta_{p,q}}}}_{_{\frac{2d}{N-2},\Gamma}}
\leq\,c^{^{\frac{2(\vartheta_{p,q}-1)}{\vartheta_{p,q}}}}_2C'\|g\|_{_{q,\Gamma}}
\|w_{_m}\|^{^{\frac{2(\vartheta_{p,q}-1)}{\vartheta_{p,q}}}}_{_{\mathcal{W}_2(\Omega;\Gamma)}},
\end{equation}
and
\begin{equation}
\label{Case1-04}\|w_{_m}\|^2_{_{2,\Omega}}\,\leq\,c^{^{\frac{2(\vartheta_{p,q}-1)}{\vartheta_{p,q}}}}_1C''\|u^+\|_{_{2,\Omega}}
\|w_{_m}\|^{^{\frac{2(\vartheta_{p,q}-1)}{\vartheta_{p,q}}}}_{_{\mathcal{W}_2(\Omega;\Gamma)}}+\epsilon\,\|w_{_m}\|^2_{_{\mathcal{W}_2(\Omega;\Gamma)}},
\end{equation}
for some constants $C,\,C',\,C''>0$. Selecting $\epsilon>0$ appropriately, and combining (\ref{3.2.12}), (\ref{Case1-02}), (\ref{Case1-03}),
and (\ref{Case1-04}), we get the existence of a constant $C_0>0$, such that
$$\|w_{_m}\|^2_{_{\mathcal{W}_2(\Omega;\Gamma)}}\,\leq\,C_0\left(\|f\|_{_{p,\Omega}}+\|g\|_{_{q,\Gamma}}+\|u^+\|_{_{2,\Omega}}\right)
\|w_{_m}\|^{^{\frac{2(\vartheta_{p,q}-1)}{\vartheta_{p,q}}}}_{_{\mathcal{W}_2(\Omega;\Gamma)}},$$
from where we arrive at
\begin{equation}\label{Case1-05}
\left\|w^{^{\frac{2}{\vartheta_{p,q}}}}_{_m}\right\|_{_{\frac{N\vartheta_{p,q}}{N-2},\Omega}}+
\left\|w^{^{\frac{2}{\vartheta_{p,q}}}}_{_m}\right\|_{_{\frac{d\vartheta_{p,q}}{N-2},\Gamma}}
\,\leq\,C_0\left(\|f\|_{_{p,\Omega}}+\|g\|_{_{q,\Gamma}}+\|u^+\|_{_{2,\Omega}}\right).
\end{equation}
Since the sequence $w^{^{\frac{2}{\vartheta_{p,q}}}}_{_m}:=\left[\psi_{_{\frac{\vartheta_{p,q}}{2},m}}(u)\right]^{^{\frac{2}{\vartheta_{p,q}}}}$
(for the function $\psi_{_{t,m}}(\cdot)$ given by (\ref{3.2.01}))
is increasing with $w^{^{\frac{2}{\vartheta_{p,q}}}}_{_m}\stackrel{m\rightarrow\infty}{\longrightarrow}u^+$, and application of
the Monotone Convergence Theorem of Lebesgue in (\ref{Case1-05}) gives
\begin{equation}\label{Case1-06}
|\|\textbf{u}^+\||_{_{\frac{N\vartheta_{p,q}}{N-2},\frac{d\vartheta_{p,q}}{N-2}}}\,\leq\,C_0\left(|\|(f,g)\||_{_{p,q}}+\|u^+\|_{_{2,\Omega}}\right).
\end{equation}
This completes the proof of the first case of the theorem.\\[2ex]
$\bullet$\,\,\,{\it \underline{\textsc{Case 2:}} $p<N/2<\frac{d}{d+2-N}\geq q$, or\,
$q<\frac{d}{d+2-N}$\, and \,$N/p\leq p<q$}.\\[2ex]
For this case, we put
\begin{equation}\label{Case2-01}
k:=\varsigma_p:=\frac{p(N-2)}{Np-2p-dp+d}.
\end{equation}
Inserting this in (\ref{3.2.13}), (\ref{3.2.14}), and (\ref{3.2.16}), and using the fact that $N-2<d<N$ and $q'<p'$, we get that
$$\displaystyle\int_{\Omega}fv_{_m}\,dx
\,\leq\,\|f\|_{_{p,\Omega}}\|w_{_m}\|^{^{\frac{2(\varsigma_{p}-1)}{\varsigma_{p}}}}_{_{\frac{2d}{N-2},\Omega}}
\leq\,c^{^{\frac{2(\varsigma_{p}-1)}{\varsigma_{p}}}}_1C_1\|f\|_{_{p,\Omega}}
\|w_{_m}\|^{^{\frac{2(\varsigma_{p}-1)}{\varsigma_{p}}}}_{_{\mathcal{W}_2(\Omega;\Gamma)}},$$
and
$$\displaystyle\int_{\Gamma}gv_{_m}\,d\mu\,\leq\,\|g\|_{_{q,\Gamma}}
\|w_{_m}\|^{^{\frac{2(\varsigma_{p}-1)}{\varsigma_{p}}}}_{_{\left(\frac{q}{q-1}\right)\frac{2(\varsigma_p-1)}{\varsigma_p},\Gamma}}
\,\leq\,C'_1\|g\|_{_{q,\Gamma}}\|w_{_m}\|^{^{\frac{2(\varsigma_{p}-1)}{\varsigma_{p}}}}_{_{\frac{2d}{N-2},\Gamma}}
\leq\,c^{^{\frac{2(\varsigma_{p}-1)}{\varsigma_{p}}}}_2C''_1\|g\|_{_{q,\Gamma}}
\|w_{_m}\|^{^{\frac{2(\varsigma_{p}-1)}{\varsigma_{p}}}}_{_{\mathcal{W}_2(\Omega;\Gamma)}},$$
and
$$\|w_{_m}\|^2_{_{2,\Omega}}\,\leq\,c^{^{\frac{2(\varsigma_{p}-1)}{\varsigma_{p}}}}_1\|u^+\|_{_{2,\Omega}}
\|w_{_m}\|c^{^{\frac{2(\varsigma_{p}-1)}{\varsigma_{p}}}}_{_{\mathcal{W}_2(\Omega;\Gamma)}}+\epsilon\,\|w_{_m}\|^2_{_{\mathcal{W}_2(\Omega;\Gamma)}},$$
for some constants $C_1,\,C'_1,\,C''_1>0$. Proceeding in the exact way as the previous case, we have
\begin{equation}\label{Case2-02}
|\|\textbf{u}^+\||_{_{\frac{N\varsigma_{p}}{N-2},\frac{d\varsigma_{p}}{N-2}}}\,\leq\,C'_0\left(|\|(f,g)\||_{_{p,q}}+\|u^+\|_{_{2,\Omega}}\right)
\end{equation}
for some positive constant $C_0>0$, which is the desired conclusion for the second case.\\[2ex]
$\bullet$\,\,\,{\it \underline{\textsc{Case 3:}} $q<\frac{d}{d+2-N}$ \,and\, $p\geq\max\{N/2,q\}$}.\\[2ex]
Letting
\begin{equation}\label{Case3-01}
k:=\varsigma_q:=\frac{q(N-2)}{Nq-2q-dq+d},
\end{equation}
this case follows in the same way as in the previous case; we deduce the existence of a constant $C''_0>0$ such that
\begin{equation}\label{Case3-02}
|\|\textbf{u}^+\||_{_{\frac{N\varsigma_{q}}{N-2},\frac{d\varsigma_{q}}{N-2}}}\,\leq\,C''_0\left(|\|(f,g)\||_{_{p,q}}+\|u^+\|_{_{2,\Omega}}\right).
\end{equation}
$\bullet$\,\,\,{\it \underline{\textsc{Case 4:}} $p=N/2$\, and \,$q\geq\frac{d}{d+2-N}$, or
\,$q=\frac{d}{d+2-N}$\, and \,$p\geq N/2$}.\\[2ex]
We will prove the case when $p=N/2$\, and \,$q\geq\frac{d}{d+2-N}$; the other situation follows in an analogous manner.
Given $r>2^{\ast}_{_N}$, \,$s>2^{\ast}_{_d}$ arbitrarily fixed, put $\tau_{r,s}:=\max\{r,s,N\}$.
Then there exists $p_0\in[(2^{\ast}_{_N})',N/2)$ such that
$$f\in L^{p_0}(\Omega)\indent\indent\textrm{and}\indent\indent\frac{p_0}{N-2p_0}=\frac{\tau_{r,s}}{d}.$$
Given $r>2^{\ast}_{_N}$, \,$s>2^{\ast}_{_d}$ arbitrarily fixed, put $\tau_{r,s}:=\max\{r,s,N\}$.
In the same way, one can find $q_0\in[(2^{\ast}_{_d})',d(d+2-N)^{-1})$ such that
$$g\in L^{q_0}_{\mu}(\Gamma)\indent\indent\textrm{and}\indent\indent\frac{q_0}{Nq_0-2q_0-dq_0+d}=\frac{\tau_{r,s}}{d}.$$
Then, selecting $\vartheta_{p_0,q_0}$ as in (\ref{Case1-01}), and applying multiple times H\"older's inequality,
we end up getting the inequality
\begin{equation}\label{Case4-01}
|\|\textbf{u}^+\||_{_{r,s}}\,\leq\,C_2\left(|\|(f,g)\||_{_{p_0,q_0}}+\|u^+\|_{_{2,\Omega}}\right)
\,\leq\,C'_2\left(|\|(f,g)\||_{_{p,q}}+\|u^+\|_{_{2,\Omega}}\right),
\end{equation}
for some constants $C_2,\,C'_2>0$. This establishes the fourth case.\\[2ex]
$\bullet$\,\,\,{\it \underline{\textsc{Case 5:}} $p>N/2$\, and \,$q>\frac{d}{d+2-N}$}.\\[2ex]
For this part, a complete proof is given for the parabolic case in the proof of Theorem \ref{Thm-Partial-Bounded}, so we will only sketch the main
steps of this corresponding elliptic version.
Let $u\in\mathcal{W}_2(\Omega;\Gamma)$ be a weak solution of problem (\ref{1.05}), and
let $u_k$ be the function given by (\ref{u_k def}). Then we have
\begin{equation}\label{thm0-03}
\mathcal{E}_{\mu}(u,u_k)=\displaystyle\int_{\Omega_k}fu_k\,dx+\displaystyle\int_{\Gamma_k}gu_k\,d\mu,
\end{equation}
where we recall the definition of the form $\mathcal{E}_{\mu}(\cdot,\cdot)$ in (\ref{3.1.01}). Estimating the right hand side in (\ref{thm0-03}) and taking
into account Assumptions {\bf (A3)}, {\bf (B4)}, and (\ref{1.04}), we apply Young's inequality multiple times and use Remark \ref{R-embedding}(d)
to infer (after multiple calculations) that
\begin{equation}\label{thm0-10}
\mathcal{E}_{\mu}(u,u_k)\,\geq\,\displaystyle\frac{c_0}{4}\|\nabla u_k\|^2_{_{2,\Omega_k}}+
c^{\#}\|u_k\|^2_{_{D(\Lambda_{_{\Gamma}})}}-
\displaystyle\int_{\Omega_k}\mathfrak{h}(|u_k|^2+k^2)\,dx-\displaystyle\int_{\Gamma_k}\mathfrak{k}(|u_k|^2+k^2)\,d\mu,
\end{equation}
for some constant $c^{\#}>0$, where
$$\mathfrak{h}:=\frac{1}{c_0}\displaystyle\sum^N_{i=1}|\hat{a}_i|^2+M^{\star}+|\lambda|\in L^{^{\min\left\{\frac{r_1}{2},r_3\right\}}}(\Omega)^+$$
and $$\mathfrak{k}:=\min\{1,c^{\star}_1\}(|c^{\ast}_1|+2|\beta|)+c^{\star}_3|\rho|\in L^{^{\min\{s,\theta\}}}_{\mu}(\Gamma)^+,$$
for some constant $c^{\ast}_1>0$, and where
$$M^{\star}:=\max\left\{\|\hat{\mathfrak{a}}\|_{_{r_1,\Omega}},\|\check{\mathfrak{a}}\|_{_{r_2,\Omega}}\right\}
\left(C_1\|\hat{\mathfrak{a}}\|_{_{r_1,\Omega}}+C_2\|\check{\mathfrak{a}}\|_{_{r_2,\Omega}}\right)/c_0>0,$$
for some constants $C_1,\,C_2>0$. It is also easy to see that
\begin{equation}\label{thm0-09}
\displaystyle\int_{\Omega_k}fu_k\,dx+\displaystyle\int_{\Gamma_k}gu_k\,d\mu\,
\leq\,\displaystyle\frac{1}{k}\displaystyle\int_{\Omega_k}|f|(|u_k|^2+k^2)\,dx+
\displaystyle\frac{1}{k}\displaystyle\int_{\Gamma_k}|g|(|u_k|^2+k^2)\,d\mu.
\end{equation}
Substituting all this into (\ref{thm0-03}) and recalling the equivalence of the norms $\|\cdot\|_{_{\mathcal{W}_2(\Omega;\Gamma)}}$
and $\|\cdot\|_{_{\mathcal{W}_2(\overline{\Omega})}}$, we deduce that
\begin{equation}\label{thm0-12}
\min\left\{1,\displaystyle\frac{c_0}{4},c^{\#}\right\}\|u_k\|^2_{_{\mathcal{W}_2(\overline{\Omega})}}\,\leq\,c\|u_k\|^2_{_{2,\Omega_k}}+
c\left(\displaystyle\int_{\Omega_k}\left(\frac{1}{k}|f|+\mathfrak{h}\right)(|u_k|^2+k^2)\,dx
+\displaystyle\int_{\Gamma_k}\left(\frac{1}{k}|g|+\mathfrak{k}\right)(|u_k|^2+k^2)\,d\mu\right),
\end{equation}
for some constant $c>0$. We now estimate the last two terms in (\ref{thm0-12}). Put
\begin{equation}\label{thm0-13}
\hat{k}^2_0:=\|f\|^2_{_{p,\Omega}}+\|g\|^2_{_{q,\Gamma}},\,\,\,\,\,\textrm{and}\,\,\,\,\,k\geq\hat{k}_0,
\end{equation}
and apply H\"older's inequality together with Young's inequality and Remark \ref{R-embedding}(d) to deduce that
\begin{equation}\label{thm0-14}
\displaystyle\int_{\Omega_k}\frac{1}{k}|f||u_k|^2\,dx\,\leq\,\displaystyle\frac{1}{k}\|f\|_{_{p,\Omega}}\|u_k\|^2_{_{\frac{2p}{p-1},\Omega_k}}
\,\leq\,\displaystyle\frac{k_0}{k}\left(\epsilon_1\|\nabla u_k\|^2_{_{2,\Omega_k}}+C_{\epsilon_1}\|u_k\|^2_{_{2,\Omega_k}}\right)
\,\leq\,\epsilon_1\|u_k\|^2_{_{\mathcal{W}_2(\overline{\Omega})}}+C_{\epsilon_1}\|u_k\|^2_{_{2,\Omega_k}},
\end{equation}
for all $\epsilon_1>0$, and for some constant $C_{\epsilon_1}>0$. Also, since $\mathfrak{h}\in L^{^{\min\{r_1,r_3\}}}(\Omega)$ and
$$\mathfrak{r}:=2\min\{r_1/2,\,r_3\}\left(\min\{r_1/2,\,r_3\}-1\right)^{-1}<
2N(N-2)^{-1},$$ proceeding in the same way as in the derivation of (\ref{thm0-14}), we get that
\begin{equation}\label{thm0-16}
\displaystyle\int_{\Omega_k}\mathfrak{h}|u_k|^2\,dx
\,\leq\,\|\mathfrak{h}\|_{_{\min\left\{\frac{r_1}{2},r_3\right\},\Omega}}\left(\epsilon_2\| u_k\|^2_{_{\mathcal{W}_2(\overline{\Omega})}}+C_{\epsilon_2}\|u_k\|^2_{_{2,\Omega_k}}\right),
\end{equation}
for every $\epsilon_2>0$, and for some constant $C_{\epsilon_2}>0$. Moreover, using the fact that $\mathfrak{k}\in L^{^{\min\{s,\theta\}}}_{\mu}(\Gamma)$ with
$$\mathfrak{s}:=2\min\{s,\theta\}\left(\min\{s,\theta\}-1\right)^{-1}<2d(N-2)^{-1},$$ we proceed in the exact way as above to deduce that
\begin{equation}\label{thm0-17}
\displaystyle\int_{\Gamma_k}\frac{1}{k}|g||u_k|^2\,d\mu
\,\leq\,\epsilon_3\|u_k\|^2_{_{\mathcal{W}_2(\overline{\Omega})}}+C_{\epsilon_3}\|u_k\|^2_{_{2,\Omega_k}},
\end{equation}
for all $\epsilon_3>0$, and for some constant $C_{\epsilon_3}>0$, and
\begin{equation}\label{thm0-18}
\displaystyle\int_{\Gamma_k}\mathfrak{k}|u_k|^2\,d\mu
\,\leq\,\|\mathfrak{k}\|_{_{\min\{s,\theta\},\Gamma}}\left(\epsilon_4\|u_k\|^2_{_{\mathcal{W}_2(\overline{\Omega})}}+C_{\epsilon_4}\|u_k\|^2_{_{2,\Omega_k}}\right),
\end{equation}
for every $\epsilon_4>0$, and for some constant $C_{\epsilon_4}>0$. Choosing
$$\epsilon_1=\epsilon_3=\frac{1}{8}\min\left\{1,\displaystyle\frac{c_0}{4},c^{\#}\right\},\,\,\,\,\,\,\,\,\,\,
\epsilon_2=\frac{1}{8\|\mathfrak{h}\|_{_{\min\left\{\frac{r_1}{2},r_3\right\},\Omega}}}\min\left\{1,\displaystyle\frac{c_0}{4},c^{\#}\right\},
\,\,\,\,\,\,\,\,\,\,\epsilon_4=\frac{1}{8\|\mathfrak{k}\|_{_{\min\{s,\theta\},\Gamma}}}\min\left\{1,\displaystyle\frac{c_0}{4},c^{\#}\right\},$$
we insert (\ref{thm0-14}), (\ref{thm0-16}), (\ref{thm0-17}), and (\ref{thm0-18}) into (\ref{thm0-12}), to arrive at
\begin{equation}\label{thm0-19}
\|u_k\|^2_{_{\mathcal{W}_2(\overline{\Omega})}}\,\leq\,\displaystyle\frac{2c}{\min\{1,c_0/4,c^{\#}\}}\left\{
C^{\star}_0\|u_k(t)\|^2_{_{2,\Omega_k}}\right.+
\left.k^2\displaystyle\int_{\Omega_k}\left(\frac{1}{k}|f|+\mathfrak{h}\right)\,dx\,+\,
k^2\displaystyle\int_{\Gamma_k}\left(\frac{1}{k}|g|+\mathfrak{k}\right)\,d\mu\right\},
\end{equation}
where
$$C^{\star}_0:=1+C_{\epsilon_1}+C_{\epsilon_2}\|\mathfrak{h}\|_{_{\min\{r_1,r_3\},\Omega}}
+C_{\epsilon_3}+C_{\epsilon_4}\|\mathfrak{k}\|_{_{\min\{s,\theta\},\Gamma}}.$$
Having this, we now produce estimates in the last two integral terms in (\ref{thm0-19}). In fact,
recalling that $k\geq k_0$ and applying H\"older's inequality, we see that
\begin{equation}\label{thm0-20}
\int_{\Omega_k}\frac{1}{k}|f|\,dx\,\leq\,|\Omega_k|^{^{1-\frac{1}{p}}},\,\,\,\,\,\displaystyle\int_{\Omega_k}\mathfrak{h}\,dx
\,\leq\,\|\mathfrak{h}\|_{_{\min\left\{\frac{r_1}{2},r_3\right\},\Omega}}|\Omega_k|^{2/\mathfrak{r}},
\end{equation}
and
\begin{equation}\label{thm0-22}
\displaystyle\int_{\Gamma_k}\frac{1}{k}|g|\,d\mu\,\leq\,\mu(\Gamma_k)^{^{1-\frac{1}{q}}},\,\,\,\,\,
\int_{\Gamma_k}\mathfrak{k}\,d\mu\,dt\,\leq\,\|\mathfrak{k}\|_{_{\min\{s,\theta\},\Gamma}}
\mu(\Gamma_k)^{2/\mathfrak{s}}dt.
\end{equation}
Put $$\gamma_0:=\displaystyle\frac{2c\max\left\{1,C^{\star}_0,\|\mathfrak{h}\|_{_{\min\left\{\frac{r_1}{2},r_3\right\},\Omega}},
\|\mathfrak{k}\|_{_{\min\{s,\theta\},\Gamma}}\right\}}{\min\{1,c_0/4,c^{\#}\}},$$
and we choose
$$\theta_{1,1}:=\frac{2p-N}{(N-2)p},\,\,\,\,\,\,\,
\theta_{1,2}:=\frac{2N+2\mathfrak{r}-N\mathfrak{r}}{(N-2)\mathfrak{r}},\,\,\,\,\,\,\,s_{1,1}=s_{1,2}:=2^{\ast}_{_N},$$
and
$$\theta_{2,1}:=\frac{dq+2q-Nq-d}{(N-2)q},\,\,\,\,\,\,\,
\theta_{2,2}:=\frac{2d+2\mathfrak{s}-N\mathfrak{s}}{(N-2)\mathfrak{s}},\,\,\,\,\,\,\,s_{2,1}=s_{2,2}:=2^{\ast}_{_d}.$$
Recalling the assumptions, one clearly sees that the parameters $\theta_{1,\xi},\,\theta_{2,\zeta}$ are all positive.
Using this, we substitute (\ref{thm0-20}), and (\ref{thm0-22}) into (\ref{thm0-19}), to arrive at
\begin{equation}\label{thm0-25}
\|u_k\|^2_{_{\mathcal{W}_2(\overline{\Omega})}}\,\leq\,\gamma_0\left(\|u_k\|^2_{_{2,\Omega_k}}
+k^2\displaystyle\sum^{2}_{\xi=1}|\Omega_k|
^{^{\frac{2(1+\theta_{1,\xi})}{s_{1,\xi}}}}+k^2\displaystyle\sum^{2}_{\zeta=1}\mu(\Gamma_k)^{^{\frac{2(1+\theta_{2,\zeta})}{s_{2,\zeta}}}}\right).
\end{equation}
Therefore, all the conditions of Lemma \ref{Lem1} are fulfilled, and thus
applying Lemma \ref{Lem1} to both $u$ and $-u$ (the latter one being a weak solution of (\ref{1.05}) for $f,\,g$ replaced by
$-f,\,-g$, respectively) gives the desired $L^{\infty}$-bound, completing the proof of the theorem.
\end{proof}

\begin{remark}\label{improved-norms}
Under some additional conditions, the constants $m_1(p,q)$ and $m_2(p,q)$ in Theorem \ref{elliptic-bounded} can be improved.
Here we bring two particular situations. We are assuming that $N>2$.
\begin{enumerate}
\item[(a)]\,\,\,If $p<N/2$ and $q\geq p(N-2)(N-2p)^{-1}$, then it is easy to deduce (as in the proof of Theorem \ref{elliptic-bounded})
that (\ref{3.2.06}) holds for
$$m_1(p,q):=m_1(p)=\frac{Np}{N-2p}\,\,\,\,\,\,\,\,\textrm{and}\,\,\,\,\,\,\,\, m_2(p,q):=m_2(p)=\frac{dp}{N-2p}.$$
\item[(b)]\,\,\,Let $\Omega$ be a bounded Lipschitz domain, and assume that condition {\bf (C2)} holds. Then it is well-known that
$\Gamma$ is a compact $(N-1)$-dimensional Riemannian manifold, and thus (\ref{2.04a}) holds.
Consequently, given $p\geq(2^{\ast}_{_N})'$ and $q\geq(2^{\star}_{_{\Gamma}})':=2(N-1)(N+1)^{-1}$,
if $N>3$, the conclusions of Theorem \ref{elliptic-bounded} hold in this case for
$$m_1(p,q):=\left\{\begin{array}{lcl}
\frac{N}{N-2}\min\left\{\frac{p(N-2)}{N-2p},\,\frac{q(N-3)}{N-2q-1}\right\},\,\,\textrm{if}\,\,p<N/2\,\,\textrm{and}\,\,q<\frac{N-1}{2},\\[1ex]
\frac{Np}{Np-2p},\,\,\textrm{if}\,\,p<N/2\,\,\textrm{and}\,\,\frac{N-1}{2}\leq q,\\[1ex]
\frac{(N-1)q}{N-2q-1},\,\,\textrm{if}\,\,q<\frac{N-1}{2}\,\,\textrm{and}\,\,p\geq N/2,\\[1ex]
r\in(2^{\ast}_{_N},\infty),\,\,\textrm{if}\,\,p=N/2\,\,\textrm{and}\,\,q\geq\frac{N-1}{2},\,\,\textrm{or}
\,\,q=\frac{N-1}{2}\,\,\textrm{and}\,\,p\geq N/2,\\[1ex]
\,\,\,\,\,\infty,\,\,\textrm{if}\,\,p>N/2\,\,\,\textrm{and}\,\,\,q>\frac{N-1}{2},
\end{array}
\right.$$
and
$$m_2(p,q):=\left\{\begin{array}{lcl}
\frac{N-1}{N-2q-1}\min\left\{\frac{p(N-2)}{N-2p},\,\frac{q(N-3)}{N-2q-1}\right\},\,\,\textrm{if}\,\,p<N/2\,\,\textrm{and}\,\,q<\frac{N-1}{2},\\[1ex]
\frac{Np}{Np-2p},\,\,\textrm{if}\,\,p<N/2\,\,\textrm{and}\,\,\frac{N-1}{2}\leq q,\\[1ex]
\frac{(N-1)q}{N-2q-1},\,\,\textrm{if}\,\,q<\frac{N-1}{2}\,\,\textrm{and}\,\,p\geq N/2,\\[1ex]
s\in(2^{\star}_{_\Gamma},\infty),\,\,\textrm{if}\,\,p=N/2\,\,\textrm{and}\,\,q\geq\frac{N-1}{2},\,\,\textrm{or}
\,\,q=\frac{N-1}{2}\,\,\textrm{and}\,\,p\geq N/2,\\[1ex]
\,\,\,\,\,\infty,\,\,\textrm{if}\,\,p>N/2\,\,\,\textrm{and}\,\,\,q>\frac{N-1}{2},
\end{array}
\right.$$
Furthermore, when $N=3$, we obtain that
$$m_1(p,q)=m_p:=\left\{\begin{array}{lcl}
\frac{Np}{N-2p},\,\,\textrm{if}\,\,p<N/2,\\[1ex]
r\in(2^{\ast}_{_N},\infty),\,\,\textrm{if}\,\,p=N/2,\\[1ex]
\,\,\,\,\,\infty,\,\,\textrm{if}\,\,p>N/2,
\end{array}
\right.$$
and
$$m_2(p,q)=m_q:=\left\{\begin{array}{lcl}
s\in(2^{\star}_{_\Gamma},\infty),\,\,\textrm{if}\,\,p\leq N/2,\\[1ex]
\,\,\,\,\,\infty,\,\,\textrm{if}\,\,p>N/2.
\end{array}
\right.$$
These improvements hold in the case when $\Gamma$ is a compact $(N-1)$-dimensional Riemannian manifold,
and $\mu=\mu(\mathfrak{g})$ (for $\mathfrak{g}$ the Riemannian metric on $\Gamma$), provided that all the conditions in Assumption \ref{As1}
are met, with the exception that in this case it is not required that the upper Alhfors condition for the measure $\mu(\mathfrak{g})$.
More details about this are given in subsection \ref{subsec06-01}.
\end{enumerate}
\end{remark}

\begin{remark}\label{coercivity}
We consider the situation when the form $\mathcal{E}_{\mu}(\cdot,\cdot)$ is coercive. In this case the right hand bound in (\ref{3.2.06}) becomes independent on $u$, as we show below.
\begin{enumerate}
\item[(a)]\,\,\,Let $\xi\in\mathbb{R\!}$\,, and assume that
the measurable coefficient $\lambda(x)$ in the differential operator $\mathcal{A}$ in (\ref{1.01}) is of the form
$\lambda(x):=\upsilon(x)+\xi$, where $\upsilon\in L^{r_3}(\Omega)$ ($r_3>N/2$) and $\xi>\delta^{\star}_0$, for
$\delta^{\star}_0>0$ denotes the constant given by (\ref{delta-star}) for $\lambda$ replaced by $\upsilon$ (in (\ref{delta-star})). Then one can deduce that the
bilinear form $\mathcal{E}_{\mu}$ is coercive, and consequently, by virtue of Theorem \ref{elliptic-bounded}, if $u \in \mathcal{W}_2(\Omega;\Gamma)$ solves (\ref{1.05}), we have that\\
\begin{equation}\label{coercive-infinity}
|\|\mathbf{u}\||_{_{m_1(p,q),m_2(p,q)}}\,\leq\,c\,|\|(f,g)\||_{_{p,q}},
\end{equation}
\indent\\
for some constant $c>0$, where the constants $m_1(p,q),\,m_2(p,q)$ are given by (\ref{m1}) and (\ref{m2}),
respectively. To see this, observe that if
$u \in \mathcal{W}_2(\Omega;\Gamma)$ solves (\ref{3.1.07}), then $u$ is a weak solution of the boundary value problem
\begin{equation}
\label{ENBVP2}\left\{
\begin{array}{lcl}
\mathcal{A}u\,=\,f-\xi u\,\,\,\,\,\,\,\,\,\,\textrm{in}\,\,\Omega\\
\mathcal{B}u\,=\,\,\,\,\,g\,\,\,\,\,\,\,\,\,\,\,\,\,\,\,\,\,\,\,\,\,\textrm{on}\,\,\Gamma\\
\end{array}
\right.
\end{equation}
where the differential operators $\mathcal{A}$ and $\mathcal{B}$ are given in (\ref{1.01}) and (\ref{1.02}),
respectively. Hence, following the proof of Proposition \ref{cont-coercive-form}, we notice that\\[2ex]
$\|u\|^2_{_{\mathcal{W}_2(\Omega;\Gamma)}}\,\leq\,\displaystyle\frac{2}{\min\{c_0,c^{\star}_0\}}
\left(\mathcal{E}_{\mu}(u,u)+\delta^{\star}_0\|u\|^2_{_{2,\Omega}}\right)$\\
$$=\displaystyle\frac{2}{\min\{c_0,c^{\star}_0\}}\left(\displaystyle\int_{\Omega}fu\,dx+\displaystyle\int_{\Gamma}gu\,d\mu
-\left(\xi-\delta^{\star}_0)\right)\|u\|^2_{_{2,\Omega}}\right)\,
\leq\,\displaystyle\frac{2c'}{\min\{c_0,c^{\star}_0\}}|\|(f,g)\||_{_{p,q}}\|u\|_{_{\mathcal{W}_2(\Omega;\Gamma)}},$$
for some constant $c'>0$. This proves our claim.
\item[(b)]\,\,\,For bounded extension domains whose boundaries are upper $d$-sets (with respect to $\mu$),
for $d\in(N-2,N)$, one has that Remark \ref{R-embedding}(b)
holds. Then, consider now the case when $\beta(x)=\beta_0(x)+\zeta$, for $\beta_0\in L^{s}_{\mu}(\Gamma)$ ($s>d(d+2-N)^{-1}$)
and $\zeta>\delta^{\star}_0C_{\tau^{\ast}}$, where $\delta^{\star}_0>0$ is given by (\ref{delta-star}) (for $\beta_0$ instead of $\beta$), and for
$$\tau^{\ast}:=\frac{\min\{c_0,\,c^{\star}_0\}}{4\delta^{\star}_0}\,\,\,\,\,\,\,\,\textrm{and}\,\,\,\,\,\,\,\,
C_{\tau^{\ast}}>0\,\,\,\textrm{is a constant}.$$
Given $u\in \mathcal{W}_2(\Omega;\Gamma)$ a weak solution of (\ref{1.05}), $u$ solves the boundary value problem
\begin{equation}
\label{ENBVP2b}\left\{
\begin{array}{lcl}
\mathcal{A}u\,=\,\,\,\,\,f\,\,\,\,\,\,\,\,\,\,\,\,\,\,\,\,\,\,\,\,\,\textrm{in}\,\,\Omega\\
\mathcal{B}u\,=\,g-\zeta u\,\,\,\,\,\,\,\,\,\,\,\,\textrm{on}\,\,\Gamma\\
\end{array}
\right.
\end{equation}
Following as above together with the compactness of the embedding $\mathcal{W}_2(\Omega;\Gamma)\hookrightarrow L^2_{\mu}(\Gamma)$ and
Remark \ref{R-embedding}(d), we see that\\[2ex]
$\|u\|^2_{_{\mathcal{W}_2(\Omega;\Gamma)}}\,\leq\,\displaystyle\frac{2}{\min\{c_0,c^{\star}_0\}}
\left(\mathcal{E}_{\mu}(u,u)+\delta^{\star}_0\|u\|^2_{_{2,\Omega}}\right)$\\
$$\indent\indent\leq\,\displaystyle\frac{2}{\min\{c_0,c^{\star}_0\}}
\left(\mathcal{E}_{\mu}(u,u)+\delta^{\star}_0\left\{\tau^{\ast}
\|\nabla u\|_{_{2,\Omega}}+C_{\tau^{\ast}}\|u\|^2_{_{2,\Gamma}}\right\}\right).$$
Then, as in Case (a), using the above estimate, we get that
$$\|u\|^2_{_{\mathcal{W}_2(\Omega;\Gamma)}}\,\leq\,\displaystyle\frac{4}{\min\{c_0,c^{\star}_0\}}
\left(\displaystyle\int_{\Omega}fu\,dx+\displaystyle\int_{\Gamma}gu\,d\mu
-\left(\zeta-\delta^{\star}_0C_{\tau^{\ast}})\right)\|u\|^2_{_{2,\Gamma}}\right)\,
\leq\,\displaystyle\frac{4c''}{\min\{c_0,c^{\star}_0\}}|\|(f,g)\||_{_{p,q}}\|u\|_{_{\mathcal{W}_2(\Omega;\Gamma)}},$$\\
for some constant $c''>0$, and thus (\ref{coercive-infinity}) holds as well for this case.
\item[(c)]\,\,\,If $\xi\geq\delta^{\star}_0$ in part (a), or if $\zeta\geq\delta^{\star}_0C_{\tau^{\ast}}$ in part (b), then by replacing the space
$H^1(\Omega)$ by the space $H_0(\Omega)$ in the case condition {\bf (C1)} holds, then one gets the coercivity of the form
$\mathcal{E}_{\mu}(\cdot,\cdot)$, and in particular, this shows that under these conditions, problem (\ref{1.04}) possesses a unique
(modulo constant) weak solution $u\in H^1(\Omega)$. Same conclusion happens in the case when condition {\bf (C2)} is fulfilled, by replacing the
set $H^1_{\mu}(\Omega;\Gamma)$ by the closed subspace $H_{0,\mu}(\Omega;\Gamma)$.
\end{enumerate}
\end{remark}

\indent We conclude this section by stating our second main assumption.

\begin{assumption}\label{As2}
Let the conditions {\bf (A1)}, {\bf (A2)}, {\bf (A3)}, {\bf (B1)}, {\bf (B2)}, {\bf (B3)}, and {\bf (B4)} hold.
\end{assumption}

\subsection{Inverse positivity}\label{subsec03-03}

\indent In this part we present an inverse positivity property and dependence result for
for weak solutions of the boundary value problem (\ref{1.05}).\\
\indent Given the operators $\mathcal{A}$ and $\mathcal{B}$ defined by (\ref{1.01}) and (\ref{1.02}),
we are concerned with solutions for the differential inequality
\begin{equation}
\label{Diff-Ineq}\left\{
\begin{array}{lcl}
\mathcal{A}u\,\geq\,0\,\,\,\indent\textrm{in}\,\,\Omega\\
\mathcal{B}u\,\geq\,0\,\,\,\indent\textrm{on}\,\,\Gamma\\
\end{array}
\right.
\end{equation}
The above inequalities can be understood in the weak sense, that is, we consider solutions of the inequality
\begin{equation}
\label{form-Ineq}\mathcal{E}_{\mu}(u,\varphi)\,\geq\,0,\indent\,\,\forall\,\varphi\in\mathcal{W}_2(\Omega;\Gamma)^+,
\end{equation}
where the form $\mathcal{E}_{\mu}$ is given by (\ref{3.1.01}). The following is the main result of this subsection.\\

\begin{theorem}\label{inv-positivity}
Suppose that all the conditions of Assumption \ref{As1}  are fulfilled, suppose that
conditions {\bf (A2)} and {\bf (B3)} hold, and assume that any of the cases in Remark \ref{coercivity} occur.
If $u\in\mathcal{W}_2(\Omega;\Gamma)$ is a solution of (\ref{Diff-Ineq}), then
$\mathbf{u}\geq\mathbf{0}$; that is, $u\geq0$ a.e. in $\Omega$, and $u|_{_{\Gamma}}\geq0$ $\mu$-a.e. in $\Gamma$.\\
\end{theorem}

\begin{proof}
Let $u\in\mathcal{W}_2(\Omega;\Gamma)$ be a solution of (\ref{Diff-Ineq}), and assume that Remark \ref{coercivity}(a)
holds (the other case follows in a similar way). By {\bf (A2)} and {\bf (B3)} we have
$K(u^+,u^-)\leq0$, and $\Lambda_{_{\Gamma}}(u^+,u^-)\leq0$, and consequently $\mathcal{E}_{\mu}(u^+,u^-)\leq0$. Write
$$\mathfrak{C}^{\star}:=\left\{
\begin{array}{lcl}
\xi-\delta^{\star}_0,\,\,\,\,\,\,\,\,\,\,\,\,\,\,\,\,\,\textrm{if Remark \ref{coercivity}(a) (or Remark \ref{coercivity}(c) related to (a)) holds};\\
\zeta-\delta^{\star}_0C_{\tau^{\ast}},\,\,\,\,\,\,\textrm{if Remark \ref{coercivity}(b) (or Remark \ref{coercivity}(c) related to (b)) holds};\\
\end{array}
\right.$$
Then $\mathfrak{C}^{\star}>0$. Since $u\in\mathcal{W}_2(\Omega;\Gamma)$ satisfies the inequality
(\ref{form-Ineq}), proceeding as in Proposition \ref{cont-coercive-form} and Remark \ref{coercivity}, we obtain that
\begin{equation}
\label{non-neg}|\|\mathbf{u}^-\||^2_{_{2}}
\,\leq\,c\left(\mathcal{E}_{\mu}(u^-,u^-)-\mathfrak{C}^{\star}
\|u^-\|^2_{_{2,\Omega}}\right)\,\leq\,-c\,\mathcal{E}_{\mu}(u,u^-)\,\leq0,
\end{equation}
for some constant $c>0$. Thus $\mathbf{u}^-=\mathbf{0}$, which shows that $\mathbf{u}\geq\mathbf{0}$ a.e., as desired.
\end{proof}

\indent From here one has the following immediate consequence.

\begin{corollary}
Under the assumptions in Theorem \ref{inv-positivity}, let $u\in\mathcal{W}_2(\Omega;\Gamma)$ be a weak solution of problem (\ref{1.05}).
If $f\geq0$ a.e. in $\Omega$ and $g\geq0$ $\mu$-a.e. on $\Gamma$, then $\mathbf{u}\geq\mathbf{0}$.
\end{corollary}

\section{The parabolic problem}\label{sec04}

\indent In this section we turn our attention to the solvability and regularity theory for the parabolic problem (\ref{1.06}),
under the conditions in Assumption \ref{As1}. Also, as before, we will concentrate in the case $N>2$.

\subsection{The homogeneous problem}\label{subsec04-01}

\indent We begin by discussing the solvability of the time-dependent problem (\ref{1.06}) for $f=g=0$. To begin, recall the bilinear form
$\mathcal{E}_{\mu}(\cdot,\cdot)$ given by (\ref{3.1.01}).
Since the form  $\mathcal{E}_{\mu}$ is closed in $L^2(\Omega),$ there exists a unique operator $A_{\mu}$
with domain  $D(A_{\mu})\subseteq \mathcal{W}_2(\Omega;\Gamma)$ such that $D(A_{\mu})$ is dense in $L^2(\Omega)$, and
$A_{\mu}$ is the  generator of the form $\mathcal{E}_{\mu}(\cdot,\cdot)$ in the sense that
$$\mathcal{E}_{\mu}(u,v)=\langle A_{\mu}u,v\rangle_{_{L^2(\Omega)}},\,\,\,\,\,\,\,
\textrm{for each}\,\,u\in D(A_{\mu})\,\,\,\textrm{and for all}\,\, v\in \mathcal{W}_2(\Omega;\Gamma).$$
Now let $\{T_{\mu}(t)\}_{t\geq 0}$ denote the $C_0$-semigroup generated by $-A_{\mu}$ in $L^2(\Omega)$. By the theory of forms, in views of Proposition
\ref{cont-coercive-form}, it follows that $\{T_{\mu}(t)\}_{t\geq 0}$ is analytic  and compact (Cf. \cite[Chapter 17, Section 6]{DautrayLions}).
Consequently, for each $u_0\in L^2(\Omega)$, the function $u(t):=T_{\mu}(t)u_0$ is the unique mild solution of the Cauchy problem
\begin{equation}\label{parabolicpbomog}
\left\{
     \begin{array}{ll}
       u_t= A_{\mu}u\,\,\,\,\,\,\,\,\,\,\,\,\,\,\,\,\,\,\,\,\,\,\textrm{for}\,\,t\in(0,\infty);\\
       u(0)=u_0\,\,\,\,\,\,\,\,\,\,\,\,\,\,\,\,\,\,\,\,\,\,\textrm{in}\,\,\Omega
     \end{array}
   \right.
\end{equation}
(e.g. Definition \ref{classical sol}(b)), and problem (\ref{parabolicpbomog}) is precisely the abstract formulation of problem (\ref{1.06}) for $f=g=0$.\\
\indent To proceed, we first make an assumption over the Wentzell-type space $D(\Lambda_{_{\Gamma}})$, which is the following:
\begin{enumerate}
\item[{\bf (B5)}]\,\,\,For every bounded function $\phi\in C^\infty(\mathbb{R\!}^N)$ with bounded gradient, one has
$e^\phi u\in D(\Lambda_{_{\Gamma}})$ for each $u\in D(\Lambda_{_{\Gamma}})$.
\end{enumerate}
Then, given $\phi$ as in {\bf (B5)}, one clearly have that $e^\phi u\in\mathcal{W}_2(\Omega;\Gamma)$ whenever $u\in\mathcal{W}_2(\Omega;\Gamma)$.\\
\indent  Continuing with the discussion, by recalling (\ref{2.03}) and condition
{\bf (B5)}, one proceeds as in \cite[proof of Theorem 6.8]{OUHA05} to get the existence of a constant $c>0$ such that
\begin{equation}\label{ultracontractivity}
\|T_{\mu}(t)\|_{_{\mathcal{L}(L^2(\Omega);L^\infty(\Omega))}}\,\leq\, c t^{-N/4}e^{-\delta_0 t},\,\,\,\,\,\,\textrm{for all}\,\,t>0,
\end{equation}
where $\delta_0>$ is a constant depending on the constant $\delta^{\star}_0$ given by (\ref{delta-star}). In particular, the above inequality
shows that the semigroup $\{T_{\mu}(t)\}_{t\geq 0}$ is ultracontractive, that is, $T_{\mu}(t)$ maps $L^2(\Omega)$ into $L^{\infty}(\Omega)$.
In addition, by working with the adjoint semigroup  $\{T^{\ast}_{\mu}(t)\}_{t \geq 0}$ and performing a duality argument,
we can give a corresponding $L^1-L^2$ estimate norm for $\{T_{\mu}(t)\}_{t \geq 0}$ in the sense that
$$\|T_{\mu}(t)\|_{\mathcal{L}(L^1(\Omega);L^\infty(\Omega))}\leq c t^{-N/2}e^{-\delta_0 t},\,\,\,\,\,\,\textrm{for each}\,\, t>0.$$
Therefore, by Dunford Pettis's Theorem, the semigroup $\{T_{\mu}(t)\}_{t\geq 0}$
admits an integral representation with kernel $K_{\mu}(x,y,t)\in L^\infty(\Omega\times \Omega)$ (e.g. \cite[Theorem 1.3]{Arendt-Bukl94}).
By proceeding as in \cite[Theorem 6.1]{Daners}, taking into account {\bf (B5)}, and using perturbation techniques due to Davies \cite{DAV},
we can establish the existence of a constant $c>0$ such that
$$|K_{\mu}(t,x,y)|\,\leq\,c t^{-N/2} e^{\epsilon t+\delta_0\sigma t}e^{-\frac{|x-y|^2}{4\delta_0 t}},$$
for all $\epsilon>0$, where $\sigma>0$ denotes a constant depending on the coefficients in the definition of the constant $\delta^{\star}_0$ in (\ref{delta-star})
(Cf. \cite[Theorem 6.10]{OUHA05}).\\
\indent Finally, by the above inequality the kernel $K_{\mu}$ satisfies Gaussian estimates, and consequently in views of \cite{AR-ELST97},
it follows that $\{T_{\mu}(t)\}_{t\geq 0}$ extrapolates
to a family of compact analytic semigroups over $L^p(\Omega)$ for all $p\in[1,\infty]$, which are strongly continuous whenever $p\in[1,\infty)$,
and all have the same angle of analyticity. Furthermore, one has that $\{T_{\mu}(z)\}_{_{\textrm{Re}(z)>0}}$ are kernel operators satisfying Gaussian estimates.
In particular, for each $u_0\in L^p(\Omega)$ with $1\leq p<\infty$, the function $u(t):=T_{\mu}(t)u_0\in L^p(\Omega)$ is the unique mild solution
to problem (\ref{parabolicpbomog}) (and consequently the homogeneous parabolic problem (\ref{1.06}) in $L^p(\Omega)$).
Moreover, following the approach as in \cite[Proposition 7.1]{Daners}, we find the existence of constants
$M\geq 0$ and $\omega\in \mathbb{R\!}\,$ (depending on $N$, $|\Omega|$, and the coefficients in the definition of the constant $\delta^{\ast}$ in (\ref{3.2.10})),
fulfilling
\begin{equation}\label{infinity-semigroup}
\|u(t)\|_{_{\infty,\Omega}}\,\leq\, Me^{\omega t}\|u_0\|_{_{\infty,\Omega}}
\end{equation}
for each $u_0\in L^{\infty}(\Omega)$.

\subsection{Solvability of the inhomogeneous problem}\label{subsec04-02}

\indent We now turn our attention to the realization of the inhomogeneous boundary value problem (\ref{1.06}). Due to the fact that one is considering inhomogeneous
boundary conditions, the method relying in semigroup theory is not applicable anymore, so we will proceed with an analogous procedure as in \cite{NITTKA2014},
with some modifications in order to incorporate nonlocal operators and Wentzell-type maps.\\
\indent We start right away by introducing the notion of weak solutions to problem (\ref{1.06}).

\begin{definition}\label{weak-sol-parabolic}
 Let $T>0$ be a fixed real number. Then we say that a function $u\in C([0,T];L^2(\Omega))\cap L^2((0,T), \mathcal{W}_2(\Omega;\Gamma))$ is a \textbf{weak solution of the
 parabolic equation (\ref{1.06}) on $[0,T]$}, if\\[2ex]
\,\,\,\,\,\,$-\displaystyle\int_{0}^{T}\displaystyle\int_{\Omega} u(\xi)\psi_t(\xi)\, dx d\xi +\int_{0}^{T}\mathcal{E}_{\mu}(u(\xi),\psi(\xi))\, d\xi$\\
\begin{equation}
=\displaystyle\int_{\Omega} u_0 \psi(0) dx+ \int_{0}^{T}\int_{\Omega} f(\xi,x)\psi(\xi)\,dx d\xi +\displaystyle\int_{0}^{T}\int_{\Gamma} g(\xi,x) \psi(\xi) d\mu d\xi,
\end{equation}
 \noindent for all $\psi\in H^1(0,T;\mathcal{W}_2(\Omega;\Gamma))$ such that $\psi(T)=0$, where $u_0:=u(0)$ and where $\mathcal{E}_{\mu}(\cdot,\cdot)$
 denotes the bilinear form given by (\ref{3.1.01}).
 Similarly, we say that a function $u:[0, \infty)\to L^2(\Omega)$ is a weak solution of (\ref{1.06}) on $[0,\infty)$ if for every $T>0$
 its restriction to $[0,T]$ is a weak solution on $[0,T].$
 \end{definition}

\indent We now define an operator $\mathbf{A}$ acting on $\mathbb{X\!}^{\,2}(\Omega,\Gamma)$, which is directly related with the notion of weak solutions
of problem (\ref{1.05}), and will allow to transition from the inhomogeneous elliptic problem to the inhomogeneous parabolic problem.
As commented by Nittka \cite{NITTKA2014}, this operator $\mathbf{A}$ is not densely defined and hence it is not the
generator of a strongly continuous semigroup. In fact, it turns out that $\mathbf{A}$ does not even satisfy the Hille-Yosida estimates.
However, one will show that the operator is resolvent positive and
hence generates a once integrated semigroup, which will imply existence and uniqueness of solutions for regular data and gives information about the
asymptotic behavior of solutions.

\begin{definition}\label{a-normal-a}
Given $u\in \mathcal{W}_2(\Omega;\Gamma)$, let $\mathcal{A}$ and $\mathcal{B}$ be the linear operators defined by (\ref{1.01}) and (\ref{1.02}), respectively,
and define $$\mathcal{A}_0=\mathcal{A}-\mathcal{J}_{_{\Omega}}\,\,\,\,\,\,\,\,\textrm{and}\,\,\,\,\,\,\,\,
\mathcal{N}^{\ast}_{\mu}u:=\displaystyle\sum^N_{^{i,j=1}}(\alpha_{ij}(x)\partial_{x_j}u+\hat{a}_j(x)u)\nu_{\mu_i}.$$
\begin{enumerate}
\item[(a)]\,\, We say that $\mathcal{A}u\in L^{2}(\Omega)$ if there exists a $f\in L^{2}(\Omega)$ such that
$$\mathcal{E}_{_{\Omega}}(u,v)+(\mathcal{J}_{_{\Omega}}u)v=-\langle f,v\rangle_{2},\,\,\,\,\,\,\,\,\,\,\,\textrm{for all}\,\,v\in H^1_0(\Omega),$$
where $\mathcal{E}_{_{\Omega}}(\cdot,\cdot)$ is defined by (\ref{3.1.03}), and for $\langle\cdot,\cdot\rangle_{2}$ denoting the inner product on $L^{2}(\Omega)$.
\item[(b)]\,\, Assume that $\mathcal{A}_0u\in L^{2}(\Omega)$. Then we say that $\mathcal{N}^{\ast}_{\mu}u\in L^{2}_{\mu}(\Gamma)$,
if there exists a unique function $g\in L^{2}_{\mu}(\Gamma)$ such that
$$\langle-\mathcal{A}_0u, v\rangle_{2}+\langle g, v\rangle_{_{2,\Gamma}}=\mathcal{E}_{_{\Omega}}(u,v),\,\,\,\,\,\,\,\,\,\,\textrm{for every}\,\,v\in H^1(\Omega),$$
where $\langle\cdot,\cdot\rangle_{_{2,\Gamma}}$ denotes the inner product over $L^{2}_{\mu}(\Gamma)$.
\item[(c)]\,\, We say that $\mathcal{B}_0u:=\left(\mathcal{B}-\mathcal{N}^{\ast}_{\mu}\right)u\in L^{2}_{\mu}(\Gamma)$ if there exists a $g'\in L^{2}_{\mu}(\Gamma)$ such that
$$\mathcal{E}_{\mu}(u,v)=-\langle g',v\rangle_{_{2,\Gamma}},\,\,\,\,\,\,\,\,\,\,\,\textrm{for all}\,\,v\in \mathcal{W}_2(\Omega;\Gamma),$$
where $\mathcal{E}_{\mu}(\cdot,\cdot)$ is defined by (\ref{3.1.01}).
\item[(d)]\,\, We define the operator $\mathbf{A}:\mathbb{X\!}^{\,2}(\Omega,\Gamma)\rightarrow\mathbb{X\!}^{\,2}(\Omega,\Gamma)$ by:
$$\left\{
\begin{array}{lcl}
D(\mathbf{A}):=\left\{(u,0)\mid u \in \mathcal{W}_2(\Omega;\Gamma),\,\,\,\mathcal{A}u\in L^{2}(\Omega),\,\,\,\mathcal{B}u\in
L^2_{\mu}(\Gamma)\right\};\\[1ex]
\mathbf{A}(u,0):=(\mathcal{A}u,-\mathcal{B}u)=\left(\mathcal{A}u,\,-\mathcal{N}^{\ast}_{\mu}u-\mathcal{L}_{\mu}u-\beta u-\Theta_{_{\Gamma}}u\right).
\end{array}
\right.$$
 \end{enumerate}
 \end{definition}
\indent\\
 \begin{remark}\label{caratt- D-A2}
 In views of Definition \ref{a-normal-a}(d), it is easily verified that $(u,0)\in D(\mathbf{A})$ with  $-\mathbf{A}(u,0)=(f,g)$ if and only if
 $$\mathcal{E}_{\mu}(u,v) =\langle f,v\rangle_{2}+ \langle g,v\rangle_{_{2,\Gamma}}$$
 for every $v\in \mathcal{W}_2(\Omega;\Gamma)$. In particular if $(f,g)\in\mathbb{X\!}^{\,2}(\Omega,\Gamma)$, then
$(u,0)\in D(\mathbf{A})$ with  $-\mathbf{A}(u,0)=(f,g)$ if and only if $u\in\mathcal{W}_2(\Omega;\Gamma)$ is a weak solution of (\ref{1.05}).
 \end{remark}

 \indent For the rest of the section, we will assume that $\mathcal{L}_{\mu}\neq0$ (otherwise the proofs of the results become simpler).
 We begin by establishing some basic properties for the operator $\mathbf{A}$.\\

 \begin{prop}\label{A-operator-prop}
 The operator $\mathbf{A}$ is resolvent positive. More precisely, the operator
 $\lambda-\mathbf{A}$ on $D(\mathbf{A})$ is invertible for all $\gamma>\mathfrak{C}^{\star}_0$, where
 $$\mathfrak{C}^{\star}_0:=\left\{
\begin{array}{lcl}
\delta^{\star}_0,\,\,\,\,\,\,\,\,\,\,\,\,\,\,\,\,\,\,\textrm{if Remark \ref{coercivity}(a) holds},\\
\delta^{\star}_0C_{\tau^{\ast}},\,\,\,\,\,\,\,\,\textrm{if Remark \ref{coercivity}(b) holds}
\end{array}
\right.$$
(where we recall the definition of $\delta^{\star}_0>0$ given in (\ref{delta-star})),
 and if $\mathbf{A}(u,0)=(f,g)$ for $(f,g)\in\mathbb{X\!}^{\,2}(\Omega,\Gamma)$ are such that
 $f\geq0$ and $g\geq0$, then $u\geq 0$ a.e. Moreover,
if $D(\mathbf{A})$ is equipped with the graph norm, then $D(\mathbf{A})$ is continuously embedded into $\mathcal{W}_2(\Omega;\Gamma)\times \{0\}.$
 \end{prop}

 \begin{proof}
 If  $\gamma>\mathfrak{C}^{\ast}_0$, then by virtue of Remark \ref{coercivity} one sees that (\ref{3.1.06}) holds (for $\gamma^{\ast}=\gamma>0$),
 and consequently by Proposition \ref{cont-coercive-form} together with Lax-Milgram's Lemma,
for each $(f,g)\in\mathbb{X\!}^{\,2}(\Omega,\Gamma)$, there exists a unique function $u\in {\mathcal{W}_2(\Omega;\Gamma)}$ fulfilling
$$\gamma\int_{\Omega}uv\,dx+ \mathcal{E}_\mu (u,v)=\langle f,v\rangle_{2} + \langle g,v\rangle_{_{2,\Gamma}},
\,\,\,\,\,\,\,\,\,\,\,\,\,\,\,\textrm{for all}\,\,v\in \mathcal{W}_2(\Omega;\Gamma).$$
This together with Remark \ref{caratt- D-A2} we have that $(\gamma-\mathbf{A})(u,0)=(\gamma u, 0)-\mathbf{A}(u,0)=(f,g)$, which implies
that the operator $\gamma-\mathbf{A}: D(\mathbf{A})\rightarrow\mathbb{X\!}^{\,2}(\Omega,\Gamma)$ is a bijection for each
$\gamma>\mathfrak{C}^{\ast}_0$. Assume now $f\geq 0$ and $g\geq 0.$ Then this means that
$\langle f,v\rangle_{2}\geq0$ for each $v\in H^1(\Omega)^+$, and
$\langle g,w|_{_{\Gamma}}\rangle_{_{2,\Gamma}}\geq0$ for all $w\in H^1(\Omega)^+.$
Letting $(u,0):=(\gamma-\mathbf{A})^{-1}(f,g)$ and proceeding in the exact way as in the proof of Theorem \ref{inv-positivity}, we get that
$$\|u^-\|^2_{_{2^{\ast}_{_N},\Omega}}\,\leq\,c\left(\gamma\displaystyle\int_{\Omega}(u^-)^2\,dx+ \mathcal{E}_\mu (u^-,u^-)\right)\,
\leq\,-c\left(\langle f,u^-\rangle_{\mathfrak{s}/2} + \langle g,u^-\rangle_{_{D(\Lambda_{_{\Gamma}})^{\ast},D(\Lambda_{_{\Gamma}})}}\right)\,\leq\,0,$$
for some constant $c>0$.
Therefore $u^-=0$ a.e. implying that $u\geq 0$ a.e., as asserted. Hence the resolvent $(\gamma-\mathbf{A})^{-1}$ is a positive operator.
Since every positive operator is continuous we deduce that $\gamma-\mathbf{A}$ is invertible. Moreover, since the operator $\mathbf{A}$ is clearly closed,
we have that $D(\mathbf{A})$ is a Banach space with respect to the graph norm. Since the embedding $D(\mathbf{A})\hookrightarrow\mathbb{X\!}^{\,2}(\Omega,\Gamma)$
and the embedding $\mathcal{W}_2(\Omega;\Gamma)\times \{0\}\hookrightarrow\mathbb{X\!}^{\,2}(\Omega,\Gamma)$ are both bounded, with
$D(\mathbf{A})\subseteq\mathcal{W}_2(\Omega;\Gamma)\times \{0\}$, an application of the Closed Graph Theorem implies that the embedding
$D(\mathbf{A})\hookrightarrow\mathcal{W}_2(\Omega;\Gamma)\times \{0\}$ is continuous, completing the proof.
 \end{proof}

\indent We now introduce the notion of mild and classical solution of (\ref{1.06}), according to \cite{Lunardi}.
Form now on we assume that $f\in L^2(0,T; L^2(\Omega))$ and $g\in L^2(0,T,L^2_{\mu}(\Gamma))$.\\

 \begin{definition}\label{classical sol}
Let $I=[0,T]$ or $I=[0,\infty]$.
\begin{enumerate}
\item[(a)]\,\, We say that function $u$ is a \textbf{classical solution  of problem (\ref{1.06})}, if
$u\in C^1((0,T], L^2(\Omega))$ with  $u(0)=u_0,$  the mapping $t\to (u(t),0)$ lies in $C(I,D(\mathbf{A}))$, and
$$(u_t(\xi),0)- \mathbf{A}(u(\xi),0)=(f(\xi,x), g(\xi,x)),\,\,\,\,\,\,\,\,\textrm{for all}\,\,\xi\in I.$$
\item[(b)]\,\,We call $u$ is a \textbf{mild solution of problem (\ref{1.06})},
if $u\in  C(I, L^2(\Omega)))$ with  $(\int_{0}^{t}u(\xi)\,d\xi, 0) \in D(\mathbf{A})$ for all $t\geq 0$, and
\begin{center}
$(u(t)-u_0, 0)-\mathbf{A}(\int_{0}^{t}u(\xi)\,d\xi, 0)=(\int_{0}^{t} f(\xi,x)\,d\xi,\int_{0}^{t} g(\xi,x)\,d\xi)$.
\end{center}
\end{enumerate}
 \end{definition}

\indent The next two results can be established in the exact way as in \cite[proof of Theorem 2.6 and Proposition 2.7]{NITTKA2014}, and thus
their proofs will be omitted.

\begin{theorem}\label{clas-stron-weak-mild}
Let $I=[0,T]$ or $I=[0,\infty).$ Then every classical solution of problem (\ref{1.06}) on $I$ is a weak solution of problem (\ref{1.06}) on $I$, and
every weak solution of equation (\ref{1.06}) on $I$ is a mild solution of the heat equation (\ref{1.06}) on $I.$
\end{theorem}

\begin{prop}\label{reg-sol-class}
Let $u_0\in L^2(\Omega)$,\, $f\in L^2(0,T; L^2(\Omega))$ and $g\in L^2(0,T,L^2_{\mu}(\Gamma)),$ for some $T>0$.
Then problem  (\ref{1.06}) has a unique mild solution. Furthermore, if $(u_0,0)\in D(\mathbf{A}),$\,  $\mathbf{A}(u_0,0)+(f(0,x),g(0,x)) \in D(\mathbf{A})$,\,
$f\in C^2([0,T], L^2(\Omega))$, and $g\in C^2([0,T], L^2_{\mu}(\Gamma))$, then equation (\ref{1.06}) has a classical solution.
Moreover, if $u_0\geq 0$,\, $f(t,x)\geq 0$, and $g(t,x)\geq 0$ almost everywhere for each $t\in (0,T)$, then if $u$ is a mild solution  of
problem (\ref{1.06}), then $u(t)\geq 0$ almost everywhere for every $t\in (0,T)$.
\end{prop}

\indent In order to prove that there exists a unique weak solution for given data $u_0,\,f,\,g$ square
summable we need to prove some a priori bounds for the classical $L^2$-solutions.

\begin{prop}\label{apriori estimates}
Let $u$ be a classical solution of (\ref{1.06}) on $[0,T],$ for some $T>0$. Then\\[2ex]
$\displaystyle\sup_{t\in [0,T]}\displaystyle\int_{\Omega} |u(t)|^2 dx +\displaystyle\int_{0}^{T}\|\nabla u(\xi)\|^2_{_{2,\Omega}}\,d\xi+
\displaystyle\int^T_{0}\|u(\xi)\|^2_{_{D(\Lambda_{_{\Gamma}})}}\,d\xi$\\
\begin{equation}\label{aux-estimate}
\leq\,c\left(\|u_0\|^2_{_{2,\Omega}}+\displaystyle\int_{0}^{T}\|f(\xi)\|^2_{_{2,\Omega}}d\xi+\displaystyle\int_{0}^{T}\|g(\xi)\|^2_{_{2,\Gamma}}d\xi\right),
\end{equation}
where $c$ is a positive constant depending on $T,\,|\Omega|,\,\gamma,\,\bar{\alpha}.$
\end{prop}

\begin{proof}
The proof runs in a similar manner as in \cite[proof of Lemma 2.9]{NITTKA2014}. However, since there are some modifications, we provide a complete proof.
Given $t\in [0,T]$, we apply Young's inequality together with (\ref{2.04}) and (\ref{3.1.06}) to calculate that\\[2ex]
$\displaystyle\frac{1}{2}\left(\displaystyle\int_{\Omega} |u(t)|^2\,dx-\displaystyle\int_{\Omega}|u_0|^2\,dx\right) =
\displaystyle\int_{0}^{t}\displaystyle\int_{\Omega} u(\xi)u_t(\xi)\,dx d\xi=$\\[2ex]
$=\displaystyle\int_{0}^{t}\langle\mathcal{A}u(\xi),u(\xi)\rangle_{\mathfrak{s}/2}\,d\xi+\displaystyle\int_{0}^{t}
\displaystyle\int_{\Omega}f(\xi,x)u(\xi)\,dxd\xi=\displaystyle\int_{0}^{t}\langle\mathcal{N}^{\ast}_{\mu}u(\xi),u(\xi)\rangle_{r_d/2}\,d\xi+$\\
$$\,\,\,\,\,\,\,\,\,\,\,\,\,\,+\displaystyle\int_{0}^{t}\displaystyle\int_{\Omega}f(\xi,x)u(\xi)\,dxd\xi
-\int_{0}^{t}\left\{\mathcal{E}_{_{\Omega}}(u(\xi),u(\xi))+\left(\mathcal{J}_{_{\Omega}}u(\xi)\right)u(\xi)\right\}\,d\xi$$\\
$=\displaystyle\int_{0}^{t}\displaystyle\int_{\Omega}f(\xi,x)u(\xi)\,dxd\xi+\displaystyle\int_{0}^{t}
\displaystyle\int_{\Gamma}g(\xi,x)u(\xi)\,d\mu d\xi-\displaystyle\int_{0}^{t}\mathcal{E}_{\mu}(u(\xi),u(\xi))\, d\xi$\\[2ex]
$\leq\,\displaystyle\frac{1}{2}\displaystyle\int_{0}^{t} \|f(\xi)\|^2_{_{2,\Omega}}\,d\xi+\displaystyle\frac{1}{4\epsilon}
\displaystyle\int_{0}^{t}\|g(\xi)\|^2_{_{2,\Gamma}}\,d\xi$\\
$$-(\eta-\epsilon c_2^2)\displaystyle\int_{0}^{t}\left\{\|\nabla u(\xi)\|^2_{_{2,\Omega}}+\|u(\xi)\|^2_{_{D(\Lambda_{_{\Gamma}})}}\right\}d\xi+
\left(\gamma +\frac{1}{2} +\epsilon c_2^2\right)\displaystyle\int_{0}^{t}\|u(\xi)\|^2_{_{2,\Omega}}\,d\xi.$$
By choosing $\epsilon:=\eta/2c_2^2$, and letting $c':=2\min\{\eta,1\}^{-1}>0$, we deduce that\\[2ex]
$\displaystyle\sup_{\xi\in [0,t]}\displaystyle\int_{\Omega}|u(\xi)|^2 dx+\displaystyle\int_{0}^{t}\|\nabla u(\xi)\|^2_{_{2,\Omega}}\,d\xi+
\displaystyle\int^t_{0}\|u(\xi)\|^2_{_{D(\Lambda_{_{\Gamma}})}}\,d\xi$\\
$$\leq\,\displaystyle\frac{c'}{2}\left(\displaystyle\int_{\Omega}|u_0|^2 dx+\displaystyle\int_{0}^{t}\displaystyle\int_{\Omega}|f(\xi)|^2\,dxd\xi+
\displaystyle\frac{c^2_2}{\eta}\displaystyle\int_{0}^{t}\int_{\Gamma}|g(\xi)|^2\,dxd\xi\,
+\,t(2\gamma+1+\eta)\displaystyle\sup_{\xi\in [0,t]}\int_\Omega |u(\xi)|^2\,dx\right).$$
Since the above inequality holds for all $t\in[0,T]$, by choosing $t^{\ast}:=(2c'\gamma+c'+c'\eta)^{-1}>0$ and putting $c'':=c'\max\{1,c^2_2\eta^{-1}\}>0$,
we obtain that\\[2ex]
$$\displaystyle\sup_{\xi\in [0,t]}\displaystyle\int_{\Omega} |u(\xi)|^2 dx +\displaystyle\int_{0}^{t}\|\nabla u(\xi)\|^2_{_{2,\Omega}}\,d\xi+
\displaystyle\int^t_{0}\|u(\xi)\|^2_{_{D(\Lambda_{_{\Gamma}})}}\,d\xi
\,\leq\,c''\left(\|u_0\|^2_{_{2,\Omega}}+\displaystyle\int_{0}^{t}\|f(\xi)\|^2_{_{2,\Omega}}d\xi+\displaystyle\int_{0}^{t}\|g(\xi)\|^2_{_{2,\Gamma}}d\xi\right),$$
for all $t\in[0,t^{\ast}]$. To complete the proof, given $T>0$ arbitrary fixed, we can write $[0,T]:=\displaystyle\bigcup^m_{j=1}[y_{j-1},y_j]$ for some $m\in\mathbb{N\!}\,$,
where $\displaystyle\max_{1\leq j\leq m}\{y_j-y_{j-1}\}\leq t^{\ast}$, $y_0=0$, and $y_m=T$. Applying the previous inequality in each of the subintervals
$[y_{j-1},y_j]$, we are lead into the fulfillment of (\ref{aux-estimate}), as desired.
\end{proof}

\begin{remark}\label{approx-classical}
By suitable approximation by classical solutions and Proposition \ref{reg-sol-class} (e.g. \cite[proof of Theorem 2.11]{NITTKA2014})
we can deduce that Proposition \ref{apriori estimates} is valid for weak solutions $u\in C(0,T;L^2(\Omega))\cap L^2(0,T;\mathcal{W}_2(\Omega;\Gamma))$
of problem (\ref{1.06}).
\end{remark}

\indent From here, we are ready to establish the existence and uniqueness of a weak solution for problem (\ref{1.06}). Again their proofs mimic
the procedures in \cite[proof of Theorem 2.11]{NITTKA2014}, and consequently they will be omitted.

\begin{theorem}\label{exist-uniq-weak-soln}
Given $T>0$ arbitrary fixed, let $u_0\in L^2(\Omega)$,\, $f\in L^2(0,T; L^2(\Omega))$ and $g\in L^2(0,T,L^2_{\mu}(\Gamma))$.
Then there exists a weak solution of (\ref{1.06}) in $[0,T]$ which is unique, even in the class of mild solutions.
\end{theorem}

\begin{corollary}\label{coincidence}
For problem (\ref{1.06}), the notions of weak and mild solution coincide.
\end{corollary}

\subsection{Time-dependent a priori estimates}\label{subsec04-03}

\indent In this part we develop $L^{\infty}$-estimates for weak solutions of the parabolic problem under minimal assumptions. As always, we will assume
that Assumption \ref{As1} holds.\\
\indent To begin, following the same approach as in the end of subsection \ref{subsec03-02},
given $u\in L^{\infty}(-T,0;L^2(\Omega))\cap L^2(-T,0;H^1(\Omega))$ and $k\geq0$ a fixed real number, we put
\begin{equation}\label{2.16}
u_k(t):=(u(t)-k)^+.
\end{equation}
Clearly $u_k(t)\in L^{\infty}(-T,0;L^2(\Omega))\cap L^2(-T,0;H^1(\Omega))$,
with $\partial_{x_i}u_k(t)=(\partial_{x_i}u(t))\chi_{_{\Omega_k(t)}}$ for each $i\in\{1,\ldots,N\}$, where
\begin{equation}\label{2.17}
E_k(t):=\{x\in\overline{\Omega}\mid u(t)\geq k\}\,\,\,\,\,\,\,\,\textrm{and}\,\,\,\,\,\,\,\,D_k(t):=D\cap E_k(t),
\end{equation}
for each $D\subseteq\mathbb{R}^N$ such that $D\cap\overline{\Omega}\neq\emptyset$. Then, we have the following assumption:\\
\begin{enumerate}
\item[{\bf (A3)'}]\,\,\,$K(u(t),u_k(t))\geq0$ for each $u(t)\in L^{\infty}(-T,0;L^2(\Omega))\cap L^2(-T,0;H^1(\Omega))$;
\item[{\bf (B4)'}]\,\,\,For each $u(t)\in L^{\infty}(-T,0;L^2(\Omega))\cap L^2(-T,0;D(\Lambda_{_{\Gamma}}))$ one has
$u_k(t)\in L^{\infty}(-T,0;L^2(\Omega))\cap L^2(-T,0;D(\Lambda_{_{\Gamma}}))$ and\\[2ex]
$\Lambda_{_{\Gamma}}(u(t),u_k(t))\,\geq\,c^{\star}_1\Lambda_{_{\Gamma}}(u_k(t),u_k(t))+$\\
$$+c^{\star}_2\|u_k(\xi)\|^2_{_{D(\Lambda_{_{\Gamma}})}}-c^{\star}_3\displaystyle\int_{\Gamma_k(t)}\rho(x)(|u_k(t)|^2+k^2)\,d\mu,$$
for some constants $c^{\star}_1,\,c^{\star}_2,\,c^{\star}_3\geq0$, and for some function $\rho\in L^{\theta}_{\mu}(\Gamma)$,
where $\theta\geq1$ is such that the map
$\mathcal{W}_2(\Omega;\Gamma)\hookrightarrow L^{^{\frac{2\theta}{\theta-1}}}_{\mu}(\Gamma)$ is compact.\\
\end{enumerate}
In views of {\bf (B4)'}, for each $u(t)\in L^{\infty}(-T,0;L^2(\Omega))\cap L^2(-T,0;\mathcal{W}_2(\Omega;\Gamma))$,
we have $u_k(t)\in L^{\infty}(-T,0;L^2(\Omega))\cap L^2(-T,0;\mathcal{W}_2(\Omega;\Gamma))$. Then for convenience, given $\varsigma>0$, we write
\begin{equation}\label{Special-Norm}
\Upsilon^2_{\varsigma}(u):=\displaystyle\sup_{-\varsigma\leq t\leq0}\displaystyle\int_{\Omega}|u(t)|^2\,dx+
\displaystyle\int^0_{-\varsigma}\displaystyle\int_{\Omega}|\nabla u(t)|^2\,dxdt+\displaystyle\int^0_{-\varsigma}\|u(t)\|^2_{_{D(\Lambda_{_{\Gamma}})}}\,dt.
\end{equation}
By virtue of (\ref{2.03}) and (\ref{2.04}) (Cf. \cite[II.3]{LAD-URAL68} in the case of $\Omega$ bounded Lipschitz), we get that
\begin{equation}\label{Sobolev-parabolic}
\|u\|_{_{L^{r_1}(-\varsigma,0;L^{s_1}(\Omega))}}+\|u\|_{_{L^{r_2}(-\varsigma,0;L^{s_2}_{\mu}(\Gamma))}}\,\leq\,C_0\,\Upsilon^2_{\varsigma}(u)
\end{equation}
for all $u(t)\in L^{\infty}(-T,0;L^2(\Omega))\cap L^2(-T,0;\mathcal{W}_2(\Omega;\Gamma))$ and for some constant $C_0>0$, where
$r_{1}\in[2,\infty)$, \,$s_{1}\in[2,2N(N-2)^{-1}]$,\, $r_{2}\in[2,\infty)$, \,$s_{2}\in[2,2d(N-2)^{-1}]$, are such that
$$\displaystyle\frac{1}{r_{1}}+\displaystyle\frac{N}{2s_{1}}=\displaystyle\frac{N}{4}\,\,\,\,\,\textrm{and}
\,\,\,\,\,\displaystyle\frac{1}{r_{2}}+\displaystyle\frac{d}{2s_{2}}=\displaystyle\frac{N}{4}.$$
\indent\\
\begin{lemma}\label{Lem1}
Given $T>0$ fixed, let $u\in L^{\infty}(-T,0;L^2(\Omega))\cap L^2(-T,0;\mathcal{W}_2(\Omega;\Gamma))$, and assume that
$u_k(t)\in L^{\infty}(-T,0;L^2(\Omega))\cap L^2(-T,0;D(\Lambda_{_{\Gamma}}))$. Given $K_1,\,K_2\geq1$ integers and
$\xi\in\{1,\ldots,K_1\}$, and $\zeta\in\{1,\ldots,K_2\}$, let $r_{1,\xi}\in[2,\infty)$, \,$s_{1,\xi}\in[2,2N(N-2)^{-1}]$,\,
$r_{2,\zeta}\in[2,\infty)$, \,$s_{2,\zeta}\in[2,2d(N-2)^{-1}]$, such that
\begin{equation}\label{Lem1-01}
\displaystyle\frac{1}{r_{1,\xi}}+\displaystyle\frac{N}{2s_{1,\xi}}=\displaystyle\frac{N}{4}\,\,\,\,\,\textrm{and}
\,\,\,\,\,\displaystyle\frac{1}{r_{2,\zeta}}+\displaystyle\frac{d}{2s_{2,\zeta}}=\displaystyle\frac{N}{4}.
\end{equation}
Assume that there exists constants $\hat{k}_0,\,\gamma_0\geq0$ and positive numbers $\theta_{1,\xi},\,\theta_{2,\zeta}$, such that
for all $\varsigma\in(0,T]$,\, $\varepsilon\in(0,1/2)$, and $k\geq\hat{k}_0$, we have\\[2ex]
$\Upsilon^2_{(1-\varepsilon)\varsigma}(u_k)\,\leq\,\displaystyle\frac{\gamma_0}{\varepsilon\varsigma}
\displaystyle\int^0_{-\varsigma}\displaystyle\int_{\Omega}|u_k(t)|^2\,dx\,dt+$\\
\begin{equation}\label{Lem1-02}
+\gamma_0k^2\displaystyle\sum^{K_1}_{\xi=1}\left(\displaystyle\int^0_{-\varsigma}|\Omega_k(t)|
^{^{\frac{r_{1,\xi}}{s_{1,\xi}}}}\,dt\right)^{^{\frac{2(1+\theta_{1,\xi})}{r_{1,\xi}}}}
+\gamma_0k^2\displaystyle\sum^{K_2}_{\zeta=1}\left(\displaystyle\int^0_{-\varsigma}\mu(\Gamma_k(t))
^{^{\frac{r_{2,\zeta}}{s_{2,\zeta}}}}\,dt\right)^{^{\frac{2(1+\theta_{2,\zeta})}{r_{2,\zeta}}}}.
\end{equation}
Then there exists a constant $C^{\ast}_0>0$ (independent of $u$ and $\hat{k}_0$) such that
\begin{equation}\label{Lem1-03}
\textrm{ess}\displaystyle\sup_{\left[-\frac{T}{2},0\right]\times\overline{\Omega}}[u]
\,\leq\,C^{\ast}_0\left(\|u\|^2_{_{L^{2}(-T,0;L^{2}(\Omega))}}+\hat{k}^2_0\right)^{1/2}.
\end{equation}
Moreover, if instead of (\ref{Lem1-02}) one has\\[2ex]
$\Upsilon^2_{_T}(u_k)\,\leq\,\gamma_0\displaystyle\int^0_{-T}\displaystyle\int_{\Omega}|u_k(t)|^2\,dx\,dt+$\\
\begin{equation}\label{Lem1-04}
+\gamma_0k^2\displaystyle\sum^{K_1}_{\xi=1}\left(\displaystyle\int^0_{-T}|\Omega_k(t)|
^{^{\frac{r_{1,\xi}}{s_{1,\xi}}}}\,dt\right)^{^{\frac{2(1+\theta_{1,\xi})}{r_{1,\xi}}}}
+\gamma_0k^2\displaystyle\sum^{K_2}_{\zeta=1}\left(\displaystyle\int^0_{-T}\mu(\Gamma_k(t))
^{^{\frac{r_{2,\zeta}}{s_{2,\zeta}}}}\,dt\right)^{^{\frac{2(1+\theta_{2,\zeta})}{r_{2,\zeta}}}},
\end{equation}
Then there exists a constant $C^{\star}>0$ (independent of $u$ and $k_0$) such that
\begin{equation}\label{Lem1-05}
\textrm{ess}\displaystyle\sup_{\left[-T,0\right]\times\overline{\Omega}}[u]
\,\leq\,C^{\ast}_0\left(\|u\|^2_{_{L^{2}(-T,0;L^{2}(\Omega))}}+\hat{k}^2_0\right)^{1/2}.
\end{equation}
\end{lemma}

\begin{proof}
The proof is a direct modification of \cite[Theorem A.2 and Corollary A.3]{NITTKA2014} to the case of irregular domains. However, for the sake of completeness we provide a complete proof here.
We begin by observing that since $|\Omega_k(t)|\leq|\Omega|$ and $\mu(\Gamma_k(t))\leq\mu(\Gamma)$, it suffices to consider
$$\theta_{1,\xi}=\theta_{2,\zeta}=\theta:=\min\left\{\min\{\theta_{1,\xi}\mid1\leq\xi\leq K_1\}\cup\min\{\theta_{2,\zeta}\mid1\leq\zeta\leq K_2\}\right\}$$
for each $\xi\in\{1,\ldots,K_1\}$ and $\zeta\in\{1,\ldots,K_2\}$, provided that $\gamma_0$ is replaced by a larger constant $\gamma'_0$
(depending on $r_{1,\xi}$,\, $s_{1,\xi}$,\, $\theta_{1,\xi}$,\, $r_{2,\zeta}$,\, $s_{2,\zeta}$,\, $\theta_{2,\zeta}$,\, $T$,\, $\gamma_0$,\, $|\Omega|$,
$\mu(\Gamma)$), since by performing this selection, both inequalities (\ref{Lem1-02}) and (\ref{Lem1-04}) still hold. Then, given $\hat{k}>\hat{k}_0\geq0$ fixed,
define the sequences $\{\varsigma_n\}_{n\geq0}$,\, $\{k_n\}_{n\geq0}$,\, $\{y_n\}_{n\geq0}$,\, $\{z_n\}_{n\geq0}$ of positive real numbers, by:
\begin{equation}\label{Lem1-06}
\varsigma_n:=\displaystyle\frac{T}{2}(1+2^{-(n+1)}),\,\,\,\,\,\,k_n:=(2-2^{-n})\hat{k},\,\,\,\,\,\,
y_n:=\displaystyle\frac{1}{\hat{k}^2}\displaystyle\int^0_{-\varsigma_n}\displaystyle\int_{\Omega}|u_{k_n}(t)|^2\,dx\,dt,
\end{equation}
and
\begin{equation}\label{Lem1-07}
z_n:=\displaystyle\sum^{K_1}_{\xi=1}\left(\displaystyle\int^0_{-\varsigma_n}|\Omega_{k_n}(t)|
^{^{\frac{r_{1,\xi}}{s_{1,\xi}}}}\,dt\right)^{^{\frac{2}{r_{1,\xi}}}}
+\displaystyle\sum^{K_2}_{\zeta=1}\left(\displaystyle\int^0_{-\varsigma_n}\mu(\Gamma_{k_n}(t))
^{^{\frac{r_{2,\zeta}}{s_{2,\zeta}}}}\,dt\right)^{^{\frac{2}{r_{2,\zeta}}}}.
\end{equation}
Clearly $\varsigma_n\in[T/2,T]$ and $k_n\geq\hat{k}_0$ for each nonnegative integer $n$. Moreover, it is clear that
$$|u_{k_n}(t)|\geq(k_{n+1}-k_n)\chi_{_{\overline{\Omega}_{k_{n+1}}(t)}},$$ so taking into account
the definitions of the sequences defined above together with H\"older's inequality and (\ref{Sobolev-parabolic}), we calculate and get that\\[2ex]
$\hat{k}^2y_{n+1}\,\leq\,\left(\displaystyle\int^0_{-\varsigma_{n+1}}|\Omega_{k_{n+1}}(t)|\,dt\right)^{1/N}
\left(\displaystyle\int^0_{-\varsigma_{n+1}}\left\{\displaystyle\int_{\Omega}\left|u_{k_{n+1}}(t)\right|^{^{\frac{2N}{N-1}}}\,dx\right\}^{^{\frac{N-1}{N}}}dt\right)$\\
\begin{equation}\label{Lem1-08}
\leq\,c\left(\displaystyle\int^0_{-\varsigma_{n+1}}|\Omega_{k_{n+1}}(t)|\,dt\right)^{1/N}\Upsilon^2_{\varsigma_{n+1}}(u_{k_{n+1}})
\,\leq\,c\left\{(k_{n+1}-k_n)^{-2}\hat{k}^2y_n\right\}^{1/N}\Upsilon^2_{\varsigma_{n+1}}(u_{k_{n+1}})
\,\leq\,c\,4^{n+1}y_n^{1/N}\Upsilon^2_{\varsigma_{n+1}}(u_{k_{n+1}}),
\end{equation}
where $c>0$ is a constant that varies from line to line (this will be assumed for the remaining of the proof). In the same way, we get that\\[2ex]
$4^{-(n+1)}\hat{k}^2z_{n+1}\,=\,(k_{n+1}-k_n)^2z_{n+1}$\\
\begin{equation}\label{Lem1-09}\leq\,\displaystyle\sum^{K_1}_{\xi=1}\left(\displaystyle\int^0_{-\varsigma_{n+1}}\left\{\displaystyle\int_{\Omega}|u_{k_n}(t)|^{s_{1,\xi}}\,dx\right\}
^{^{\frac{r_{1,\xi}}{s_{1,\xi}}}}\,dt\right)^{^{\frac{2}{r_{1,\xi}}}}
+\displaystyle\sum^{K_2}_{\zeta=1}\left(\displaystyle\int^0_{-\varsigma_{n+1}}\left\{\displaystyle\int_{\Gamma}|u_{k_n}(t)|^{s_{2,\zeta}}\,d\mu\right\}
^{^{\frac{r_{2,\zeta}}{s_{2,\zeta}}}}\,dt\right)^{^{\frac{2}{r_{2,\zeta}}}}\,\leq\,c\,\Upsilon^2_{\varsigma_{n+1}}(u_{k_{n}}).
\end{equation}
Now, taking $\varsigma=\varsigma_n$ and $\varepsilon=1-\varsigma_{n+1}/\varsigma_n\geq2^{-(n+3)}$, we apply the assumption (\ref{Lem1-02})
to obtain that
\begin{equation}\label{Lem1-10}
\Upsilon^2_{\varsigma_{n+1}}(u_{k_{n+1}})\,\leq\,\Upsilon^2_{\varsigma_{n+1}}(u_{k_{n}})\,\leq\,
\frac{\gamma_0}{\varepsilon\varsigma_n}\hat{k}^2y_n+\gamma_0k^2_nz^{1+\theta}_n
\,\leq\,\gamma_02^{n+4}\hat{k}^2(T^{-1}+1)(y_n+z^{1+\theta}_n),
\end{equation}
and then combining (\ref{Lem1-08}), (\ref{Lem1-09}), and (\ref{Lem1-10}), we arrive at
\begin{equation}\label{Lem1-11}
y_{n+1}\,\leq\,c'\,8^n(y^{1+\delta}_n+z^{1+\theta}_ny^{\delta}_n)\,\,\textrm{and}\,\,
z_{n+1}\,\leq\,c'\,8^n(y_n+z^{1+\theta}_n),\,\,\textrm{for each integer}\,\,n\geq0,
\end{equation}
for some constant $c'>0$, where $\delta:=1/N$. To complete the arguments, we seek estimations for $y_0$ and $z_0$. In fact, we easy see that
\begin{equation}\label{Lem1-12}
y_0=\displaystyle\frac{1}{\hat{k}^2}\displaystyle\int^0_{-\frac{3T}{4}}\displaystyle\int_{\Omega}|u_{k_0}(t)|^2\,dx\,dt\,\leq\,
\displaystyle\frac{1}{\hat{k}^2}\displaystyle\int^0_{-T}\displaystyle\int_{\Omega}|u(t)|^2\,dx\,dt.
\end{equation}
On the other hand, using (\ref{Lem1-02}) and following the approach as in (\ref{Lem1-08}) and (\ref{Lem1-09}), we have that\\[2ex]
$(\hat{k}-\hat{k}_0)^2z_0\,\leq\,\displaystyle\sum^{K_1}_{\xi=1}\left(\displaystyle\int^0_{-\varsigma_{0}}\left\{\displaystyle\int_{\Omega}|u_{k_0}(t)|^{s_{1,\xi}}\,dx\right\}
^{^{\frac{r_{1,\xi}}{s_{1,\xi}}}}\,dt\right)^{^{\frac{2}{r_{1,\xi}}}}+
\displaystyle\sum^{K_2}_{\zeta=1}\left(\displaystyle\int^0_{-\varsigma_{0}}\left\{\displaystyle\int_{\Gamma}|u_{k_0}(t)|^{s_{2,\zeta}}\,d\mu\right\}
^{^{\frac{r_{2,\zeta}}{s_{2,\zeta}}}}\,dt\right)^{^{\frac{2}{r_{2,\zeta}}}}$\\
$$\leq\,c\,\Upsilon^2_{\varsigma_{0}}(u_{k_{0}})
\leq\,\displaystyle\frac{4c\gamma_0}{T}\displaystyle\int^0_{-T}\displaystyle\int_{\Omega}|u_{k_0}(t)|^2\,dx\,dt+
c\gamma_0k^2_0\displaystyle\sum^{K_1}_{\xi=1}\left[T|\Omega|^{^{\frac{r_{1,\xi}}{s_{1,\xi}}}}\right]^{^{\frac{2(1+\theta)}{r_{1,\xi}}}}
+\,c\gamma_0k^2_0\displaystyle\sum^{K_2}_{\zeta=1}\left[T\mu(\Gamma)^{^{\frac{r_{2,\zeta}}{s_{2,\zeta}}}}\right]^{^{\frac{2(1+\theta)}{r_{2,\zeta}}}}.$$
Thus letting
$$c'':=c\gamma_0\max\left\{\frac{4}{T}\,,\,\displaystyle\sum^{K_1}_{\xi=1}\left[T|\Omega|^{^{\frac{r_{1,\xi}}{s_{1,\xi}}}}\right]^{^{\frac{2(1+\theta)}{r_{1,\xi}}}}
+\displaystyle\sum^{K_2}_{\zeta=1}\left[T\mu(\Gamma)^{^{\frac{r_{2,\zeta}}{s_{2,\zeta}}}}\right]^{^{\frac{2(1+\theta)}{r_{2,\zeta}}}}\right\},$$
we arrive at
\begin{equation}\label{Lem1-13}
z_0\,\leq\,\displaystyle\frac{c''}{(\hat{k}-\hat{k}_0)^2}\left(\displaystyle\int^0_{-T}\displaystyle\int_{\Omega}|u_{k_0}(t)|^2\,dx\,dt+k^2_0\right),
\,\,\,\,\,\,\,\,\,\,\,\,\textrm{for all}\,\,\,\hat{k}\geq k_0.
\end{equation}
Then, we define $$d:=\min\left\{\delta,\,\frac{\theta}{1+\theta}\right\}\,\,\,\,\,\textrm{and}\,\,\,\,\,
\eta:=\min\left\{\frac{1}{(2c')^{^{\frac{1}{\delta}}}8^{^{\frac{1}{\delta d}}}},\,\frac{1}{(2c')^{^{\frac{1+\theta}{\theta}}}\,8^{^{\frac{1}{\theta d}}}}\right\},$$
and then choose
\begin{equation}\label{Lem1-14}
\hat{k}:=\max\left\{\displaystyle\frac{1}{\sqrt{\eta}}\left(\displaystyle\int^0_{-T}\displaystyle\int_{\Omega}|u(t)|^2\,dx\,dt\right)^{1/2},\,
\displaystyle\frac{\sqrt{c_1}}{\eta^{^{\frac{1}{2(1+\theta)}}}}
\left(\displaystyle\int^0_{-T}\displaystyle\int_{\Omega}|u(t)|^2\,dx\,dt+\hat{k}^2_0\right)^{1/2}+\hat{k}_0\right\}.
\end{equation}
Then one sees that
\begin{equation}\label{Lem1-15}
\hat{k}\,\leq\,c^{\star}\left(\displaystyle\int^0_{-T}\displaystyle\int_{\Omega}|u(t)|^2\,dx\,dt+\hat{k}^2_0\right)^{1/2}
\end{equation}
for some constant $c^{\star}>0$, and moreover the selection implies that
\begin{equation}\label{Lem1-16}
y_0\,\leq\,\eta\,\,\,\,\,\,\,\,\textrm{and}\,\,\,\,\,\,\,\,z_0\,\leq\,\eta^{^{\frac{1}{1+\theta}}}.
\end{equation}
In views of (\ref{Lem1-11}) and (\ref{Lem1-16}), we see that the sequences $\{y_n\}$,\, $\{z_n\}$ satisfy the conditions of Lemma \ref{lemma2}.
Thus, applying Lemma \ref{lemma2} for $\hat{k}$ given by (\ref{Lem1-14})
shows that $\displaystyle\lim_{n\rightarrow\infty}z_n=0$, and consequently
\begin{equation}\label{Lem1-17}
u(t)\,\leq\,\displaystyle\lim_{n\rightarrow\infty}k_n=2\hat{k}\,\,\,\,\textrm{a.e. in}\,\,\overline{\Omega},\,\,\,\,\,\textrm{for a.e.}\,\,
t\in\displaystyle\bigcap_{n\in\mathbb{N\!}}[-\varsigma_n,0]=[-T/2,0].
\end{equation}
Combining (\ref{Lem1-17}) with (\ref{Lem1-15}), we get that
\begin{equation}\label{Lem1-18}
\displaystyle\textrm{ess}\hspace{-0.70cm}\sup_{(t,x)\in\left[-\frac{T}{2},0\right]\times\Omega}\{u(t,x)\}+
\displaystyle\textrm{ess}\hspace{-0.60cm}\sup_{(t,x)\in\left[-\frac{T}{2},0\right]\times\Gamma}\{u(t,x)\}
\,\leq\,C^{\ast}\left(\displaystyle\int^0_{-T}\displaystyle\int_{\Omega}
|u(t)|^2\,dx\,dt+\hat{k}_0\right)^{1/2}.
\end{equation}
This gives (\ref{Lem1-03}). To complete the proof,
assume now that (\ref{Lem1-04}) is valid. Then by selecting $\varsigma_n:=T$ in the previous arguments, the proof runs in the exact way as in the previous case.
In particular, one has that (\ref{Lem1-17}) is valid for almost all $t\in\displaystyle\bigcap_{n\in\mathbb{N\!}}[-\varsigma_n,0]=[-T,0]$, leading to
the inequality
\begin{equation}\label{Lem1-19}
\displaystyle\textrm{ess}\hspace{-0.63cm}\sup_{(t,x)\in\left[-T,0\right]\times\Omega}\{u(t,x)\}+
\displaystyle\textrm{ess}\hspace{-0.60cm}\sup_{(t,x)\in\left[-T,0\right]\times\Gamma}\{u(t,x)\}
\,\leq\,C^{\star}\left(\displaystyle\int^T_{0}\displaystyle\int_{\Omega}
|u(t)|^2\,dx\,dt+\hat{k}_0\right)^{1/2}.
\end{equation}
This completes the proof.
\end{proof}

\indent We now establish a key result concerning the global regularity (in space) of weak solutions of problem (\ref{1.06}).
In time, we also derive local estimates. The proof is motivated by the procedures in \cite{NITTKA2014}, but as our problem contains
multiple generalizations and additional nonlocal maps and multiple boundary conditions, complete proofs will be provided in order to
deal with the corresponding modifications.

\begin{theorem}\label{Thm-Partial-Bounded}
Assume that conditions {\bf (A3)'} and {\bf (B4)'} hold.
Given $T>0$ fixed, let $\kappa_1,\,\kappa_2,\,p,\,q\in[2,\infty)$ be such that
\begin{equation}\label{Thm1-00}
\displaystyle\frac{1}{\kappa_1}+\displaystyle\frac{N}{2p}<1\,\,\,\,\,\,\,\,\textrm{and}
\,\,\,\,\,\,\,\,\displaystyle\frac{1}{\kappa_2}+\displaystyle\frac{d}{2q(d+2-N)}<\frac{1}{2}.
\end{equation}
Given $u_0\in L^2(\Omega)$, let $f\in L^{\kappa_1}(0,T;L^p(\Omega))$ and $g\in L^{\kappa_2}(0,T;L^q_{\mu}(\Gamma))$, and
let $u\in C(0,T;L^2(\Omega))\cap L^2(0,T;\mathcal{W}_2(\Omega;\Gamma))$ be a weak solution of problem (\ref{1.06}).
Then there exists a constant $C^{\ast}_1=C^{\ast}_1(T,N,d,\kappa_1,\kappa_2,p,q,|\Omega|,\mu(\Gamma))>0$ such that
\begin{equation}\label{Thm1-01}
|\|\mathbf{u}\||_{_{L^{\infty}(T/2,T;\mathbb{X\!}^{\,\infty}(\Omega;\Gamma))}}
\,\leq\,C^{\ast}_1\left(\|u\|_{_{L^{2}(0,T;L^{2}(\Omega))}}+\|f\|_{_{L^{\kappa_1}(0,T;L^{p}(\Omega))}}+
\|g\|_{_{L^{\kappa_2}(0,T;L^{q}_{\mu}(\Gamma))}}\right).
\end{equation}
If in addition $u_0=0$, then we get the global estimate
\begin{equation}\label{Thm1-02}
|\|\mathbf{u}\||_{_{L^{\infty}(0,T;\mathbb{X\!}^{\,\infty}(\Omega;\Gamma))}}
\,\leq\,C^{\ast}_1\left(\|u\|_{_{L^{2}(0,T;L^{2}(\Omega))}}+\|f\|_{_{L^{\kappa_1}(0,T;L^{p}(\Omega))}}+
\|g\|_{_{L^{\kappa_2}(0,T;L^{q}_{\mu}(\Gamma))}}\right).
\end{equation}
\end{theorem}

\begin{proof}
Let $u\in C(0,T;L^2(\Omega))\cap L^2(0,T;\mathcal{W}_2(\Omega;\Gamma))$ be a weak solution of problem (\ref{1.06}).
The proof of this result is technical and long, and will be divided into three main steps.\\
$\bullet$\,\, {\it \underline{Step 1}}. Assume in addition that $u\in C(0,T;L^2(\Omega))\cap L^2(0,T;\mathcal{W}_2(\Omega;\Gamma))$
is a classical solution of (\ref{1.06}), and that $T\leq T_0$ for some constant $T_0>0$ (depending only in $N$, $d$, $\kappa_1$,
$\kappa_2$, $p$, $q$, $|\Omega|$, $\mu(\Gamma)$, and the coefficients of the operators $\mathcal{A}$ and $\mathcal{B}$). By a linear
translation in time, we may assume that problem (\ref{1.06}) is solved over $[-T,0]$, with initial value $u_0=u(-T)$.
Let $u_k(t)$ be the function given by (\ref{2.16}), fix $\varsigma\in[-T,0]$, and select a function $\psi\in H^1(-\varsigma,0)$
such that $0\leq\psi(t)\leq1$ for all $t\in[-\varsigma,0]$. Assume in addition that either $\psi(-\varsigma)=0$, or $\varsigma=T$ and
$u_k(-T)=0$. Then for each $t\in[-\varsigma,0]$, observe that\\[2ex]
$\displaystyle\frac{\psi(t)^2}{2}\displaystyle\int_{\Omega_k(t)}|u_k(t)|^2\,dx=\displaystyle\int^t_{-\varsigma}
\displaystyle\frac{d}{d\xi}\left(\displaystyle\frac{\psi(\xi)^2}{2}\displaystyle\int_{\Omega_k(\xi)}|u_k(\xi)|^2\,dx\right)d\xi$\\
\begin{equation}\label{thm1-02}
=\displaystyle\int^t_{-\varsigma}\psi(\xi)\psi'(\xi)\left(\displaystyle\int_{\Omega_k(\xi)}|u_k(\xi)|^2\,dx\right)d\xi+
\displaystyle\int^t_{-\varsigma}\psi(\xi)^2\left(\displaystyle\int_{\Omega_k(\xi)}u_k(\xi)\frac{ du_k(\xi)}{d\xi}\,dx\right)d\xi.
\end{equation}
Now, using the fact that $u\in C(0,T;L^2(\Omega))\cap L^2(0,T;\mathcal{W}_2(\Omega;\Gamma))$
is a classical solution of (\ref{1.06}), we see that
\begin{equation}\label{thm1-03}
\displaystyle\int_{\Omega_k(\xi)}u_k(\xi)\frac{ du_k(\xi)}{d\xi}\,dx=\displaystyle\int_{\Omega_k(\xi)}u_k(\xi)\frac{ du(\xi)}{d\xi}\,dx
=\displaystyle\int_{\Omega_k(\xi)}f(\xi,x)u_k(\xi)\,dx+\displaystyle\int_{\Gamma_k(\xi)}g(\xi,x)u_k(\xi)\,d\mu-\mathcal{E}_{\mu}(u(\xi),u_k(\xi)),
\end{equation}
where we recall the definition of the form $\mathcal{E}_{\mu}(\cdot,\cdot)$ in (\ref{3.1.01}). Estimating the right hand side in (\ref{thm1-03}) and taking
into account Assumptions {\bf (A3)'} and {\bf (B4)'}, we see that\\[2ex]
$\mathcal{E}_{\mu}(u(\xi),u_k(\xi))\,\geq\,\mathcal{E}_{_{\Omega}}(u_k(\xi),u_k(\xi))+\min\{1,c^{\star}_1\}\mathcal{E}_{_{\Gamma}}(u_k(\xi),u_k(\xi))+$\\
$$+\displaystyle\int_{\Omega_k(\xi)}\displaystyle\sum^N_{i=1}\hat{a}_ik\partial_{x_i}u_k(\xi)\,dx+\displaystyle\int_{\Omega_k(\xi)}\lambda ku_k(\xi)\,dx
+\min\{1,c^{\star}_1\}\displaystyle\int_{\Gamma_k(\xi)}\beta ku_k(\xi)\,d\mu+$$
\begin{equation}\label{thm1-04}
+c^{\star}_2\|u_k(\xi)\|^2_{_{D(\Lambda_{_{\Gamma}})}}-c^{\star}_3\displaystyle\int_{\Gamma_k(t)}\rho(x)(|u_k(t)|^2+k^2)\,d\mu,
\end{equation}
where the constants $c^{\star}_m$ ($m\in\{1,2,3\}$) and the function $\rho$ are described in condition {\bf (B4)'}.
Recalling (\ref{1.04}), we apply Young's inequality multiple times and use Remark \ref{R-embedding}(d) to produce the following estimates.\\[2ex]
$c_0\|\nabla u_k(\xi)\|^2_{_{2,\Omega_k(\xi)}}\,\leq\,\mathcal{E}_{_{\Omega}}(u_k(\xi),u_k(\xi))+\|\hat{\mathfrak{a}}\|_{_{r_1,\Omega}}
\|u_k(\xi)\|_{_{\frac{2r_1}{r_1-2},\Omega_k(\xi)}}\|\nabla u_k(\xi)\|_{_{2,\Omega_k(\xi)}}+$\\
$$\,\,\,\,\,\,\,\,\,\,\,\,\,\,\,\,\,\,\,\,+\|\check{\mathfrak{a}}\|_{_{r_2,\Omega}}
\|u_k(\xi)\|_{_{\frac{2r_2}{r_2-2},\Omega_k(\xi)}}\|\nabla u_k(\xi)\|_{_{2,\Omega_k(\xi)}}+
\displaystyle\int_{\Omega_k(\xi)}|\lambda||u_k(\xi)|^2\,dx$$
$$\leq\,\mathcal{E}_{_{\Omega}}(u_k(\xi),u_k(\xi))+\|\hat{\mathfrak{a}}\|_{_{r_1,\Omega}}\left(2\epsilon_1\|\nabla u_k(\xi)\|^2_{_{2,\Omega_k(\xi)}}+
\frac{C_{\epsilon^2_1}}{\epsilon_1}\|u_k(\xi)\|^2_{_{2,\Omega_k(\xi)}}\right)+\,\,\,\,\,\,\,\,\,\,\,\,\,\,\,\,\,\,\,$$
\begin{equation}\label{thm1-05}
+\|\check{\mathfrak{a}}\|_{_{r_2,\Omega}}\left(2\epsilon_1\|\nabla u_k(\xi)\|^2_{_{2,\Omega_k(\xi)}}+
\frac{C'_{\epsilon^2_1}}{\epsilon_1}\|u_k(\xi)\|^2_{_{2,\Omega_k(\xi)}}\right)+\displaystyle\int_{\Omega_k(\xi)}|\lambda||u_k(\xi)|^2\,dx,
\end{equation}
for all $\epsilon_1>0$, for some constant $C_{\epsilon^2_1},\,C'_{\epsilon^2_1}>0$. Writing $C_1:=C_{\epsilon^2_1}$,\, $C_2=C'_{\epsilon^2_1}$,
and $$M^{\star}:=\max\left\{\|\hat{\mathfrak{a}}\|_{_{r_1,\Omega}},\|\check{\mathfrak{a}}\|_{_{r_2,\Omega}}\right\}
\left(C_1\|\hat{\mathfrak{a}}\|_{_{r_1,\Omega}}+C_2\|\check{\mathfrak{a}}\|_{_{r_2,\Omega}}\right)/c_0>0,$$
we insert $\epsilon_1:=c_0\left(8\max\left\{\|\hat{\mathfrak{a}}\|_{_{r_1,\Omega}},\|\check{\mathfrak{a}}\|_{_{r_2,\Omega}}\right\}\right)^{-1}$
into (\ref{thm1-05}) to get that
\begin{equation}\label{thm1-06}
\mathcal{E}_{_{\Omega}}(u_k(\xi),u_k(\xi))\,\geq\,\displaystyle\frac{c_0}{2}\|\nabla u_k(\xi)\|^2_{_{2,\Omega_k(\xi)}}-
\displaystyle\int_{\Omega_k(\xi)}(M^{\star}+|\lambda|)|u_k(\xi)|^2\,dx.
\end{equation}
Also, recalling that the ``Wentzell-type" boundary form $(\Lambda_{_{\Gamma}},D(\Lambda_{_{\Gamma}}))$ is assumed to be weakly coercive,
we get the existence of constants $c^{\ast}_1>0$ and $c^{\ast}_2\in\mathbb{R\!}\,$ such that
\begin{equation}\label{thm1-07}
\mathcal{E}_{_{\Gamma}}(u_k(\xi),u_k(\xi))\,\geq\,c^{\ast}_1\|u_k(\xi)\|^2_{_{D(\Lambda_{_{\Gamma}})}}-
\displaystyle\int_{\Gamma_k(\xi)}(|c^{\ast}_2|+|\beta|)|u_k(\xi)|^2\,d\mu.
\end{equation}
It is also easy to see that
\begin{equation}\label{thm1-08}
\displaystyle\int_{\Omega_k(\xi)}\displaystyle\sum^N_{i=1}\hat{a}_ik\partial_{x_i}u_k(\xi)\,dx\,\geq\,-\displaystyle
\frac{k^2}{c_0}\displaystyle\int_{\Omega_k(\xi)}\displaystyle\sum^N_{i=1}|\hat{a}_i|^2\,dx-\displaystyle
\frac{c_0}{4}\|\nabla u_k(\xi)\|^2_{_{2,\Omega_k(\xi)}},
\end{equation}
and\\[2ex]
\indent$\displaystyle\int_{\Omega_k(\xi)}f(\xi,x)u_k(\xi)\,dx+\displaystyle\int_{\Gamma_k(\xi)}g(\xi,x)u_k(\xi)\,d\mu$\\
\begin{equation}\label{thm1-09}
\leq\,\displaystyle\frac{1}{k}\displaystyle\int_{\Omega_k(\xi)}|f(\xi,x)|(|u_k(\xi)|^2+k^2)\,dx+
\displaystyle\frac{1}{k}\displaystyle\int_{\Gamma_k(\xi)}|g(\xi,x)|(|u_k(\xi)|^2+k^2)\,d\mu.
\end{equation}
Letting $$\mathfrak{h}:=\frac{1}{c_0}\displaystyle\sum^N_{i=1}|\hat{a}_i|^2+M^{\star}+|\lambda|\in L^{^{\min\left\{\frac{r_1}{2},r_3\right\}}}(\Omega)^+$$
and $$\mathfrak{k}:=\min\{1,c^{\star}_1\}(|c^{\ast}_2|+2|\beta|)+c^{\star}_3|\rho|\in L^{^{\min\{s,\theta\}}}_{\mu}(\Gamma)^+,$$
we insert (\ref{thm1-06}), (\ref{thm1-07}), and (\ref{thm1-08}) into (\ref{thm1-04}), to deduce that\\[2ex]
$\mathcal{E}_{\mu}(u(\xi),u_k(\xi))\,\geq\,\displaystyle\frac{c_0}{4}\|\nabla u_k(\xi)\|^2_{_{2,\Omega_k(\xi)}}+
c^{\ast\star}\|u_k(\xi)\|^2_{_{D(\Lambda_{_{\Gamma}})}}$\\
\begin{equation}\label{thm1-10}
-\displaystyle\int_{\Omega_k(\xi)}\mathfrak{h}(|u_k(\xi)|^2+k^2)\,dx-\displaystyle\int_{\Gamma_k(\xi)}\mathfrak{k}(|u_k(\xi)|^2+k^2)\,d\mu,
\end{equation}
where $c^{\ast\star}:=c^{\ast}_1\min\{1,c^{\star}_1\}+c^{\star}_2$.
Substituting this into (\ref{thm1-03}) and recalling (\ref{thm1-09}) gives that\\[2ex]
$\displaystyle\int_{\Omega_k(\xi)}u_k(\xi)\frac{ du_k(\xi)}{d\xi}\,dx\,\leq\,-\displaystyle\frac{c_0}{4}\|\nabla u_k(\xi)\|^2_{_{2,\Omega_k(\xi)}}
-c^{\ast\star}\|u_k(\xi)\|^2_{_{D(\Lambda_{_{\Gamma}})}}$\\
\begin{equation}\label{thm1-11}
+\displaystyle\int_{\Omega_k(\xi)}\left(\frac{1}{k}|f(\xi,x)|+\mathfrak{h}\right)(|u_k(\xi)|^2+k^2)\,dx+
\displaystyle\int_{\Gamma_k(\xi)}\left(\frac{1}{k}|g(\xi,x)|+\mathfrak{k}\right)(|u_k(\xi)|^2+k^2)\,d\mu.
\end{equation}
From here, we insert (\ref{thm1-11}) into (\ref{thm1-02}) to get that\\[2ex]
$\min\left\{\displaystyle\frac{1}{2},\displaystyle\frac{c_0}{4},c^{\ast\star}\right\}\Upsilon^2_{\varsigma}(\psi u_k)
\,\leq\,\displaystyle\sup_{-\varsigma\leq t\leq0}\left(\displaystyle\frac{\psi(t)^2}{2}\|u_k(t)\|^2_{_{2,\Omega_k(t)}}\right)+
\frac{c_0}{4}\displaystyle\int^0_{-\varsigma}\psi(t)^2\|\nabla u_k(t)\|^2_{_{2,\Omega_k(t)}}\,dt\,+\,c^{\ast\star}
\displaystyle\int^0_{-\varsigma}\psi(t)^2\|u_k(t)\|^2_{_{D(\Lambda_{_{\Gamma}})}}\,dt$\\
$$\leq\,L_{\psi'}\displaystyle\int^0_{-\varsigma}\|u_k(t)\|^2_{_{2,\Omega_k(t)}}\,dt
+\displaystyle\int^0_{-\varsigma}\displaystyle\int_{\Omega_k(t)}\left(\frac{1}{k}|f(t,x)|+\mathfrak{h}\right)(\psi(t)^2|u_k(t)|^2+k^2)\,dx\,dt+\,\,\,\,\,$$
\begin{equation}\label{thm1-12}
\indent\indent\indent\indent
+\displaystyle\int^0_{-\varsigma}\displaystyle\int_{\Gamma_k(t)}\left(\frac{1}{k}|g(t,x)|+\mathfrak{k}\right)(\psi(t)^2|u_k(t)|^2+k^2)\,d\mu\,dt,
\end{equation}
where $L_{\psi'}:=\displaystyle\sup_{-\varsigma\leq t\leq0}|\psi'(t)|$. We now seek to estimate the last two terms in (\ref{thm1-12}). For this,
we choose $$\theta_1:=\frac{4\tilde{p}+2\tilde{\kappa}_1N-\tilde{p}\tilde{\kappa}_1N}{\tilde{p}\tilde{\kappa}_1N}>0\,\,\,\,\,
\textrm{and}\,\,\,\,\,\theta_2:=\frac{4\tilde{q}+2\tilde{\kappa}_2d-\tilde{q}\tilde{\kappa}_2N}{\tilde{q}\tilde{\kappa}_2N}>0,$$
where $\tilde{\tau}:=2\tau(\tau-1)^{-1}$ for $\tau\in(1,\infty)$. Then
$$\frac{1}{\tilde{\kappa}_1(1+\theta_1)}+\frac{N}{2\tilde{p}(1+\theta_1)}=\frac{1}{\tilde{\kappa}_2(1+\theta_2)}+\frac{d}{2\tilde{q}(1+\theta_2)}=\frac{N}{4}.$$
Put
\begin{equation}\label{thm1-13}
\hat{k}^2_0:=\|f\|^2_{_{L^{\kappa_1}(0,T;L^{p}(\Omega))}}+\|g\|^2_{_{L^{\kappa_2}(0,T;L^{q}_{\mu}(\Gamma))}},
\end{equation}
and apply H\"older's inequality to deduce that
$$\displaystyle\int^0_{-\varsigma}\displaystyle\int_{\Omega_k(t)}\frac{1}{k}|f(t,x)|\psi(t)^2|u_k(t)|^2\,dx\,dt
\,\leq\,\displaystyle\frac{1}{k}\|f\|_{_{L^{\kappa_1}(-\varsigma,0;L^{p}(\Omega))}}\|\psi u_k\|^2_{_{L^{\tilde{\kappa}_1}(-\varsigma,0;L^{\tilde{p}}(\Omega))}}$$
\begin{equation}\label{thm1-14}
\leq\,\displaystyle\frac{\hat{k}_0}{k}\left(\displaystyle\int^0_{-\varsigma}|\Omega_k(t)|^{\tilde{\kappa}_1/\tilde{p}}dt\right)^{^{\frac{2\kappa_1\theta_1}{1+\theta_1}}}
\|\psi u_k\|^2_{_{L^{\tilde{\kappa}_1(1+\theta_1)}(-\varsigma,0;L^{\tilde{p}(1+\theta_1)}(\Omega))}}.
\end{equation}
Now observe that $$\displaystyle\lim_{\varsigma\rightarrow0^+}\left\{\left(\displaystyle\int^0_{-\varsigma}|\Omega_k(t)|^{\tilde{\kappa}_1/\tilde{p}}dt\right)
^{^{\frac{2\kappa_1\theta_1}{1+\theta_1}}}\|\psi u_k\|^2_{_{L^{\tilde{\kappa}_1(1+\theta_1)}(-\varsigma,0;L^{\tilde{p}(1+\theta_1)}(\Omega))}}\right\}=0,$$
so using this together with the selection of $\theta_1$ and (\ref{Sobolev-parabolic}) entail the existence of a constant
$T_1=T_1(N,\kappa_1,\theta_1,p,|\Omega|)>0$ such that
\begin{equation}\label{thm1-15}
\displaystyle\int^0_{-\varsigma}\displaystyle\int_{\Omega_k(t)}\frac{1}{k}|f(t,x)|\psi(t)^2|u_k(t)|^2\,dx\,dt
\,\leq\,\displaystyle\frac{\hat{k}_0}{8k}\min\left\{\displaystyle\frac{1}{2},\displaystyle\frac{c_0}{4},c^{\ast\star}\right\}\Upsilon^2_{\varsigma}(\psi u_k)
\end{equation}
whenever $\varsigma\leq T_1$. Also, since $\mathfrak{h}\in L^{^{\min\left\{\frac{r_1}{2},r_3\right\}}}(\Omega)$ and
$$\mathfrak{r}:=2\min\{r_1/2,\,r_3\}\left(\min\{r_1/2,\,r_3\}-1\right)^{-1}<
2N(N-2)^{-1},$$ proceeding in the same way as in the derivation of (\ref{thm1-15}), we get the existence of a constant
$T_2=T_2(N,\min\{r_1/2,\,r_3\},\mathfrak{h},|\Omega|)>0$ fulfilling
\begin{equation}\label{thm1-16}
\displaystyle\int^0_{-\varsigma}\displaystyle\int_{\Omega_k(t)}\mathfrak{h}\psi(t)^2|u_k(t)|^2\,dx\,dt\,
\leq\,\|\mathfrak{h}\|_{_{\min\left\{\frac{r_1}{2},r_3\right\},\Omega}}\|\psi u_k\|^2_{_{L^{2}(-\varsigma,0;L^{\mathfrak{r}}(\Omega))}}
\,\leq\,\displaystyle\frac{1}{8}\min\left\{\displaystyle\frac{1}{2},\displaystyle\frac{c_0}{4},c^{\ast\star}\right\}\Upsilon^2_{\varsigma}(\psi u_k)
\end{equation}
for $\varsigma\leq T_2$. Moreover, recalling the selection of $\theta_2$, and using the fact that $\mathfrak{k}\in L^{^{\min\{s,\theta\}}}_{\mu}(\Gamma)$ with
$$\mathfrak{s}:=2\min\{s,\theta\}\left(\min\{s,\theta\}-1\right)^{-1}<2d(N-2)^{-1},$$ we proceed in the exact way as above to deduce the existence
of positive constants $T_3=T_3(N,d,\kappa_2,\theta_2,q,\mu(\Gamma))$ and $T_4=T_4(N,d,\min\{s,\theta\},\mathfrak{k},\mu(\Gamma))$ such that
\begin{equation}\label{thm1-17}
\displaystyle\int^0_{-\varsigma}\displaystyle\int_{\Gamma_k(t)}\frac{1}{k}|g(t,x)|\psi(t)^2|u_k(t)|^2\,d\mu\,dt
\,\leq\,\displaystyle\frac{\hat{k}_0}{8k}\min\left\{\displaystyle\frac{1}{2},\displaystyle\frac{c_0}{4},c^{\ast\star}\right\}\Upsilon^2_{\varsigma}(\psi u_k),
\end{equation}
if $\varsigma\leq T_3$, and
\begin{equation}\label{thm1-18}
\displaystyle\int^0_{-\varsigma}\displaystyle\int_{\Gamma_k(t)}\mathfrak{k}\psi(t)^2|u_k(t)|^2\,d\mu\,dt\,
\leq\,\|\mathfrak{k}\|_{_{\min\{s,\theta\},\Gamma}}\|\psi u_k\|^2_{_{L^{2}(-\varsigma,0;L^{\mathfrak{s}}_{\mu}(\Gamma))}}
\,\leq\,\displaystyle\frac{1}{8}\min\left\{\displaystyle\frac{1}{2},\displaystyle\frac{c_0}{4},c^{\ast\star}\right\}\Upsilon^2_{\varsigma}(\psi u_k),
\end{equation}
whenever $\varsigma\leq T_4$. Letting $T_0:=\min\{T_1,T_2,T_3,T_4\}$, and taking $\varsigma\in(0,T_0]$ and $k\geq\hat{k}_0$, we insert
(\ref{thm1-15}), (\ref{thm1-16}), (\ref{thm1-17}), and (\ref{thm1-18}) into (\ref{thm1-12}), to arrive at\\[2ex]
$\Upsilon^2_{\varsigma}(\psi u_k)\,\leq\,\displaystyle\frac{2}{\min\{1/2,c_0/4,c^{\ast\star}\}}\left\{
L_{\psi'}\displaystyle\int^0_{-\varsigma}\|u_k(t)\|^2_{_{2,\Omega_k(t)}}\,dt\right.+$\\
\begin{equation}\label{thm1-19}
+\left.k^2\displaystyle\int^0_{-\varsigma}\displaystyle\int_{\Omega_k(t)}\left(\frac{1}{k}|f(t,x)|+\mathfrak{h}\right)\,dx\,dt\,+\,
k^2\displaystyle\int^0_{-\varsigma}\displaystyle\int_{\Gamma_k(t)}\left(\frac{1}{k}|g(t,x)|+\mathfrak{k}\right)\,d\mu\,dt\right\}.
\end{equation}
Having this, we now produce estimates in the last two integral terms in (\ref{thm1-19}). In fact,
recalling that $k\geq\hat{k}_0$ and applying H\"older's inequality, we see that
\begin{equation}\label{thm1-20}
\displaystyle\int^0_{-\varsigma}\displaystyle\int_{\Omega_k(t)}\frac{1}{k}|f(t,x)|\,dx\,dt\,\leq\,\displaystyle\frac{1}{k}
\|f\|_{_{L^{\kappa_1}(-\varsigma,0;L^{p}(\Omega))}}\left(\displaystyle\int^0_{-\varsigma}|\Omega_k(t)|^{^{\frac{\kappa_1(p-1)}{p(\kappa_1-1)}}}dt\right)
^{^{\frac{\kappa_1-1}{\kappa_1}}}
\,\leq\,\left(\displaystyle\int^0_{-\varsigma}|\Omega_k(t)|^{^{\frac{\kappa_1(p-1)}{p(\kappa_1-1)}}}dt\right)^{^{\frac{\kappa_1-1}{\kappa_1}}}.
\end{equation}
Proceeding in the exact way, we find that
\begin{equation}\label{thm1-21}
\displaystyle\int^0_{-\varsigma}\displaystyle\int_{\Omega_k(t)}\mathfrak{h}\,dx\,dt\,\leq\,\|\mathfrak{h}\|_{_{\min\left\{\frac{r_1}{2},r_3\right\},\Omega}}
\displaystyle\int^0_{-\varsigma}|\Omega_k(t)|^{2/\mathfrak{r}}dt,
\end{equation}
and
\begin{equation}\label{thm1-22}
\displaystyle\int^0_{-\varsigma}\displaystyle\int_{\Gamma_k(t)}\frac{1}{k}|g(t,x)|\,d\mu\,dt\,\leq\,
\left(\displaystyle\int^0_{-\varsigma}\mu(\Gamma_k(t))^{^{\frac{\kappa_2(q-1)}{q(\kappa_2-1)}}}dt\right)
^{^{\frac{\kappa_2-1}{\kappa_2}}},
\end{equation}
and
\begin{equation}\label{thm1-23}
\displaystyle\int^0_{-\varsigma}\displaystyle\int_{\Gamma_k(t)}\mathfrak{k}\,d\mu\,dt\,\leq\,\|\mathfrak{k}\|_{_{\min\{s,\theta\},\Gamma}}
\displaystyle\int^0_{-\varsigma}\mu(\Gamma_k(t))^{2/\mathfrak{s}}dt.
\end{equation}
Put $$\mathfrak{c}_1:=\displaystyle\frac{2}{\min\{1/2,c_0/4,c^{\ast\star}\}}\,\,\,\textrm{and}\,\,\,
\mathfrak{c}_2:=\displaystyle\frac{2\max\left\{1,\|\mathfrak{h}\|_{_{\min\{r_1,r_3\},\Omega}},
\|\mathfrak{k}\|_{_{\min\{s,\theta\},\Gamma}}\right\}}{\min\{1/2,c_0/4,c^{\ast\star}\}},$$
and we choose
$$\theta_{1,1}:=\frac{4\tilde{p}+2\tilde{\kappa}_1N-\tilde{p}\tilde{\kappa}_1N}{\tilde{p}\tilde{\kappa}_1N},\,\,\,\,\,\,\,
r_{1,1}:=(1+\theta_{1,1})\tilde{\kappa}_1,\,\,\,\,\,\,\,s_{1,1}:=(1+\theta_{1,1})\tilde{p},$$
$$\theta_{1,2}:=\frac{2N+2\mathfrak{r}-N\mathfrak{r}}{N\mathfrak{r}},\,\,\,\,\,\,\,
r_{1,2}:=2(1+\theta_{1,2}),\,\,\,\,\,\,\,s_{1,2}:=(1+\theta_{1,2})\mathfrak{r},$$
$$\theta_{2,1}:=\frac{4\tilde{q}+2\tilde{\kappa}_2d-\tilde{q}\tilde{\kappa}_2N}{\tilde{q}\tilde{\kappa}_2N},\,\,\,\,\,\,\,
r_{2,1}:=(1+\theta_{2,1})\tilde{\kappa}_2,\,\,\,\,\,\,\,s_{2,1}:=(1+\theta_{2,1})\tilde{q},$$
$$\theta_{2,2}:=\frac{2d+2\mathfrak{s}-N\mathfrak{s}}{N\mathfrak{s}},\,\,\,\,\,\,\,
r_{2,2}:=2(1+\theta_{2,2}),\,\,\,\,\,\,\,s_{2,2}:=(1+\theta_{2,2})\mathfrak{s}.$$
A straightforward calculation shows that the parameters $\theta_{1,\xi},\,\theta_{2,\zeta},\,r_{1,\xi},\,r_{2,\zeta},\,s_{1,\xi},\,s_{2,\zeta}$
($\xi,\,\zeta\in\{1,2\}$) fulfill (\ref{Lem1-01}).
Then using all these notations, and substituting (\ref{thm1-20}), (\ref{thm1-21}), (\ref{thm1-22}), and (\ref{thm1-23}) into (\ref{thm1-19}), we arrive at\\[2ex]
$\Upsilon^2_{\varsigma}(\psi u)\,\leq\,\mathfrak{c}_1L_{\psi'}\displaystyle\int^0_{-\varsigma}\|u_k(t)\|^2_{_{2,\Omega_k(t)}}\,dt+$\\
\begin{equation}\label{thm1-24}
+\mathfrak{c}_2k^2\displaystyle\sum^{2}_{\xi=1}\left(\displaystyle\int^0_{-\varsigma}|\Omega_k(t)|
^{^{\frac{r_{1,\xi}}{s_{1,\xi}}}}\,dt\right)^{^{\frac{2(1+\theta_{1,\xi})}{r_{1,\xi}}}}
+\mathfrak{c}_2k^2\displaystyle\sum^{2}_{\zeta=1}\left(\displaystyle\int^0_{-\varsigma}\mu(\Gamma_k(t))
^{^{\frac{r_{2,\zeta}}{s_{2,\zeta}}}}\,dt\right)^{^{\frac{2(1+\theta_{2,\zeta})}{r_{2,\zeta}}}}.
\end{equation}
Next, given $\varepsilon\in(0,1/2)$ fixed, choose $$\psi_{\varepsilon}(t):=\left\{
\begin{array}{lcl}
\displaystyle\frac{t+\varsigma}{\varepsilon\varsigma},\,\,\,\,\,\,\,\,\,\,\,\,\,\,\,\,\,\,\textrm{if}\,\,t\in[-\varsigma,-(1-\varepsilon)\varsigma],\\[2ex]
\,\,\,\,1,\,\,\,\,\,\,\,\,\,\,\,\,\,\,\,\,\,\,\,\,\,\,\,\,\,\textrm{if}\,\,t\in[-(1-\varepsilon)\varsigma,0].\\
\end{array}
\right.$$
Clearly $\psi_{\varepsilon}\in H^1(-\varsigma,0)$ with $0\leq\psi_{\varepsilon}\leq1$,\, $\psi_{\varepsilon}(-\varsigma)=0$, and
$L_{\psi'_{\varepsilon}}:=\displaystyle\sup_{-\varsigma\leq t\leq0}|\psi'_{\varepsilon}(t)|\leq(\varepsilon\varsigma)^{-1}$. Moreover, we
obtain that $$\Upsilon^2_{_{(1-\varepsilon)\varsigma}}(u)\,\leq\,\Upsilon^2_{\varsigma}(\psi_{\varepsilon}u),$$
and thus letting $\gamma_0:=\max\{\mathfrak{c}_1,\mathfrak{c}_2\}$ and replacing $\psi=\psi_{\varepsilon}$ in (\ref{thm1-24}), we obtain that\\[2ex]
$\Upsilon^2_{_{(1-\varepsilon)\varsigma}}(u)\,\leq\,\displaystyle
\frac{\gamma_0}{\varepsilon\varsigma}\displaystyle\int^0_{-\varsigma}\|u_k(t)\|^2_{_{2,\Omega_k(t)}}\,dt+$\\
\begin{equation}\label{thm1-25}
+\gamma_0k^2\displaystyle\sum^{2}_{\xi=1}\left(\displaystyle\int^0_{-\varsigma}|\Omega_k(t)|
^{^{\frac{r_{1,\xi}}{s_{1,\xi}}}}\,dt\right)^{^{\frac{2(1+\theta_{1,\xi})}{r_{1,\xi}}}}
+\gamma_0k^2\displaystyle\sum^{2}_{\zeta=1}\left(\displaystyle\int^0_{-\varsigma}\mu(\Gamma_k(t))
^{^{\frac{r_{2,\zeta}}{s_{2,\zeta}}}}\,dt\right)^{^{\frac{2(1+\theta_{2,\zeta})}{r_{2,\zeta}}}}.
\end{equation}
Therefore, (\ref{thm1-25}) shows that (\ref{Lem1-02}) and all the conditions of Lemma \ref{Lem1} are fulfilled, and thus
applying Lemma \ref{Lem1} to both $u$ and $-u$ (the latter one being a weak solution of (\ref{1.06}) for $f,\,g$ replaced by
$-f,\,-g$, respectively), we establish (\ref{Thm1-01}) in the case $T\leq T_0$.\\
$\bullet$\,\, {\it \underline{Step 2}}. Assume now in addition that $u(-T)=u_0=0$. Then by choosing $\varsigma:=T$ and
and $\psi(t):=1$ for al $t\in[-T,0]$, equality (\ref{thm1-02}) becomes
\begin{equation}\label{thm1-26}
\displaystyle\frac{1}{2}\displaystyle\int_{\Omega_k(t)}|u_k(t)|^2\,dx
=\displaystyle\int^t_{-T}\left(\displaystyle\int_{\Omega_k(\xi)}u_k(\xi)\frac{ du_k(\xi)}{d\xi}\,dx\right)d\xi.
\end{equation}
From here, the calculations follow in the exact way as in the previous step, and in particular (\ref{thm1-12}) is transformed into\\[2ex]
$\min\left\{\displaystyle\frac{1}{2},\displaystyle\frac{c_0}{4},c^{\ast\star}\right\}\Upsilon^2_{_T}(u)
\,\leq\,\displaystyle\sup_{-T\leq t\leq0}\left(\displaystyle\frac{1}{2}\|u_k(t)\|^2_{_{2,\Omega_k(t)}}\right)+
\frac{c_0}{4}\displaystyle\int^0_{-T}\|\nabla u_k(t)\|^2_{_{2,\Omega_k(t)}}\,dt\,+\,c^{\ast\star}
\displaystyle\int^0_{-T}\|u_k(t)\|^2_{_{D(\Lambda_{_{\Gamma}})}}\,dt$\\
\begin{equation}\label{thm1-27}
\leq\,\displaystyle\int^0_{-T}\left(\|u_k(t)\|^2_{_{2,\Omega_k(t)}}
+\displaystyle\int_{\Omega_k(t)}\left(\frac{1}{k}|f(t,x)|+\mathfrak{h}\right)(|u_k(t)|^2+k^2)\,dx
+\displaystyle\int_{\Gamma_k(t)}\left(\frac{1}{k}|g(t,x)|+\mathfrak{k}\right)(|u_k(t)|^2+k^2)\,d\mu\right)\,dt.
\end{equation}
Continuing analogously as in the previous case, we arrive a the inequality\\[2ex]
$\Upsilon^2_{_T}(u)\,\leq\,\gamma_0\displaystyle\int^0_{-T}\|u_k(t)\|^2_{_{2,\Omega_k(t)}}\,dt+$\\
\begin{equation}\label{thm1-28}
+\gamma_0k^2\displaystyle\sum^{2}_{\xi=1}\left(\displaystyle\int^0_{-T}|\Omega_k(t)|
^{^{\frac{r_{1,\xi}}{s_{1,\xi}}}}\,dt\right)^{^{\frac{2(1+\theta_{1,\xi})}{r_{1,\xi}}}}
+\gamma_0k^2\displaystyle\sum^{2}_{\zeta=1}\left(\displaystyle\int^0_{-T}\mu(\Gamma_k(t))
^{^{\frac{r_{2,\zeta}}{s_{2,\zeta}}}}\,dt\right)^{^{\frac{2(1+\theta_{2,\zeta})}{r_{2,\zeta}}}},
\end{equation}
where all the constants and parameters appearing in (\ref{thm1-28}) are defined as in the previous step. Consequently,
(\ref{Lem1-04}) is valid for this case, and another application of Lemma \ref{Lem1} gives (\ref{Thm1-02}).
This completes the proof of Theorem \ref{Thm-Partial-Bounded} in the case when the weak solution
$u\in C(0,T;L^2(\Omega))\cap L^2(0,T;\mathcal{W}_2(\Omega;\Gamma))$ of (\ref{1.06}) is also a
classical solution of (\ref{1.06}), and for $T\leq T_0$.\\
$\bullet$\,\, {\it \underline{Step 3}}. We now prove the general case. Since the proof runs in a similar way as in
\cite[proof of Proposition 3.1]{NITTKA2014}, we will only sketch the main ingredients of the proof. Let
$u\in C(0,T;L^2(\Omega))\cap L^2(0,T;\mathcal{W}_2(\Omega;\Gamma))$ be a weak solution of problem (\ref{1.06}). Since $D(A_{\mu})$ is dense
in $L^2(\Omega)$, we pick $\{u_{0,n}\}\subseteq D(A^2_{\mu})$ such that $u_{0,n}\stackrel{n\rightarrow\infty}{\longrightarrow}u_0$ in $L^2(\Omega)$.
We also take sequences $\{f_n\}\subseteq C^2([0,T];L^{\infty}(\Omega))$ and $\{g_n\}\subseteq C^2([0,T];L^{\infty}_{\mu}(\Gamma))$
fulfilling $f_{n}\stackrel{n\rightarrow\infty}{\longrightarrow}f$ in $L^{\kappa_1}(0,T;L^p(\Omega))$ and
$g_{n}\stackrel{n\rightarrow\infty}{\longrightarrow}g$ in $L^{\kappa_2}(0,T;L^q_{\mu}(\Gamma))$, and $f_n(0,x)=g_n(0,x)=0$ for every $n\in\mathbb{N\!}\,$.
Using $f_n,\,g_n$ as the data functions in problem (\ref{1.06}), from an application of Theorem \ref{reg-sol-class}, Theorem \ref{exist-uniq-weak-soln},
and Corollary \ref{coincidence}, we get that problem (\ref{1.06}) (with data $f_n,\,g_n$) admits a unique classical solution
$u_n$, and $u_{n}\stackrel{n\rightarrow\infty}{\longrightarrow}u$ in $C([0,T];L^2(\Omega))\cap L^2(0,T;\mathcal{W}_2(\Omega;\Gamma))$.
Then, given $T>0$ arbitrary, put $T'_0:=\min\{T_0,T\}>0$, for $T_0>0$ the parameter given in the previous two steps, and let $I'\subseteq[T'_0/2,T'_0]$ be an interval.
Then the inequality (\ref{Thm1-01}) is valid (over $I'$) for the function $u_n$ (with data $f_n,\,g_n$), and moreover an application of (\ref{Thm1-01}) to
the function $u_n-u_m$ ($n,\,m\in\mathbb{N\!}\,$) shows that $\{u_n\}$ is a Cauchy sequence in $L^{\infty}(I';L^{\infty}(\Omega))$. Thus
$u_{n}\stackrel{n\rightarrow\infty}{\longrightarrow}u$ over $L^{\infty}(I';L^{\infty}(\Omega))$. Therefore, passing to the limit and
using the fact that $[T/2,T]\subseteq\displaystyle\bigcup^m_{j=1}I'_i$ (for some $m\in\mathbb{N\!}\,$) with $\ell(I'_i)\leq L'_0/2$, we are lead into the fulfillment
of the inequality (\ref{Thm1-01}). If in addition $u_0=0$, then the same procedure is repeated over the inequality (\ref{Thm1-02}), and this
completes the proof of the theorem.
\end{proof}

\indent Now we close this section by providing a priori estimates depending in the initial condition of problem (\ref{1.06}). We begin
with an estimate global in time. In this case, since the initial condition is only assumed to be bounded in $\Omega$, if $u_0\neq0$,
then it is still an open problem to obtain global $L^{\infty}$-estimates in the sense that in this case weak solutions to equation (\ref{1.06}) are only known to be bounded
over $\Omega$ (in fact, they are globally bounded whenever $t>0$, but $u_0$ may fail to be bounded unless specified otherwise).
To be more precise, we have the following result whose proof runs as in \cite[proof of Theorem 3.2]{NITTKA2014}.

\begin{theorem}\label{infitiny-u_0-infinity}
Suppose that conditions {\bf (A3)'} and {\bf (B4)'} hold.
Given $T>0$ arbitrary fixed, let $\kappa_1,\,\kappa_2,\,p,\,q\in[2,\infty)$ be such that
\begin{equation}\label{Thm2-00}
\displaystyle\frac{1}{\kappa_1}+\displaystyle\frac{N}{2p}<1\,\,\,\,\,\,\,\,\textrm{and}
\,\,\,\,\,\,\,\,\displaystyle\frac{1}{\kappa_2}+\displaystyle\frac{d}{2q(d+2-N)}<\frac{1}{2}.
\end{equation}
Given $u_0\in L^{\infty}(\Omega)$, let $f\in L^{\kappa_1}(0,T;L^p(\Omega))$ and $g\in L^{\kappa_2}(0,T;L^q_{\mu}(\Gamma))$, and
let $u\in C(0,T;L^2(\Omega))\cap L^2(0,T;\mathcal{W}_2(\Omega;\Gamma))$ be a weak solution of problem (\ref{1.06}).
Then there exists a constant $C^{\ast}_2=C^{\ast}_2(T,N,d,\kappa_1,\kappa_2,p,q,|\Omega|,\mu(\Gamma))>0$ such that
\begin{equation}\label{Thm2-01}
\|u\|_{_{L^{\infty}(0,T;L^{\infty}(\Omega))}}
\,\leq\,C^{\ast}_2\left(\|u_0\|_{_{\infty,\Omega}}+\|f\|_{_{L^{\kappa_1}(0,T;L^{p}(\Omega))}}+
\|g\|_{_{L^{\kappa_2}(0,T;L^{q}_{\mu}(\Gamma))}}\right).
\end{equation}
\end{theorem}

\begin{proof}
If $u\in C(0,T;L^2(\Omega))\cap L^2(0,T;\mathcal{W}_2(\Omega;\Gamma))$ solves (\ref{1.06}), then by linearity one can write $u(t):=T_{\mu}(t)u_0+\tilde{u}(t)$,
where $T_{\mu}(t)u_0$ denotes the weak solution of the homogeneous problem (\ref{1.06}) for $f=g=0$ (recall Corollary \ref{coincidence}),
and $\tilde{u}(t)$ solves (\ref{1.06}) for $u_0=0$. Then an application of (\ref{infinity-semigroup}) together with (\ref{Thm1-02}) entail that\\[2ex]
$\|u\|^2_{_{L^{\infty}(0,T;L^{\infty}(\Omega))}}\,\leq\,2\left(\displaystyle\sup_{t\in[0,T]}\|T_{\mu}(t)u_0\|^2_{_{\infty,\Omega}}
+\|\tilde{u}\|^2_{_{L^{\infty}(0,T;L^{\infty}(\Omega))}}\right)$\\
\begin{equation}\label{Thm2-02}
\leq\,2\left\{M^2e^{2|\omega|T}\|u_0\|^2_{_{\infty,\Omega}}+C^{\ast}_1\left(\|\tilde{u}\|_{_{L^{2}(0,T;L^{2}(\Omega))}}+
\|f\|_{_{L^{\kappa_1}(0,T;L^{p}(\Omega))}}+\|g\|_{_{L^{\kappa_2}(0,T;L^{q}_{\mu}(\Gamma))}}\right)\right\}.
\end{equation}
Furthermore, an application of Proposition \ref{apriori estimates} with Remark \ref{approx-classical} and H\"older's inequality gives that
\begin{equation}\label{Thm2-03}
\|\tilde{u}\|_{_{L^{2}(0,T;L^{2}(\Omega))}}
\leq\,cT\left(\|u_0\|^2_{_{\infty,\Omega}}+
\|f\|_{_{L^{\kappa_1}(0,T;L^{p}(\Omega))}}+\|g\|_{_{L^{\kappa_2}(0,T;L^{q}_{\mu}(\Gamma))}}\right),
\end{equation}
for some constant $c>0$.
Combining (\ref{Thm2-02}) and (\ref{Thm2-03}) yield the inequality (\ref{Thm2-01}), as asserted.
\end{proof}

\indent Finally, we show that if problem (\ref{1.06}) is local in time, then we obtain global boundedness in space.

\begin{theorem}\label{infitiny-u_0-L2}
Suppose that conditions {\bf (A3)'} and {\bf (B4)'} hold.
Given $T_2>T_1>0$ arbitrary fixed, let $\kappa_1,\,\kappa_2,\,p,\,q\in[2,\infty)$ be as in Theorem \ref{infitiny-u_0-infinity}.
Given $u_0\in L^{2}(\Omega)$, let $f\in L^{\kappa_1}(0,T;L^p(\Omega))$ and $g\in L^{\kappa_2}(0,T;L^q_{\mu}(\Gamma))$, and
let $u\in C(0,T;L^2(\Omega))\cap L^2(0,T;\mathcal{W}_2(\Omega;\Gamma))$ be a weak solution of problem (\ref{1.06}).
Then there exists a constant $C^{\ast}_3=C^{\ast}_3(T_1,T_2,N,d,\kappa_1,\kappa_2,p,q,|\Omega|,\mu(\Gamma))>0$ such that
\begin{equation}\label{Thm3-01}
|\|\mathbf{u}\||_{_{L^{\infty}(T_1,T_2;\mathbb{X\!}^{\,\infty}(\Omega;\Gamma))}}
\,\leq\,C^{\ast}_3\left(\|u_0\|_{_{2,\Omega}}+\|f\|_{_{L^{\kappa_1}(0,T;L^{p}(\Omega))}}+
\|g\|_{_{L^{\kappa_2}(0,T;L^{q}_{\mu}(\Gamma))}}\right).
\end{equation}
\end{theorem}

\begin{proof}
If $u\in C(0,T;L^2(\Omega))\cap L^2(0,T;\mathcal{W}_2(\Omega;\Gamma))$ is a weak solution of (\ref{1.06}), then by using the fact that $[T_1,T_2]\subseteq\displaystyle\bigcup^m_{j=1}[T^{\ast}_j/2,T^{\ast}_j]$
with $T^{\ast}_1<T^{\ast}_2<\ldots<T^{\ast}_m$ (for some $m\in\mathbb{N\!}\,$), and $T^{\ast}_1=T_1$, \,$T^{\ast}_m=T_2$,
one can the full strength of Theorem \ref{Thm-Partial-Bounded} to get that
\begin{equation}\label{Thm3-02}
|\|\mathbf{u}\||_{_{L^{\infty}(T_1,T_2;\mathbb{X\!}^{\,\infty}(\Omega;\Gamma))}}
\,\leq\,C^{\ast}_1\left(\|u\|_{_{L^{2}(0,T;L^{2}(\Omega))}}+\|f\|_{_{L^{\kappa_1}(0,T;L^{p}(\Omega))}}+
\|g\|_{_{L^{\kappa_2}(0,T;L^{q}_{\mu}(\Gamma))}}\right).
\end{equation}
Moreover, an application of Proposition \ref{apriori estimates} with Remark \ref{approx-classical} and H\"older's inequality gives that
\begin{equation}\label{Thm3-03}
\|u\|_{_{L^{2}(0,T;L^{2}(\Omega))}}
\leq\,c'\left(\|u_0\|^2_{_{2,\Omega}}+
\|f\|_{_{L^{\kappa_1}(0,T;L^{p}(\Omega))}}+\|g\|_{_{L^{\kappa_2}(0,T;L^{q}_{\mu}(\Gamma))}}\right).
\end{equation}
for some constant $c'>0$.
Combining (\ref{Thm3-02}) and (\ref{Thm3-03}) gives (\ref{Thm3-01}), as desired.
\end{proof}

\section{Examples: domains}\label{sec05}

\indent In the following section we provide multiple examples of bounded connected non-smooth and irregular domains. Most of the domains
will have the extension property, but at the end we will briefly discuss the situation when the domain fails to be an extension domain.

\subsection{$(\epsilon,\delta)$-domains}\label{subsec5-01}

\begin{definition}\label{ep-del-dom}
A domain $\Omega\subseteq\mathbb{R\!}^N$ is said to be an\, \textbf{($\mathbf{\epsilon,\,\delta}$)}-\textbf{domain},
if there exists $\delta\in (0,+\infty]$ and there exists $\epsilon\in (0,1]$, such that for each two points $x,\,y\in\Omega$
with $|x-y|\leq\delta$, there exists a continuous rectifiable curve $\gamma :[0,t]\rightarrow\Omega$ such that $\gamma(0)=x$ and
$\gamma(t)=y$, with the following properties:
\begin{enumerate}
\item[(a)]\,\,\,\,$\ell(\{\gamma\})\leq\frac{1}{\epsilon}|x-y|$.
\item[(b)]\,\,\, $\textrm{dist}(z,\partial\Omega)\geq\,\epsilon\,\min
\{|x-z|,|y-z|\},\,\,\,\textrm{for all}\,\,z\in\{\gamma\}$.
\end{enumerate}
Also, an $(\epsilon,\,\infty)$-domain is called an \textbf{uniform domain}.
\end{definition}

\indent For this class of domains, it was established by Jones in a well-known result (see \cite[Theorem 1]{JON}) that
$(\epsilon,\delta)$-domains are extension domains for all $p\in[1,\infty]$. Since we are assuming that the domains are connected, then again
in \cite{JON}, it was shown that extension domains are equivalent to $(\epsilon,\delta)$-domains when $N=2$. For $N>2$ however, there exist
extension domains who fail to be $(\epsilon,\delta)$-domains (see for instance \cite{JON80,YANG}).
Furthermore, we recall that if $\Omega$ is a bounded $(\epsilon,\delta)$-domain whose boundary is a
$d$-set with respect to a measure $\mu$, then it was obtained in \cite{JO-WAL} that this condition is equivalent to say that
the $d$-dimensional Hausdorff measure $\mathcal{H}^d$ is a $d$-Ahlfors measure on $\Gamma$.\\
\indent Now we present several more concrete examples of $(\epsilon,\delta)$-domains.

\begin{example}\label{Ex5.01}
If $\Omega\subseteq\mathbb{R\!}^N$ is a bounded Lipschitz domain,
then it is well known that $\Omega$ is an extension domain, and the classical surface measure
$\mathcal{H}^{N-1}$ is an upper $(N-1)$-Ahlfors measure on $\Gamma$.
\end{example}

\begin{example}\label{Ex5.02}
Let $\Omega\subseteq\mathbb{R\!}^{\,2}$ be the classical Koch snowflake domain (see the image below).
\begin{center}
\begin{tikzpicture}
\shadedraw[shading=color wheel]
[l-system={rule set={F -> F-F++F-F}, step=1.5pt, angle=60,
   axiom=F++F++F, order=4}] lindenmayer system -- cycle;
\end{tikzpicture}
\end{center}
Then it is well known that $\Omega$ is an extension domain, and the $d$-dimensional Hausdorff measure
$\mathcal{H}^{d}$ restricted to $\Gamma:=\partial\Omega$ is an upper $d$-Ahlfors measure on $\Gamma$, for $d:=\log(4)/\log(3)$ (e.g. \cite{WAL}).
\end{example}

\begin{example}\label{Ex5.03}
Motivated by some recent results in \cite{FERR-VELEZ18-1}, given $m >2$ such that $m \not\equiv 0 \,(\bmod\,\, 3)$, we define the \textbf{Koch $\mathbf{m}$-Crystal} $\Omega_m\subseteq\mathbb{R\!}^{\,3}$ as the interior of the closed set enclosed by four congruent Koch $m$-Surfaces,
each pair of which intersect at precisely one edge (see Figure 1, and refer to \cite{FERR-VELEZ18-1} for more details of this construction).
We then define $\Gamma_{m}$ as the boundary of $\Omega_{m}$.
\begin{figure}[h]
    \centering
    \includegraphics[scale=0.20]{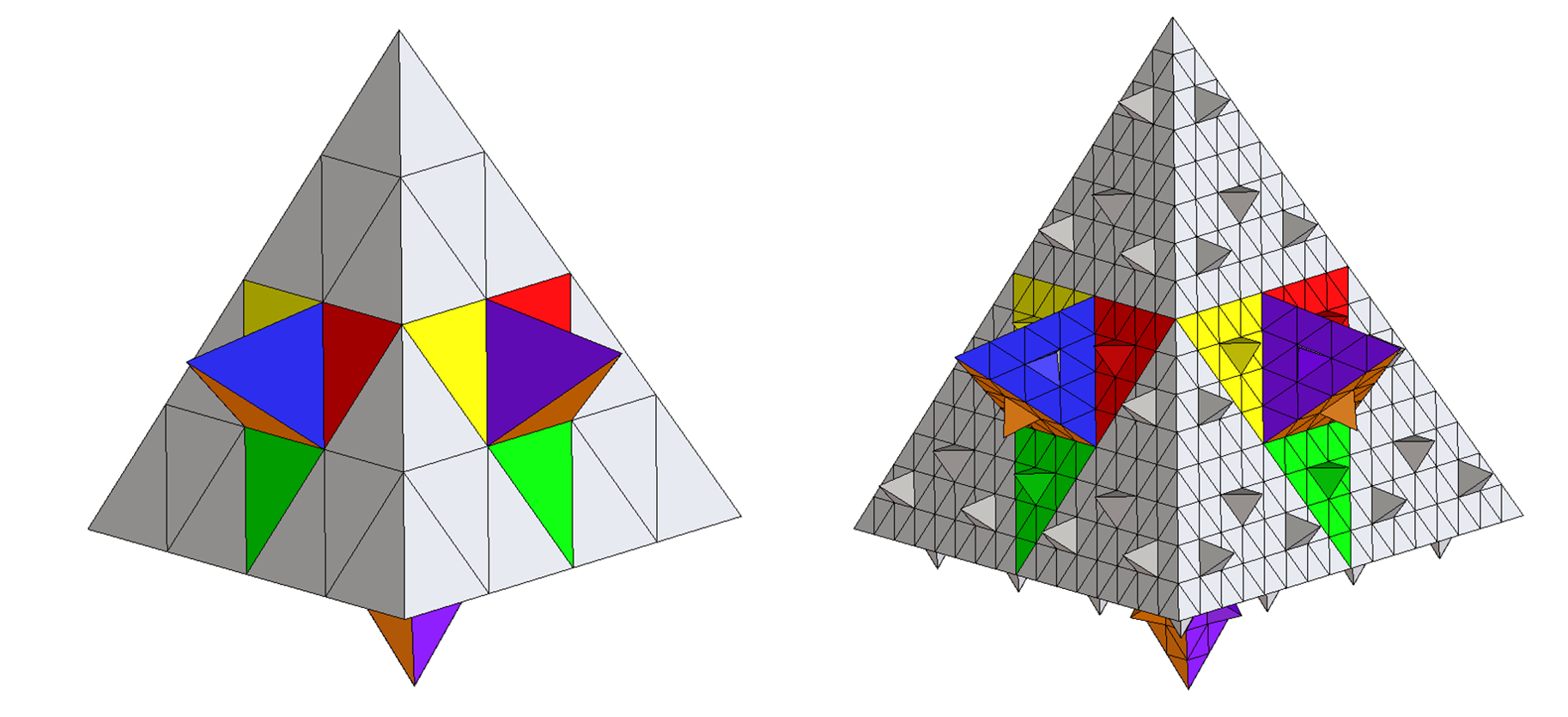}

    \includegraphics[scale=0.10]{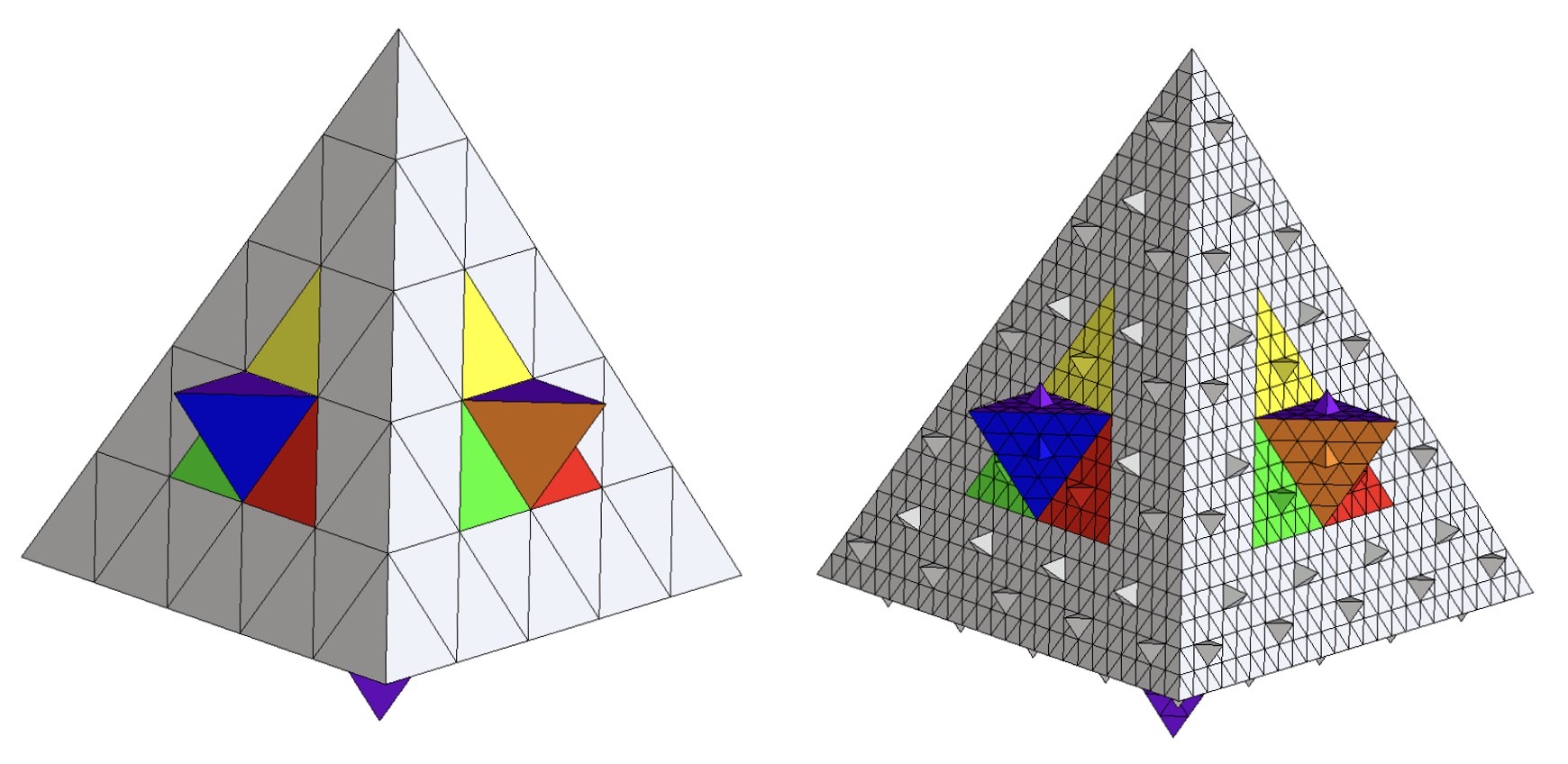}
    \caption{Koch 4-Crystal and 5-Crystal,  First and Second Iterations}
    \label{fig:my_label}
\end{figure}
Then it is shown in \cite{FERR-VELEZ18-1} that $\Omega_m$ is an extension domain, and the $d$-dimensional Hausdorff measure
$\mathcal{H}^{d}$ restricted to $\Gamma:=\partial\Omega$ is an upper $d$-Ahlfors measure on $\Gamma$, for $d:=\log (m^2 + 2) / \log(m)$.
\end{example}

\subsection{$\delta$-Reifenberg flat domains}\label{subsec5-02}

\indent In this subsection we present a class of rough domains who, under certain conditions, fulfill the required properties required for the
validity of all the results in previous sections. We begin with the following definition, taken from \cite{KENIG-TORO03,LEM-MIL-SP}.

\begin{definition}
A domain $\Omega\subseteq\mathbb{R\!}^N$ is said to be {\bf $(\delta,\rho_0)$-Reifenberg flat}, if the following conditions hold:
\begin{enumerate}
\item[(a)]\,\, For each $x_0\in\Gamma:=\partial\Omega$ and $\rho\in(0,\rho_0)$, there exits an hyperplane $\Upsilon$ containing $x_0$,  such that
$$d_{_{\mathcal{H}}}\left(\Gamma\cap B(x_0,\rho),\,\Upsilon\cap B(x_0,\rho)\right)\,\leq\,\delta\rho,$$
where $d_{_{\mathcal{H}}}(\cdot,\cdot)$ denotes the Hausdorff distance function.
\item[(b)]\,\, The set $\left\{x\in\Omega\cap B(x_0,\rho)\mid d(x,\Gamma)\geq 2\delta\rho\right\}$ is contained in one component of $\mathbb{R\!}^N\setminus\Upsilon$.
\end{enumerate}
Furthermore, if in addition the boundary $\Gamma$ is a $(N-1)$-set wit respect to the $(N-1)$-dimensional Hausdorff measure, then $\Omega$
is called a {\bf $(\rho_0,\delta)$-Reifenberg flat chord arc domain}.
\end{definition}

\indent The following key result was established by Lemenant, Milakis, and Spinolo \cite{LEM-MIL-SP}.

\begin{theorem}\label{Reifenberg-extension}\,(see \cite{LEM-MIL-SP})\, If $\Omega\subseteq\mathbb{R\!}^N$ is a $(\delta,\rho_0)$-Reifenberg flat domain for
$\delta\leq1/600$ and $\rho_0>0$, then $\Omega$ is an $(\epsilon,\,\delta')$-domain for $\epsilon:=1/450$ and $\delta':=\rho_0/7$.
\end{theorem}

\indent As a direct consequence of the above theorem, one sees that if $\Omega\subseteq\mathbb{R\!}^N$ is a $(\delta,r_0)$-Reifenberg flat chord arc domain for
$\delta\leq1/600$, then the domain satisfy all the conditions in Assumption \ref{As1}.

\subsection{Ramified domains}\label{subsec5-03}

\indent We now consider a class of domains with ramified fractal-like boundaries. Such domains are of particular interest in applications to
bronchial trees (see for instance \cite{ACH-SAB06-01,ACH08,Silva-Perez-Velez-Santiago-2024} among others).\\
\indent In fact, we consider the following two particular domains $\Omega_1$ and $\Omega_2$ investigated in
\cite{ACH-SAB06-01,ACH-SAB06-02,ACH10,ACH08} (among others).
Let $\{F_1,F_2\}$ and $\{G_1,G_2\}$ denote two sets, each one consisting of a pair of similitudes on $\mathbb{R\!}^{\,2}$ defined by

$$F_i(x_1,x_2):=\left(\begin{array}{lcl}
(-1)^i\displaystyle\frac{3}{2}+\displaystyle\frac{x_1}{2}\\
\indent\\
\,\,\,\,\,\,\,\,\,3+\displaystyle\frac{x_2}{2}\\
\end{array}
\right)$$
and
$$G_i(x_1,x_2):=\left(\begin{array}{lcl}
(-1)^i\left[1-\displaystyle\frac{\tau}{\sqrt{2}}\right]+\displaystyle\frac{\tau}{\sqrt{2}}[x_1+(-1)^ix_2]\\
\indent\\
1+\displaystyle\frac{\tau}{\sqrt{2}}+\displaystyle\frac{\tau}{\sqrt{2}}[x_2+(-1)^{i+1}x_1]\\
\end{array}
\right)$$

\begin{figure}[h!]
    \centering
    \includegraphics[scale=0.21]{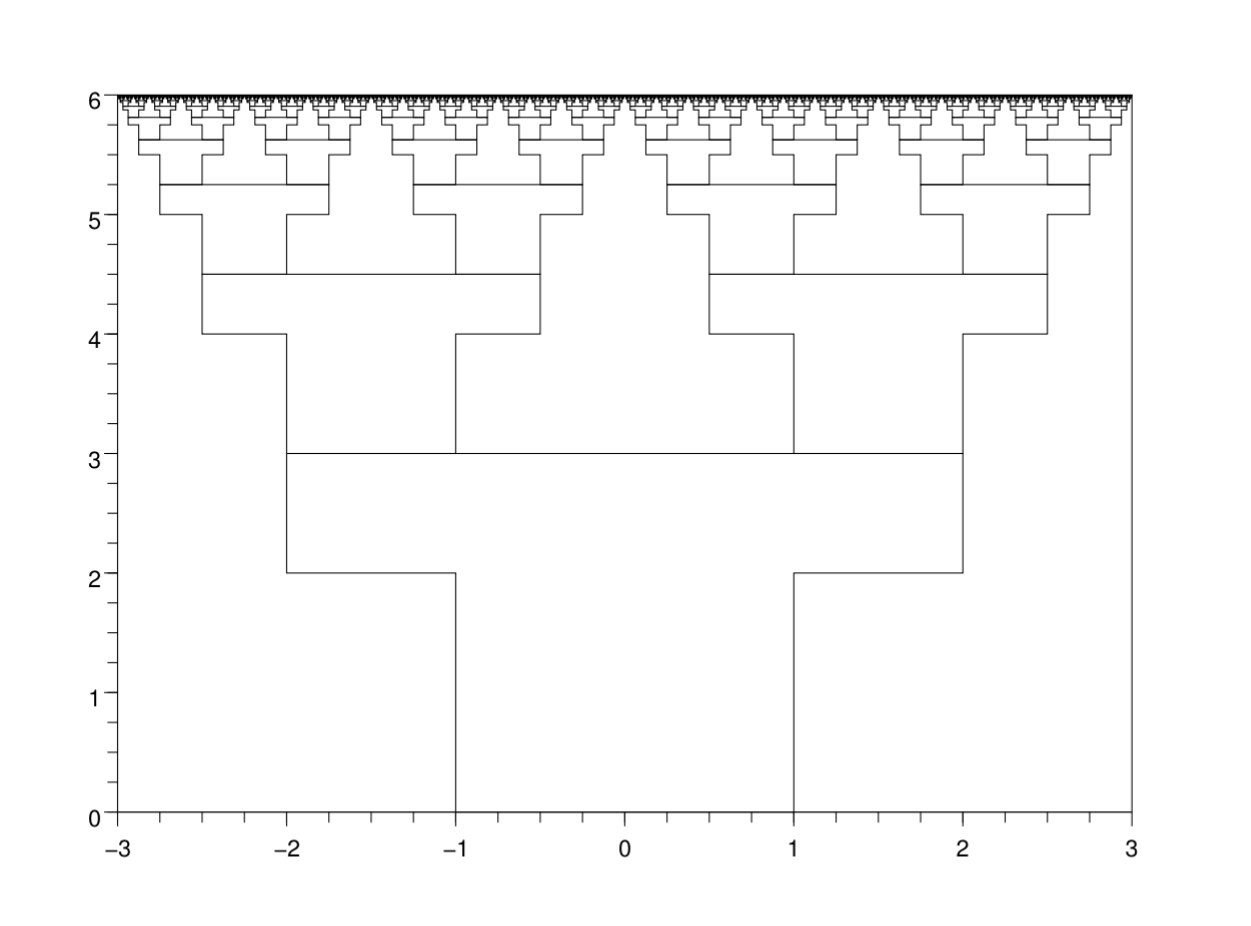}

    \includegraphics[scale=0.21]{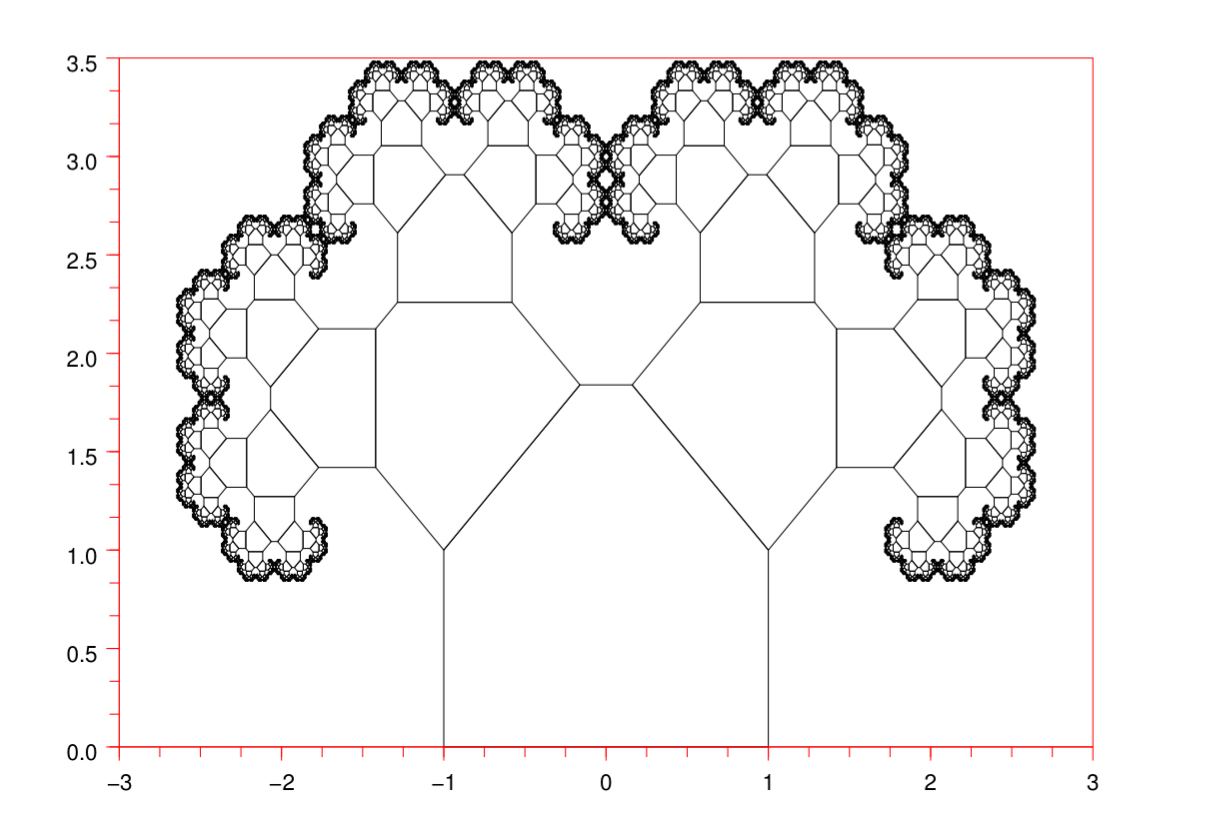}
    \caption{\underline{Figure above}: the domain $\Omega_1$ generated by the maps $F_i$ (taken from \cite{ACH-SAB06-01});
     \,\underline{Figure below}: the domain $\Omega_2$ generated by the maps $G_i$ when $\tau=\tau^{\ast}$ (taken from \cite{ACH08}).}
    \label{fig:my_label}
\end{figure}

\noindent ($i\in\{1,2\}$) for $1/2\leq\tau<\tau^{\ast}:=0.593465$.
Let $\mathfrak{A}_n$ be the set containing all the maps from
$\{1,\ldots,n\}$ to $\{1,2\}$. Then for $\sigma\in\mathfrak{A}_n$, we define the affine map
$$\mathcal{M}_{\sigma}(H_1,H_2):=H_{\sigma(1)}\circ\cdot\cdot\cdot\circ H_{\sigma(n)}\,\,\,\,\,\,\,\,\,\,\,\,\,\,\,\,\,\,\,
\,\,\,\,\textrm{(}\,H\,\,\textrm{stands as either}\,\,F,\,\,\textrm{or}\,\,G\,\textrm{)}.$$
Then, letting $V_1$ be the interior of the set $([-1,1]\times[0,2])\cup([-2,2]\times[2,3])$, and $V_2$ be the interior of the convex hull of the
points $P_1=(-1,0)$,\, $P_2=(1,0)$,\, $P_3=G_1(P_1)$,\, $P_4=G_2(P_2)$,\, $P_5=G_1(P_2)$, and $P_6=G_2(P_1)$
we define
$$\Omega_1:=\textrm{Interior}\left(\overline{V_1}\cup\left(\displaystyle\bigcup^{\infty}_{n=1}\displaystyle\bigcup_{\sigma\in\mathfrak{A}_n}
\mathcal{M}_{\sigma}(F_1,F_2)(\overline{V_1})\right)\right),$$
and
$$\Omega_2:=\textrm{Interior}\left(\overline{V_2}\cup\left(\displaystyle\bigcup^{\infty}_{n=1}\displaystyle\bigcup_{\sigma\in\mathfrak{A}_n}
\mathcal{M}_{\sigma}(G_1,G_2)(\overline{V_2})\right)\right).$$
It follows that the ramified boundary $\Gamma^{\infty}_i$ ($i\in\{1,2\}$) corresponds to the unique compact self-similar set in $\mathbb{R\!}^{\,2}$ fulfilling
$$\Gamma^{\infty}_1=F_1(\Gamma^{\infty}_1)\cup F_2(\Gamma^{\infty}_1)\,\,\,\,\,\,\,\,\,\,\,\,\,\,\,\,\textrm{and}
\,\,\,\,\,\,\,\,\,\,\,\,\,\,\,\,\Gamma^{\infty}_2=G_1(\Gamma^{\infty}_2)\cup G_2(\Gamma^{\infty}_2).$$
For more information and properties of these sets, refer to \cite{ACH-SAB06-01,ACH-SAB06-02,ACH10,ACH08} (among others). In particular, one can summarize
several of the main properties in the following result.

\begin{theorem}\label{prop-ramified}\,(see \cite{ACH-SAB06-02,ACH-SAB06-01,ACH10,ACH08})\,The following properties are valid for the sets
$\Omega_1$ and $\Omega_2$.
\begin{enumerate}
\item[(a)]\,\, $\Omega_1$ and $\Omega_2$ have the open set condition.
\item[(b)]\,\, $\Omega_1$ is not an extension domain ($q=2$), but it is a $W^{1,q}$- extension domain for all $q\in[1,2)$.
\item[(c)]\,\, $\Omega_2$ is an $(\epsilon,\delta)$-domain if $1/2\leq\tau<\tau^{\ast}$.
\item[(d)]\,\, $\Omega_2$ is not an extension domain when $\tau=\tau^{\ast}$, but it is a $W^{1,q}$- extension domain for each $q\in[1,2)$.
\item[(e)]\,\, The ramified boundary $\Gamma^{\infty}_i$ ($1\leq i\leq2$) is a $d_i$-set with respect to the self-similar
measure supported on $\Gamma^{\infty}_i$, where
$$d_1=1\,\,\,\,\,\,\,\,\,\,\,\,\,\,\,\,\textrm{and}
\,\,\,\,\,\,\,\,\,\,\,\,\,\,\,\, d_2:=-\log(2)/\log(\tau).$$
\end{enumerate}
\end{theorem}

\indent For each $i\in\{1,2\}$, define now the function space $$\widetilde{\mathcal{W}}_2(\Omega_i;\Gamma^{\infty}_i):=\left\{w\in\mathcal{W}_2(\Omega_i;\Gamma_i)\mid
w|_{_{\Gamma_i\setminus\Gamma^{\infty}_i}}=0\right\}.$$ Clearly $\widetilde{\mathcal{W}}_2(\Omega_i;\Gamma^{\infty}_i)$ is a closed subspace of
$\mathcal{W}_2(\Omega_i;\Gamma_i)$. Then, combining this fact with a consequence of Theorem \ref{prop-ramified} , we see that the maps
$\widetilde{\mathcal{W}}_2(\Omega_i;\Gamma^{\infty}_i)\hookrightarrow L^r(\Omega_i)$ and
$\widetilde{\mathcal{W}}_2(\Omega_i;\Gamma^{\infty}_i)\hookrightarrow L^s_{\mu}(\Gamma^{\infty}_i)$ are linear and compact for each $r,\,s\in[1,\infty)$.\\
\indent In this case, we then consider the solvability and regularity theory for the mixed-type inhomogeneous boundary elliptic problem
\begin{equation}
\label{1.05m}\left\{
\begin{array}{lcl}
\mathcal{A}u\,=\,f(x)\indent\,\textrm{in}\,\,\Omega_i\\
\,\,\,u=0\,\,\,\,\,\,\,\,\,\,\,\,\,\,\,\,\,\,\,\,\textrm{in}\,\,\Gamma_i\setminus\Gamma^{\infty}_i\\
\mathcal{B}u\,=\,g(x)\indent\,\textrm{on}\,\,\Gamma^{\infty}_i\\
\end{array}
\right.
\end{equation}
and the corresponding inhomogeneous parabolic problem
\begin{equation}
\label{1.06m}\left\{
\begin{array}{lcl}
u_t-\mathcal{A}u\,=\,f(t,x)\indent\,\textrm{in}\,\,(0,\infty)\times\Omega_i\\
\,\,\,u=0\,\,\,\,\,\,\,\,\,\,\,\,\,\,\,\,\,\,\,\,\,\,\,\,\,\,\,\,\,\,\,\,\,\,\,\,\,\,\,\textrm{in}\,\,(0,\infty)\times\Gamma_i\setminus\Gamma^{\infty}_i\\
\,\,\,\,\mathcal{B}u\,\,\,\,\,\,\,\,=\,g(t,x)\indent\,\textrm{on}\,\,(0,\infty)\times\Gamma^{\infty}_i\\
\,u(0,x)=u_0\,\,\,\,\,\,\,\,\,\,\,\,\,\,\,\,\,\,\,\,\,\,\,\,\,\,\,\textrm{in}\,\,\Omega_i
\end{array}
\right.
\end{equation}
for $i\in\{1,2\}$. Then, one proceeds in the exact way as in the derivation of the general results,to achieve the following consequence.

\begin{theorem}\label{Ramified-Result}
Given $i\in\{1,2\}$ and $T>0$ a fixed constant, let $f\in L^2(\Omega_i)$, \,$g\in L^2_{\mu}(\Gamma^{\infty}_i)$,\, $\tilde{f}\in L^2(0,T;L^2(\Omega_i))$, and
$\tilde{g}\in L^2(0,T;L^2_{\mu}(\Gamma^{\infty}_i))$. Then the following hold.
\begin{enumerate}
\item[(a)]\,\,\,Problem (\ref{1.05}) is uniquely solvable over $\widetilde{\mathcal{W}}_2(\Omega_i;\Gamma^{\infty}_i)$.
Moreover, if $f\in L^p(\Omega_i)$ for $p>1$, and $g\in L^q_{\mu}(\Gamma^{\infty}_i)$ for $q>1$, then $u$ is globally bounded with
$$|\|\mathbf{u}\||_{_{\infty}}\,\leq\,C\left(\|f\|_{_{p,\Omega_i}}+
\|g\|_{_{q,\Gamma^{\infty}_i}}+\|u\|_{_{2,\Omega_i}}\right),$$
for some constant $C>0$. Furthermore, if any of the conditions in Remark \ref{R-embedding} holds, then one has that
$$|\|\mathbf{u}\||_{_{\infty}}\,\leq\,C\left(\|f\|_{_{p,\Omega_i}}+\|g\|_{_{q,\Gamma^{\infty}_i}}\right).$$
\item[(b)]\,\,\,Problem (\ref{1.06}) is uniquely solvable over $C([0,T];L^2(\Omega_i))\cap L^2((0,T),\widetilde{\mathcal{W}}_2(\Omega_i;\Gamma^{\infty}_i))$. Moreover,
given $\tilde{p},\,\tilde{q}\in[2,\infty)$ fulfilling
$$\displaystyle\frac{1}{\kappa_{\tilde{p}}}+\displaystyle\frac{1}{\tilde{p}}<1,\,\,\,\,\,\,\,\,\,\,\,\,\,\,\,\,\textrm{and}\,\,\,\,\,\,\,\,
\,\,\,\,\,\,\,\,\displaystyle\frac{1}{\kappa_{\tilde{q}}}+\displaystyle\frac{1}{\tilde{q}}<1,$$
if $\tilde{f}\in L^{\kappa_{\tilde{p}}}(0,T;L^{\tilde{p}}(\Omega_i))$,\, $\tilde{g}\in L^{\kappa_{\tilde{q}}}
(0,T;L^{\tilde{q}}_{\mu}(\Gamma^{\infty}_i))$, and $u_0\in L^{\infty}(\Omega_i)$, then
$$\|u\|_{_{L^{\infty}(0,T;L^{\infty}(\Omega_i))}}
\,\leq\,C\left(\|u_0\|_{_{\infty,\Omega_i}}+\|\tilde{f}\|_{_{L^{\kappa_{\tilde{p}}}(0,T;L^{\tilde{p}}(\Omega_i))}}+
\|\tilde{g}\|_{_{L^{\kappa_{\tilde{q}}}(0,T;L^{\tilde{q}}_{\mu}(\Gamma^{\infty}_i))}}\right),$$
for some constant $C>0$. Additionally, all the results in subsections \ref{subsec04-02} and \ref{subsec04-03} are valid for this given situation.
\end{enumerate}
\end{theorem}

\indent Recently, new fine global regularity results for the local version of problems \eqref{1.05m} and \eqref{1.06m} have been obtained in \cite{Silva-Perez-Velez-Santiago-2024}.

\subsection{General bounded domains}\label{subsec5-04}

\indent To complete this section, let us consider the general case in which $\Omega\subseteq\mathbb{R\!}^N$ is an arbitrary bounded domain.
For this case, we assume that $\Gamma:=\partial\Omega$ has positive and finite $(N-1)$-dimensional Hausdorff measure ($0<\mathcal{H}^{N-1}(\Gamma)<\infty$),
and we assume in addition that $\mathcal{H}^{N-1}$ is an admissible measure in the sense of Arendt and Warma \cite{AR-WAR03}. Observe that this situation excludes
classes of domains such as the snowflake domain given in Example \ref{Ex5.02} (since the Koch curve has infinite $(N-1)$-Hausdorff measure), but may include other
general regions lacking the extension property, and such that their boundaries are not upper $(N-1)$-sets with respect to $\mathcal{H}^{N-1}$. For this case, it is
clear that Sobolev functions may lack a well-defined trace, and if such trace exists, it may not be $p$-integrable ($p\in[1,\infty)$). Consequently, following the approach by
Maz'ya \cite{MAZ}, we define the {\it Maz'ya's space} $W^1_p(\overline{\Omega})$ as the abstract completion of the space
$$V^1_p(\overline{\Omega}):=\left\{u\in W^{1,p}(\Omega)\cap C_c(\overline{\Omega})\mid u|_{_{\Gamma}}\in L^p_{_{\mathcal{H}^{N-1}}}(\Gamma)\right\}$$
with respect to the norm $$\|u\|_{_{W^1_p(\overline{\Omega})}}:=\left(\|\nabla u\|^p_{_{p,\Omega}}+\|u|_{_{\Gamma}}\|^p_{_{p,\Gamma}}\right)^{1/p}.$$
A key result by Maz'ya \cite{MAZ} establishes the existence of a constant $c>0$ such that
\begin{equation}\label{Mazya-Ineq}
\|u\|_{_{\frac{Np}{N-1}}}\,\leq\,c\,\|u\|_{_{W^1_p(\overline{\Omega})}}\,\,\,\,\,\,\,\,\,\,\,\,\,\,\,\,\,\,\textrm{for all}\,\,\,u\in W^1_p(\overline{\Omega}),
\end{equation}
and moreover the linear map $S:W^1_p(\overline{\Omega})\rightarrow L^r(\Omega)$ is compact whenever $1\leq r<Np(N-1)^{-1}$.
Inequality (\ref{Mazya-Ineq}) is optimal unless one assumes additional regularity properties over $\Omega$ (see for instance \cite[Example 2.11]{MAZ-POB97}).\\
\indent To consider nonlocal-type boundary value problems with structure similar as in (\ref{1.05}), we need to restrict the problem to a subspace of
$W^1_2(\overline{\Omega})$. In fact, define the space
$$\mathbb{W}^{\,1}_{\,p}(\overline{\Omega}):=\left\{w\in W^1_p(\overline{\Omega})\cap\mathbb{B\!}^{\,p}_{\,\mathfrak{s}}(\Omega)\mid w|_{_{\Gamma}}\in
\mathbb{B\!}^{\,p}_{_{\,1-\frac{N-d}{p}}}(\Gamma,\mathcal{H}^{N-1})\right\},$$ endowed with the norm
\begin{equation}\label{Mazya-Besov_Norm}
\|w\|_{_{\mathbb{W}^{\,1}_{\,p}(\overline{\Omega})}}:=\|w\|_{_{W^1_p(\overline{\Omega})}}+\mathcal{M}^p_{\mathfrak{s}}(w,\Omega)+
\mathcal{N}^p_{_{1-\frac{N-d}{p}}}(w,\Gamma,\mathcal{H}^{N-1}),
\end{equation}
where the expressions $\mathcal{M}^p_{\mathfrak{s}}(w,\Omega)$ and $\mathcal{N}^p_{_{1-\frac{N-d}{p}}}(w,\Gamma,\mathcal{H}^{N-1})$ are given by
(\ref{nonlocal-interior}) and (\ref{nonlocal-boundary}), respectively. Then one clearly deduce that the inequality (\ref{Mazya-Ineq}) remain valid if
one replaces the space $W^1_2(\overline{\Omega})$ by the subspace $\mathbb{W}^{\,1}_{\,p}(\overline{\Omega})$. Furthermore, since inequality (\ref{Mazya-Ineq})
holds for $p=1$, proceeding as in Remark \ref{R-embedding}(e), for each $\epsilon>0$, one gets the existence of a constant $c'>0$ such that
\begin{equation}\label{Mazya-Besov-Epsilon}
\|u\|^2_{_{\frac{2N}{N-1},\Omega}}\,\leq\,\epsilon\|\nabla u\|^2_{_{2,\Omega}}+\displaystyle\frac{c'}{\epsilon}\left(|\|\mathbf{u}\||^2_{_2}+
[\mathcal{M}^2_{\mathfrak{s}}(u,\Omega)]^2+[\mathcal{N}^2_{_{1-\frac{N-d}{2}}}(u,\Gamma,\mathcal{H}^{N-1})]^2\right),
\end{equation}
for each $u\in\mathbb{W}^{\,1}_{\,2}(\overline{\Omega})$.\\
\indent We now turn the solvability (over arbitrary bounded domains) of the boundary value problems
\begin{equation}
\label{arbitrary-1.05}\left\{
\begin{array}{lcl}
\,\,\,\,\,\mathcal{A}u\,=\,f(x)\indent\,\,\,\,\,\,\,\,\,\,\,\,\,\,\textrm{in}\,\,\Omega\\
(\mathcal{B}-\mathcal{L}_{\mu})u\,=\,0\indent\,\,\textrm{on}\,\,\Gamma\\
\end{array}
\right.
\end{equation}
and
\begin{equation}
\label{arbitrary-1.06}\left\{
\begin{array}{lcl}
u_t-\mathcal{A}u\,=\,f(t,x)\indent\,\textrm{in}\,\,(0,\times)\times\Omega\\
(\mathcal{B}-\mathcal{L}_{\mu})u\,=\,0\indent\,\textrm{on}\,\,(0,\infty)\times\Gamma\\
\,u(x,0)=u_0\,\,\,\,\,\,\,\,\,\,\,\,\,\,\,\,\,\,\,\,\,\,\,\,\,\,\,\,\textrm{in}\,\,\Omega,
\end{array}
\right.
\end{equation}
where we recall that the operator $\mathcal{L}_{\mu}$ is defined in (\ref{1.03}), and we are assuming the fulfillment of (\ref{1.04}).
In this case, we assume that $\alpha_{ij}\in L^{\infty}(\Omega)$,\,$\hat{a}_i\in L^{r_1}(\Omega)$, \,$\check{a}_i \in L^{r_2}(\Omega)$ ($1\leq i,\,j\leq N$),\,$\lambda
\in L^{r_3}(\Omega)$, and $\beta\in L^{\infty}_{_{\mathcal{H}^{N-1}}}(\Gamma)$, where $r_1,\,r_2\geq2N$, and $r_3\geq N$. Then, by virtue of (\ref{Mazya-Ineq}) and
(\ref{Mazya-Besov-Epsilon}), one proceeds in the same way as in the proofs of Theorem 3, Theorem 4, and Theorem 8, we can deduce the following result.

\begin{theorem}\label{Arbitrary-Result}
Given $T_2>T_1>0$ fixed constants, let $f\in L^2(\Omega)$ and $\tilde{f}\in L^2(0,T;L^2(\Omega))$, and assume that $\beta\in L^{\infty}_{\mu}(\Gamma)$ ($\mu(\cdot)=\mathcal{H}^{N-1}(\cdot)$) is such that $\textrm{ess}\displaystyle\inf_{\Gamma}[\beta]\geq C_{_{\Gamma}}$ for a sufficient large positive constant $C_{_{\Gamma}}$. Then the following hold.
\begin{enumerate}
\item[(a)]\,\,\,Problem (\ref{arbitrary-1.05}) is uniquely solvable over $\mathbb{W}^{\,1}_{\,2}(\overline{\Omega})$.
Moreover, one gets that
$$|\|\mathbf{u}\||_{_{m_1(p),m_2(p)}}\,\leq\,C\left(\|f\|_{_{p,\Omega}}+\|u\|_{_{2,\Omega}}\right),$$
for some constant $C>0$, where
\begin{equation}\label{m1-p}
m_1(p):=\left\{\begin{array}{lcl}
\frac{Np}{N-p},\,\,\textrm{if}\,\,p<N,\\[1ex]
r\in(2^{\ast}_{_N},\infty),\,\,\textrm{if}\,\,p=N,\\[1ex]
\,\,\,\,\,\infty,\,\,\textrm{if}\,\,p>N,
\end{array}
\right.
\,\,\,\,\,\,\,\,\,\,\textrm{and}\,\,\,\,\,\,\,\,\,\,
m_2(p):=\left\{\begin{array}{lcl}
\frac{(N-1)p}{N-p},\,\,\textrm{if}\,\,p<N;\\[1ex]
s\in(2,\infty),\,\,\textrm{if}\,\,p\geq N,\\[1ex]
\end{array}
\right.
\end{equation}
(for the infinity norm, it is only known in the interior for this case).
\item[(b)]\,\,\,Problem (\ref{arbitrary-1.06}) is uniquely solvable over $C([0,T];L^2(\Omega))\cap L^2((0,T),\mathbb{W}^{\,1}_{\,2}(\overline{\Omega}))$. Moreover,
given $\tilde{p}\in[2,\infty)$ fulfilling
$$\displaystyle\frac{1}{\kappa_{\tilde{p}}}+\displaystyle\frac{N}{\tilde{p}}<1,$$
if $\tilde{f}\in L^{\kappa_{\tilde{p}}}(0,T;L^{\tilde{p}}(\Omega))$, and $u_0\in L^{\infty}(\Omega)$, then
$$\|u\|_{_{L^{\infty}(T_1,T_2;L^{\infty}(\Omega))}}
\,\leq\,C\left(\|u\|_{_{2,\Omega}}+\|\tilde{f}\|_{_{L^{\kappa_{\tilde{p}}}(0,T;L^{\tilde{p}}(\Omega))}}\right),$$
for some constant $C>0$.
\end{enumerate}
\end{theorem}

\indent Theorem \ref{Arbitrary-Result}(a) was established by Daners \cite{DAN} for the local case with bounded measurable coefficients, and a nonlocal model
involving the boundary nonlocal boundary operator given in subsection \ref{subsec07-01} has been investigated in \cite{ALEJO-PhD10,WAR12}. The homogeneous local parabolic problem
of type (\ref{arbitrary-1.06}) with bounded coefficients has been investigated by Daners \cite{Daners,DAN}. Therefore, Theorem \ref{Arbitrary-Result} generalizes all
the regularity results in \cite{Daners,DAN,ALEJO-PhD10,WAR12} to the case of unbounded interior coefficients. On the other hand, a priori estimates for the inhomogeneous parabolic problem over arbitrary domains are developed for the first time here.

\begin{remark}\label{arbitrary-remark}
In general domains, the embedding $\mathbb{W}^{\,1}_{\,2}(\overline{\Omega})\hookrightarrow L^2_{_{\mathcal{H}^{N-1}}}(\Gamma)$ cannot be improved, and
consequently, although existence and uniqueness of a solution for problems (\ref{arbitrary-1.05}) and (\ref{arbitrary-1.06}) with inhomogeneous boundary conditions
can be achieved whenever $g\in L^q_{_{\mathcal{H}^{N-1}}}(\Gamma)$ for $q\geq2$, it is still an open problem to deduce $L^{\infty}$-estimates for weak solutions
of equations (\ref{arbitrary-1.05}) and (\ref{arbitrary-1.06}) with inhomogeneous boundary conditions, even if the boundary data $g$ is assumed to be smooth.
If the map $\mathbb{W}^{\,1}_{\,2}(\overline{\Omega})\hookrightarrow L^r_{_{\mathcal{H}^{N-1}}}(\Gamma)$ is continuous for some $r>2$, then one could obtain
a priori estimates for the case of inhomogeneous boundary conditions (under the suitable assumption on the boundary data $g$), but clearly there are multiple
domains with ``bad" boundaries in which such embedding fails to be satisfied.
\end{remark}

\subsection{Importance of boundary value problems over irregular domains}

\indent In the real world, many of the regions and surfaces where phenomena occur posses rough and irregular structure. In fact, when dealing with objects such as
clouds, trees, leaves, corals, snowflakes, blood vessels, and rocks (among many others), one cannot expect that such objects may have a structure compatible
with smooth manifolds (e.g. \cite{MAND82}). The interest in analysis and modeling of boundary value problems over domains with irregular boundaries such as
of fractal-like structure, began with the goal of understanding transmission problems, which, in electrostatics and magnetostatics, describe heat transfer
through a fractal-like interface (see for instance \cite{LAN-VER12-3,LAN-VER10-1} and the references therein). Furthermore, the regions showcased in
subsection \ref{subsec5-03} have been used to simulate the diffusion of medical sprays in the bronchial tree (e.g. \cite{ACH-DEH14,ACH-DEH12,ACH-SAB06-01,
ACH-SAB06-02,ACH10,ACH08,ACH07,DEH16,Silva-Perez-Velez-Santiago-2024}). It follows that in order to get an accurate model of diffusion, inhomogeneous Neumann or Robin boundary value problems
need to be considered (Cf. \cite{ACH08,Silva-Perez-Velez-Santiago-2024}). Other applications include plant leaf identification \cite{BCNM13}, oceanography \cite{MAND82}, and
probability theory \cite{BASS-CHEN08}. Further applications can be found in \cite{STRIC06}.

\section{Examples: Wentzell-type energy forms}\label{sec06}

\indent In this section, we will present several Wentzell-type bilinear forms, which give rise to several generalized Wentzell
boundary value problems in which all the results in the previous sections hold. We will assume that all the conditions in
Assumption \ref{As1} are satisfied.

\subsection{Wentzell energy over compact Riemannian manifolds}\label{subsec06-01}

\indent Let us assume that the boundary $\Gamma$ is a compact $(N-1)$-dimensional Riemannian manifold with Riemannian metric $\mathfrak{g}$,
for $N\geq3$. Then, we consider
\begin{equation}\label{Riemannian-Energy}
\Lambda_{_{\Gamma}}(u,v):=\displaystyle\int_{\Gamma}
\left(\displaystyle\sum^{N-1}_{^{i,j=1}}(\omega_{ij}(x)\partial_{x_{\tau_j}}u+
\hat{b}_i(x)u)\partial_{x_{\tau_i}}v+\displaystyle\sum^{N-1}_{^{i=1}}\check{b}_i(x)\partial_{x_{\tau_i}}uv\right)\,d\mu(\mathfrak{g}),
\end{equation}
for each $u,\,v\in D(\Lambda_{_{\Gamma}}):=H^1_{\mu(\mathfrak{g})}(\Gamma)$, where
$\partial_{x_{\tau_j}}u$ denotes the directional derivative of $u$ along the tangential directions $\tau_j$ at each point on $\Gamma$
($j\in\{1,\ldots,N-1\}$). Here $\omega_{ij}\in L^{\infty}_{\mu(\mathfrak{g})}(\Gamma)$,\,$\hat{b}_i\in L^{\hat{s}}_{\mu(\mathfrak{g})}(\Gamma)$,
\,$\check{b}_i\in L^{\check{s}}_{\mu(\mathfrak{g})}(\Gamma)$ ($i,\,j\in\{1,\ldots,N-1\}$), where
$\hat{s},\,\check{s}>(N-1)\chi_{_{\{N>3\}}}+\chi_{_{N=3}}$. We also assume that
a strongly ellipticity property holds, in the sense that there exists a constant $\eta'_0>$ such that
\begin{equation}
\label{ellipticity}\indent\indent\displaystyle\sum^{N-1}_{^{i,j=1}}\omega_{ij}(x)\xi_i\xi_j\,\geq\,\eta'_0\,|\xi|^2,
\,\,\,\,\,\forall\,x\in\Gamma,\,\,\,\forall\,\xi=(\xi_1,\ldots,\xi_{_{N-1}})\in\mathbb{R\!}^{N-1}.
\end{equation}
Observe that in general the form $\Lambda_{_{\Gamma}}(\cdot,\cdot)$ is non-symmetric. In fact, it follows that the bilinear form (\ref{3.1.04})
has the same structure as the interior form given in (\ref{3.1.03}). Then clearly conditions {\bf (B1)},
{\bf (B2)}, and {\bf (B3)}, are fulfilled. Moreover, since the embedding $H^1_{\mu(\mathfrak{g})}(\Gamma)\hookrightarrow L^{\bar{s}}_{\mu(\mathfrak{g})}(\Gamma)$
is compact whenever $1\leq\bar{s}<2(N-1)(N-3)^{-1}\chi_{_{\{N>3\}}}+\infty\chi_{_{\{N=3\}}}$, in the same way as in the derivation of Remark \ref{R-embedding}(d),
we get that $$\|u\|^2_{_{\bar{s},\Gamma}}\,\leq\,\epsilon\left\|\nabla_{_{\Gamma}}u\right\|^2_{_{2,\Gamma}}+C_{\epsilon}\|u\|^2_{_{2,\Gamma}}$$
for each $u\in H^1_{\mu(\mathfrak{g})}(\Gamma)$, for every $\epsilon>0$, and for some constant $C_{\epsilon}>0$. Consequently, proceeding in an analogous way
as in the proof of Proposition \ref{cont-coercive-form}, we deduce that the form $(\Lambda_{_{\Gamma}},D(\Lambda_{_{\Gamma}}))$ is continuous and weakly coercive.
Moreover, note that the operator $\mathcal{L}_{\mu}$ abstractly given by (\ref{1.03}) can be written explicitly by
\begin{equation}\label{L-explicit}
\mathcal{L}_{\mu}u:=-\displaystyle\sum^{N-1}_{^{i,j=1}}\partial_{x_{\tau_i}}\left(\omega_{ij}(x)\partial_{x_{\tau_j}}u+\hat{b}_i(x)u\right)
+\displaystyle\sum^{N-1}_{i=1}\check{b}_i(x)\partial_{x_{\tau_i}}u,
\end{equation}
for each $u\in H^1_{\mu(\mathfrak{g})}(\Gamma)$. We also recall that in this case, the bound improvements given in Remark \ref{improved-norms} hold, and in addition,
for this case we can assume that $\beta\in L^s_{\mu(\mathfrak{g})}(\Gamma)$ for $s>\frac{N-1}{2}$.\\
\indent A particular example of the above situation occurs when $\Omega$ is at least a bounded Lipschitz domain (see Example \ref{Ex5.01}). In this case,
it is well-known that the boundary $\Gamma$ is a compact Riemannian manifold, and $\mu(\mathfrak{g})=\mathcal{H}^{N-1}$, the
$(N-1)$-dimensional Hausdorff measure restricted to $\Gamma$, which in this case coincide $\mathcal{H}^{N-1}$-a.e. with the usual
Surface measure on $\Gamma$. Classical Wentzell-type problems have been investigated mostly over Lipschitz domains, but a boundary value problem including a boundary
operator of type (\ref{L-explicit}) is practically unknown in the literature. Recently in \cite{APU-NAZ-PAL-SOF22,APU-NAZ-PAL-SOF21}, the authors considered a Wentzell-type
problem (also known as Venttsel' problem) involving the operator $\mathcal{L}_{\mu}$ given by (\ref{L-explicit}) over a bounded domain of class $C^{1,1}$.
However, in their problem they assume that the coefficient $\omega_{ij}$ is symmetric, and that $\hat{b}_i\equiv0$. In our knowledge, these are the only
papers considering Wentzell boundary value problems whose structure are very similar to the one considered in this paper
(but less general in some sense than the one considered in (\ref{L-explicit})). A big generalization in \cite{APU-NAZ-PAL-SOF22,APU-NAZ-PAL-SOF21} comes from the fact that
in their problems they are considering the principal matrix coefficients $a_{ij}$ and $\omega_{ij}$ to be VMO (which include some unbounded cases). We will not go into further details here.

\subsection{The Wentzell problem on the Koch snowflake domain}\label{subsec06-02}

\indent We denote by $\vert P-P_0 \vert$ the Euclidean distance in $\mathbb{R}^2$ and the Euclidean balls by $B(P_0,r)=\lbrace P\in \mathbb{R}^2: |P-P_0|<r \rbrace$, $P_0 \in \mathbb{R}^2, r>0$. By the snowflake $\Gamma:=\partial\Omega$ we denote the union of three coplanar
Koch curves $K_i$ (see Example \ref{Ex5.02}, and \cite{FAL90}).
We assume that the junction points $A_1, A_3, A_5$ are the vertices of a regular triangle with unit side length, that is $|A_1-A_3|=|A_1-A_5|=|A_3-A_5|=1.$
$K_1$ is the uniquely determined self-similar set with respect to a family $\Psi^{1}$ of four suitable contractions $\psi_{1}^{(1)},...,\psi_{4}^{(1)}$, with respect to the same ratio $\frac{1}{3}$. Let $V_0^{(1)}:=\lbrace A_1,A_3 \rbrace$, $V_{j_1...j_h}^{(1)}:=\psi_{j_1}^{(1)}\circ ...\circ \psi_{j_h}^{(1)}(V_0^{(1)})$ and
\begin{center} $V_h^{(1)}:=\bigcup\limits_{j_1...j_h=1}^4 V_{j_1...j_h}^{(1)}$. \end{center}
We set $V_{\star}^{(1)}:=\cup_{h\geq 0}V_h^{(1)}$. It holds that $K_1=\overline{V_{\star}^{(1)}}$.\\
\indent In a similar way, it is possible to approximate $K_2,K_3$ by the sequences $(V_h^{(2)})_{h\geq 0}$, $(V_h^{(3)})_{h\geq 0}$, and denote their limits by $V_{\star}^{(2)},V_{\star}^{(3)}$. In order to approximate $\Gamma$, we define the increasing sequence of finite sets of points $\mathcal{V}_h:=\cup_{i=1}^3 V_h^{(i)}, h\geq 1$ and $\mathcal{V}_{\star}:=\cup_{h\geq 1}\mathcal{V}_h$. It holds that $\mathcal{V}_{\star}=\cup_{i=1}^3 V_{\star}^{(i)}$ and $\Gamma=\overline{V_{\star}}$.\\
\indent Moreover, on each $K_i$ there exists a unique Borel probability measure $\mu_i$ which is self-similar with respect to the family
$\Psi^{(i)}$. Therefore, one can define, in a natural way, a finite Borel measure $\mu_{_{\Gamma}}$ supported on $\Gamma$ by
\begin{equation} \label{eq1} \mu_{_{\Gamma}}:=\mu_1+\mu_2+\mu_3. \end{equation}
The measure $\mu_{_{\Gamma}}$ is a $d$-measure (see \cite{JO-WAL}), in the sense that there exist two positive constants $c_1,c_2$
$$\label{dmeasure}c_1 r^{d}\leq \mu_{_{\Gamma}}(B(P,r)\cap \Gamma)\leq c_2 r^d,\indent\textrm{for all}\,\,\,P\in \Gamma,$$
where $d=d_f=\frac{\log(4)}{\log(3)}.$\\
\indent In \cite{FRE-LAN04} the energy form $(\Lambda_{_{\Gamma}}^{},D(\Lambda_{_{\Gamma}}))$ on the Koch snowflake has been constructed.
$(\Lambda_{_{\Gamma}}^{},D(\Lambda_{_{\Gamma}}))$ is a nonnegative  energy functional in $L^2_{\mu}(\Gamma)$ and the
following result holds.

\begin{theorem}\,(see \cite{CapAcca40}) \label{fractals01}{\it The following properties hold.}
\begin{enumerate}\label{T5.1.2}
\item[(1)]\,\,\, $D(\Lambda_{_{\Gamma}})$ {\it is complete in the norm} $\|u\|_{_{D(\Lambda_{_{\Gamma}})}}:=
\|u\|_{_{2,\Gamma}}+\Lambda_{_{\Gamma}}(u,u)^{1/2}$.
\item[(2)]\,\,\, $D(\Lambda_{_{\Gamma}})$ {\it is dense in} $L^2_{\mu}(\Gamma)$.
\end{enumerate}
\end{theorem}

\indent It follows that $D(\Lambda_{_{\Gamma}})$ is injected in $L^2_{\mu}(\Gamma)$ and it is a Hilbert space with
the scalar product associated to the norm $\|\cdot\|_{\mathcal{D}(\Lambda_{_{\Gamma}})}$. Then
we extend $\Lambda_{_{\Gamma}}$ as usual on the completed space $D(\Lambda_{_{\Gamma}})$.
By $\Lambda_{_{\Gamma}}(\cdot,\cdot)$ we denote the bilinear form defined on $D\left(\Lambda_{_{\Gamma}}
\right)\times D\left(\Lambda_{_{\Gamma}}\right)$ by polarization, i.e.
$$\Lambda_{_{\Gamma}}(u,v):=\frac{1}{2}\left(\Lambda_{_{\Gamma}}[u+v]-\Lambda_{_{\Gamma}}[u]-\Lambda_{_{\Gamma}}[v] \right),
\quad u,v \in D\left(\Lambda_{_{\Gamma}}\right),$$
where $\Lambda_{_{\Gamma}}[w]:=\Lambda_{_{\Gamma}}(w,w)$. The form $\Lambda_{_{\Gamma}}(\cdot,\cdot)$ is regular and strongly local.
Moreover, the functions in $D(\Lambda_{_{\Gamma}})$ posses
continuous representatives, which are actually H\"{o}lder continuous
on $F$ (as the next result shows).

\begin{prop}\,(see \cite[Proposition 2.4 and Corollary 3.3]{LV99})\label{hold}
{\it The space $D(\Lambda_{_{\Gamma}})$ is continuously embedded into $\mathbb{B\!}^{\,2}_{\,\beta}(\Gamma,\mu)$
and $C^{0,\beta}(\Gamma)$, respectively, where $C^{0,\beta}(\Gamma)$ denotes the space of H\"{o}lder
continuous functions with exponent} $\beta=\frac{\ln 4}{2 \ln
3}=\frac{d_f}{2}$ over $\Gamma$.
\end{prop}

\indent In the following we identify $u \in D(\Lambda_{_{\Gamma}})$ with its continuous
representative, still denoted by $u$. By proceeding as in \cite{FRE-LAN04}, one can define on $\Gamma$ a Lagrangian $(\mathcal{L},D(\Lambda_{_{\Gamma}}))$
and an energy form $(\Lambda_{_{\Gamma}}, D(\Lambda_{_{\Gamma}}))$. More precisely, $\Lambda_{_{\Gamma}}(u,v)=\int_{\Gamma} d\mathcal{L}(u,v), $
for every $u,\,v\in D(\Lambda_{_{\Gamma}})$ where $$D(\Lambda_{_{\Gamma}})=\{ u\in C(\Gamma)\mid\, u|_{K_i}\in D(\Lambda_{_{\Gamma_i}}),\,\,\,
\textrm{for}\,\,\,i\in\{1,2,3\}\}.$$
In particular it holds that
\begin{equation}
\label{energyonF}
\Lambda_{_{\Gamma}}[u]= \displaystyle\sum_{i=1}^3\Lambda_{_{\Gamma_i}}i[u|_{K_i}].
\end{equation}
Moreover, since the form $(\Lambda_{_{\Gamma}}, D(\Lambda_{_{\Gamma}}))$ is a regular Dirichlet form on $L^2_{\mu}(\Gamma)$ with
domain $D(\Lambda_{_{\Gamma}})$ dense in $L^2_{\mu}(\Gamma)$,
by \cite[Chap. 6, Theorem 2.1]{KATO77}, there exists an unique self-adjoint, non-positive operator on $L^2_{\mu}(\Gamma)$, which we denote by
$\mathcal{L}_{\mu}$, with domain $D(\mathcal{L}_{\mu})\subseteq D(\Lambda_{_{\Gamma}})$, such that
\begin{equation}\label{Laplace-Koch}
\displaystyle\int_{\Gamma} d\mathcal{L}(u,v)=\Lambda_{_{\Gamma}}(u,v)=-\displaystyle\int_{\Gamma}(\mathcal{L}_{\mu}u)v\,d\mu,\indent\textrm{for all}
\,\,\,u\in D(\mathcal{L}_{\mu}),\,\,v\in D(\Lambda_{_{\Gamma}}).
\end{equation}

\indent From all the above constructions and known properties, one sees that
the form in \eqref{3.1.04} is symmetric and clearly satisfies {\bf (B1)} and  {\bf (B2)}. Furthermore,
 since $\mathcal{E}[u]$ is strongly local, one has {\bf (B3)} and a straightforward verification shows that conditions {\bf (B4)}
 and {\bf (B4)'} are fulfilled for $c^{\star}_1=1$, and $c^{\star}_2=c^{\star}_3=0$.

\subsection{Applications of Wentzell boundary value problems}\label{subsec06-03}

\indent There is a vast number of applications involving differential equations with Wentzell-type boundary conditions.
One can name few applications, such as in phase-transition phenomena,
fluid diffusion, heat flow subject to nonlinear cooling on the boundary, suspended transport energy,
fermentation, population dynamics, and climatology, among many others (see for example \cite{ARON-WIN78,DIAZ-TEL08,EVANS76-77,PAO92}).
Furthermore, Wentzell problems over fractal-like domains is valuable due to the fact that several natural and industrial
processes lead to the formation of rough surfaces or occur across them. Fractal boundaries and fractal layers may be of great interest for
those applications in which the surface effects are enhanced with respect to the surrounding volume
(see \cite{LAN-CEF-DELL12,Ce-La-Hd,Ce-La} for details and motivations).

\section{Examples: nonlocal operators}\label{sec07}

\indent In this part we give concrete examples of nonlocal operators $\mathcal{J}_{_{\Omega}}$ and $\Theta_{_{\Gamma}}$. Unless specified,
we will assume all the conditions of Assumption \ref{As1}.

\subsection{Besov-type nonlocal operators}\label{subsec07-01}

\indent We begin by considering the classical Besov-type operators, which are clear examples of nonlocal maps. These play a crucial role
in fractional differential equations, and thus could bring some sort of ``fractional flavor" into the boundary value problems
(\ref{1.05}).\\
\indent To begin, the interior and boundary Besov-type operators are defined by the maps $\mathcal{J}_{_{\Omega}}:H^{\mathfrak{s}/2}(\Omega)\rightarrow
H^{-\mathfrak{s}/2}(\Omega)$ and $\Theta_{_{\Gamma}}:H^{r_d/2}_{\mu}(\Gamma)\rightarrow H^{-r_d/2}_{\mu}(\Gamma)$, explicitly given by
\begin{equation}\label{Besov-Interior-Map}
(\mathcal{J}_{_{\Omega}}u)v:=\displaystyle\int_{\Omega}\int_{\Omega}\mathfrak{a}(x,y)\frac{(u(x)-u(y))(v(x)-v(y))}
{|x-y|^{2\mathfrak{s}+N}}\,dxdy,
\end{equation}
and
\begin{equation}\label{Besov-Boundary-Map}
(\Theta_{_{\Gamma}}u)v:=\displaystyle\int_{\Gamma}\int_{\Gamma}\mathfrak{b}(x,y)\frac{(u(x)-u(y))(v(x)-v(y))}{|x-y|^{2+2d-N}}\,d\mu_xd\mu_y,
\end{equation}
for $\mathfrak{s}\in(0,1)$, where $\mathfrak{a}\in L^{\infty}(\Omega\times\Omega,dx\oplus dy)$ and $\mathfrak{b}\in L^{\infty}(\Gamma\times\Gamma,d\mu_x\oplus d\mu_y)$
are both nonnegative almost everywhere in their respective domains. Then in this case
the nonlocal form $K(\cdot,\cdot)$ given by (\ref{nonlocal-term}) is explicitly written as

\begin{equation}\label{4.1.01}
K(u,v):=\displaystyle\int_{\Omega}\int_{\Omega}\mathfrak{a}\frac{(u(x)-u(y))(v(x)-v(y))}
{|x-y|^{2s+N}}\,dxdy+\displaystyle\int_{\Gamma}\int_{\Gamma}\mathfrak{b}\frac{(u(x)-u(y))(v(x)-v(y))}{|x-y|^{2+2d-N}}\,d\mu_xd\mu_y.
\end{equation}
Then a straightforward calculation shows that the conditions {\bf (A1)}, {\bf (A2)}, {\bf (A3)}, and {\bf (A3)'} are valid for this case.

\subsection{Dirichlet-to-Neumann map}\label{subsec07-02}

\indent We now consider a generalized notion of the classical Dirichlet-to-Neumann map, which is nonlocal
boundary operator. To begin, given $\gamma\in L^{\infty}(\Gamma)$ with $\displaystyle{\textrm{ess}}\inf_{^{x\in\partial\Omega}}
\gamma(x)\,\geq\,\gamma_0$ for some constant $\gamma_0>0$, let $\mathfrak{w}$ be the solution of the Dirichlet problem
\begin{equation}
\label{4.1.06}\left\{
\begin{array}{lcl}
-\textrm{div}(\gamma\nabla\mathfrak{w})\,=\,0\indent\,\textrm{in}\,\,\Omega\\
\,\,\,\mathfrak{w}|_{_{\Gamma}}\,=\,w\indent\indent\indent\,\,\,\,\,\textrm{on}\,\Gamma,\\
\end{array}
\right.
\end{equation}
for $w\in H^{r_d/2}_{\mu}(\Gamma)$ given. Then we define the {\it Dirichlet-to-Neumann map} $\mathcal{J}_{_{\Gamma}}:H^{r_d/2}_{\mu}(\Gamma)\rightarrow
H^{-r_d/2}_{\mu}(\Gamma)$ by
\begin{equation}\label{4.1.07}
\mathcal{J}_{_{\Gamma}}w:=\gamma\displaystyle\frac{\partial\mathfrak{w}}{\partial\nu_{\mu}},
\end{equation}
where the normal derivative here is interpreted in the generalized sense (see Definition \ref{Def-gen-normal} or \cite[Definition 1.1]{AVS-WAR}).
For properties of this nonlocal map on sufficiently smooth domains (that is, for at least bounded Lipschitz domains),
one can refer to \cite{AR-MAZZ12,JONN99,JONS97,SILV-UHLM} (and the references therein). On rough non-smooth domain, one can refer
to the work by Arendt and ter Elst \cite{AR-ELST11}.\\
\indent Before proceeding with the verification of the properties of this nonlocal operator, we first state a
general result due to Arendt and Elst (e.g. \cite[Theorem 3.2]{AR-ELST12}), which will be needed for the next important result.

\begin{theorem}\label{sectorial}
\,(see \cite{AR-ELST12}) \,Let $(a,D(a))$ be a sesquilinear form over a Hilbert space $H$, and let $j:D(a)\rightarrow H$ be a linear map.
If $a$ is $j$-sectorial (e.g. \cite[section 3, pag. 8]{AR-ELST12}) and $j(D(a))$ is dense in $H$, then the following conditions hold:
\begin{enumerate}
\item[(a)]\,\, There exists an operator $A$ in $H$ such that for all $x,\,y\in H$, one has $x\in D(A)$ and $Ax=y$ if and only if
there exists a sequence $\{u_n\}\subseteq D(a)$ such that $\displaystyle\lim_{n\rightarrow\infty}j(u_n)=x$ weakly in $H$,
$\displaystyle\sup_{n\in\mathbb{N\!}}\textrm{Re}\{a(u_n,u_n)\}<\infty$, and $\displaystyle\lim_{n\rightarrow\infty}a(u_n,v)=\langle y,j(v)\rangle_{H}$
for all $v\in D(a)$.
\item[(b)]\,\, The operator $A$ of Statement (a) is $m$-sectorial.
\item[(c)]\,\, Given $x,\,y\in H$, one has $x\in D(a)$ and $Ax=y$ if and only if there exists a Cauchy sequence $\{u_n\}\subseteq D(a)$
such that $\displaystyle\lim_{n\rightarrow\infty}j(u_n)=x$ in H and
$\displaystyle\lim_{n\rightarrow\infty}a(u_n,v)=\langle y,j(v)\rangle_{H}$ for all $v\in D(a)$.
\end{enumerate}
\end{theorem}

\indent Next, we establish the following
result, which is a sort of generalization (in the sense of more general measures) of \cite[Theorem 3.3]{AR-ELST11}.

\begin{theorem}\label{equiv-Dirichlet-to-Neumann}
The operator $\mathcal{J}_{_{\Gamma}}$ over $L^2_{\mu}(\Gamma)$ exists, such that the following conditions hold: given $\phi,\,\psi\in L^2_{\mu}(\Gamma)$,
one has $\phi\in D(\mathcal{J}_{_{\Gamma}})$ and $\mathcal{J}_{_{\Gamma}}\phi=\psi$ if and only if there exists $u\in H^1(\Omega)$ satisfying
\begin{enumerate}
\item[(a)]\,\, $\Delta u=0$;
\item[(b)]\,\, $\phi$ is the trace of $u$;
\item[(c)]\,\, $\gamma\displaystyle\frac{\partial u}{\partial\nu_{\mu}}\in L^2_{\mu}(\Gamma)$ and
$\gamma\displaystyle\frac{\partial u}{\partial\nu_{\mu}}=\psi$.
\end{enumerate}
\end{theorem}

\begin{proof}
The proof runs as in \cite[proof of Theorem 3.3]{AR-ELST11} with minor changes, but for the sake of completeness, we spell the details below.
Given $H:=L^2_{\mu}(\Gamma)$, consider the bilinear form $(a,D(a))$ given by: $D(a):=H^1(\Omega)\cap C(\overline{\Omega})$, and
$a(v,w):=\gamma\int_{\Omega}\nabla v\nabla w\,dx$. Sea $j:D(a)\rightarrow L^2_{\mu}(\Gamma)$, defined by $j(u):=u_{_{\Gamma}}$.
Then $j(D(a))$ is dense in $L^2_{\mu}(\Gamma)$, and $a(\cdot,\cdot)$ is sectorial (e.g. \cite[Lemma 3.1]{AR-ELST12}).
Thus, let $A$ be the operator associated with $(a,j)$ in the sense of Theorem \ref{sectorial}. We show that $A$ possesses the properties
of $\mathcal{J}_{_{\Gamma}}$. Indeed, let $\phi,\,\psi\in L^2_{\mu}(\Gamma)$, and assume first that $\phi\in D(A)$ and $A\phi=\psi$.
Then one can find a sequence $\{u_n\}\subseteq D(a)$ such that
\begin{equation}\label{equiv_DN-01}
\displaystyle\lim_{n,\,m\rightarrow\infty}\|\nabla(u_n-u_m)\|^2_{_{2,\Omega}}=0,\,\,\,\,\,\,\,\,\,\,\,\displaystyle\lim_{n\rightarrow\infty}u_n|_{_{\Gamma}}=\phi
\,\,\,\textrm{in}\,\,L^2_{\mu}(\Gamma),
\end{equation}
and
\begin{equation}\label{equiv_DN-02}
\displaystyle\lim_{n\rightarrow\infty}\gamma\displaystyle\int_{\Omega}\nabla u_n\nabla v\,dx=\displaystyle\int_{\Gamma}\psi v\,d\mu,\,\,\,\,\,\,\,\,
\textrm{for every}\,\,v\in D(a).
\end{equation}
Recalling Remark \ref{R-embedding}(b), one sees that $\{u_n\}\subseteq H^1(\Omega)$ is a Cauchy sequence, and thus
$\displaystyle\lim_{n\rightarrow\infty}u_n=u$ in $H^1(\Omega)$. Moreover, by (\ref{equiv_DN-02}) we see that
\begin{equation}\label{equiv_DN-03}
\gamma\displaystyle\int_{\Omega}\nabla u\nabla v\,dx=\displaystyle\int_{\Gamma}\psi v\,d\mu,\,\,\,\,\,\,\,\,
\textrm{for all}\,\,v\in D(a).
\end{equation}
Taking $v\in C^{\infty}_c(\Omega)$ in (\ref{equiv_DN-03}), we see that $\Delta u=0$, and consequently,
\begin{equation}\label{equiv_DN-04}
\gamma\displaystyle\int_{\Omega}(\Delta u)v\,dx+\gamma\displaystyle\int_{\Omega}\nabla u\nabla v\,dx=\displaystyle\int_{\Gamma}\psi v\,d\mu,\,\,\,\,\,
\textrm{for every}\,\,v\in D(a).
\end{equation}
Integrating by parts (in the generalized way), we find that $\gamma\displaystyle\frac{\partial u}{\partial\nu_{\mu}}\in L^2_{\mu}(\Gamma)$ and
$\gamma\displaystyle\frac{\partial u}{\partial\nu_{\mu}}=\psi$, as desired.\\
\indent Conversely, suppose that there exists $u\in H^1(\Omega)$ fulfilling the conditions in (a),\,(b),\,(c) of the theorem.
Then by density, there exists $\{u_n\}\subseteq D(a)$ such that $\displaystyle\lim_{n\rightarrow\infty}u_n=u$ in $H^1(\Omega)$ and
$\displaystyle\lim_{n\rightarrow\infty}u_n|_{_{\Gamma}}=\phi$ in $L^2_{\mu}(\Gamma)$. Thus, we have that
$\displaystyle\lim_{n,\,m\rightarrow\infty}a(u_n-u_m,u_n-u_m)=0$, and using the fact that $\Delta u=0$ together with the definition of
$\displaystyle\frac{\partial u}{\partial\nu_{\mu}}$, we find that
\begin{equation}\label{equiv_DN-05}
\displaystyle\lim_{n\rightarrow\infty}a(u_n,v)=\displaystyle\lim_{n\rightarrow\infty}\gamma\displaystyle\int_{\Omega}\nabla u_n\nabla v\,dx
=\gamma\displaystyle\int_{\Omega}(\Delta u)v\,dx+\gamma\displaystyle\int_{\Omega}\nabla u\nabla v\,dx=\displaystyle\int_{\Gamma}\psi v\,d\mu,\,\,\,\,\,
\textrm{for all}\,\,v\in D(a).
\end{equation}
Equality (\ref{equiv_DN-05}) shows that $\phi\in D(A)$ and $A\phi=\psi$. In particular, this shows that the operator with the properties
of $\mathcal{J}_{_{\Gamma}}$ is well defined, and equals $A$. This completes the proof of the theorem.
\end{proof}

\indent Next, for $u,\,v\in\mathcal{W}_2(\Omega;\Gamma)$, we define concretely the nonlocal form $K(\cdot,\cdot)$ by
\begin{equation}\label{4.1.08}
K(u,v):=\langle\mathcal{J}_{_{\Gamma}}u,v\rangle=\displaystyle\int_{\Gamma}(\mathcal{J}_{_{\Gamma}}u)v\,d\mu,
\end{equation}
for $\mathcal{J}_{_{\Gamma}}$ given by (\ref{4.1.07}). Notice that we are assuming that $\mathcal{J}_{_{\Gamma}}u\in L^2_{\mu}(\Gamma)$ in the weak sense
(e.g. Theorem \ref{equiv-Dirichlet-to-Neumann} and Definition \ref{a-normal-a}(c)). Then $K(\cdot,\cdot)$ is symmetric and positive (e.g. \cite[Theorem 4.1]{JONS97}).
We now show the validity of {\bf (A1)} for this case. In fact,
given $u\in\mathcal{W}_2(\Omega;\Gamma)$, let $\mathfrak{u}$ be the solution of (\ref{4.1.06}) for
$g\equiv u$. Then we let $v_{_m}$ be defined by (\ref{3.2.02}), and consider the solution $\mathfrak{v}_{_m}$ of problem
(\ref{4.1.06}) for $g\equiv v_{_m}$. Then one has
$$\mathfrak{v}_{_m}|_{_{\Gamma}}=v_{_m}|_{_{\Gamma}}=\psi_{_{k-1,m}}(u)|_{_{\Gamma}}=\psi_{_{k-1,m}}(\mathfrak{u})|_{_{\Gamma}}\,\,\,
\mu\textrm{-a.e. on}\,\,\Gamma,$$
where we recall the definition of the function $\psi_{_{t,m}}(\cdot)$ given in (\ref{3.2.01}).
Then, in views of (\ref{4.1.07}) and (\ref{4.1.08}), we get the following calculation.\\[2ex]
$K(u,v_{_m})=\displaystyle\int_{\Gamma}\gamma v_{_m}\frac{\partial\mathfrak{u}}{\partial\nu_{\mu}}\,d\mu
=\displaystyle\int_{\Gamma}\gamma\psi_{_{k-1,m}}(\mathfrak{u})\frac{\partial\mathfrak{u}}{\partial\nu_{\mu}}\,d\mu
=\displaystyle\int_{\Omega}\gamma\nabla\psi_{_{k-1,m}}(\mathfrak{u})\nabla\mathfrak{u}\,dx$\\
$$\indent\indent\indent\indent\,\geq\,m
\displaystyle\int_{\Omega\cap\{\mathfrak{u}\geq m\}}|\nabla\mathfrak{u}|^2\,dx+
(k-1)\displaystyle\int_{\Omega\cap\{\mathfrak{u}<m\}}|\mathfrak{u}|^{k-2}|\nabla\mathfrak{u}|^2\,dx\,\geq\,0$$
(recall that $k\geq2$). This shows that {\bf (A1)} holds for $\eta_0=0$. Moreover, a calculation similar as the one carried out in
\cite[Theorem 4.1]{AR-MAZZ12} yields that $K(u^+,u^-)\leq0$, establishing {\bf (A2)}.\\
\indent To establish the validity of {\bf (A3)}, again given $u\in\mathcal{W}_2(\Omega;\Gamma)$, let $\mathfrak{u}$ be the solution of (\ref{4.1.06}) for
$g\equiv u|_{_{\Gamma}}$. Then one has that $u_k|_{_{\Gamma}}=(\mathfrak{u}|_{_{\Gamma}}-k)^+:=\mathfrak{u}_k|_{_{\Gamma}}$, and consequently
$$K(u,u_{k})=\displaystyle\int_{\Gamma}\gamma u_{k}\frac{\partial\mathfrak{u}}{\partial\nu_{\mu}}\,d\mu=
\displaystyle\int_{\Gamma}\gamma\mathfrak{u}_k\frac{\partial\mathfrak{u}}{\partial\nu_{\mu}}\,d\mu=
\displaystyle\int_{\widehat{\Omega}_k}\gamma\nabla\mathfrak{u}\nabla\mathfrak{u}\,dx\,\geq\,0,$$
where $\widehat{\Omega}_k:=\{x\in\Omega\mid\mathfrak{u}(x)\geq k\}$. We have shown the fulfillment of property {\bf (A3)}. Analogously, one easily sees that
condition {\bf (A3)'} is valid, and therefore all the results on this paper hold for the Dirichlet-to-Neumann map given by (\ref{4.1.07}).

\subsection{Further remarks and results for nonlocal maps}\label{subsec07-03}

\begin{remark}
Another class of nonlocal (boundary operators) has been investigated by Gesztezy and Mitrea \cite{GES-MITR09}.
However, it may not be possible to show the fulfillment of the conditions {\bf (A1)} and {\bf (A2)} without giving more concrete
assumptions to the nonlocal operator presented in \cite{GES-MITR09}. On the other hand, it may open the door to consider more classes of
nonlocal maps, and applications. We will no go further here.
\end{remark}

\indent We now establish a relationship between boundary value problems with different nonlocal maps.\\

\begin{theorem}
Let $\widetilde{\mathcal{J}}_{_{\Omega}}:H^{s/2}(\Omega)\rightarrow H^{-s/2}(\Omega)$ and
$\widetilde{\Theta}_{_{\Gamma}}:H^{r_d/2}_{\mu}(\Gamma)\rightarrow H^{-r_d/2}_{\mu}(\Gamma)$ be linear bounded (nonlocal)
operators, such that $\widetilde{\mathcal{J}}_{_{\Omega}}1=\widetilde{\Theta}_{_{\Gamma}}1=0$,\,
$\langle\widetilde{\mathcal{J}}_{_{\Omega}}u,1\rangle_{s/2}=0=
\langle\widetilde{\Theta}_{_{\Gamma}}v,1\rangle_{r_d/2}$, \,$\langle\widetilde{\mathcal{J}}_{_{\Omega}}u,u\rangle_{s/2}\geq0$,
and $\langle\widetilde{\Theta}_{_{\Gamma}}v,v\rangle_{r_d/2}\geq0$,
for each $u\in H^{s/2}(\Omega)$ and $v\in H^{r_d/2}_{\mu}(\Gamma)$.  Given $u\in\mathcal{W}_2(\Omega;\Gamma)$ a weak solution of (\ref{1.05}),
let $\tilde{u}\in\mathcal{W}_2(\Omega;\Gamma)$ be a weak solution of (\ref{1.05}) related to the maps $\widetilde{\mathcal{J}}_{_{\Omega}}$
and $\widetilde{\Theta}_{_{\Gamma}}$. Then there exists a constant $C>0$ such that
\begin{equation}\label{4.1.13}
\|u-\tilde{u}\|_{_{\mathcal{W}_2(\Omega;\Gamma)}}\,\leq\,C\,\left(\|\mathcal{J}_{_{\Omega}}u-\widetilde{\mathcal{J}}_{_{\Omega}}u\|
_{_{H^{-s/2}(\Omega)}}+\|\Theta_{_{\Gamma}}u-\widetilde{\Theta}_{_{\Gamma}}u\|_{_{H^{-r_d/2}_{\mu}(\Gamma)}}+\|u-\tilde{u}\|_{_{2,\Omega}}\right).
\end{equation}
Moreover, if any of the conditions in Remark \ref{coercivity} hold, then
\begin{equation}\label{4.1.14}
\|u-\tilde{u}\|_{_{\mathcal{W}_2(\Omega;\Gamma)}}\,\leq\,C'\,\left(\|\mathcal{J}_{_{\Omega}}u-\widetilde{\mathcal{J}}_{_{\Omega}}u\|
_{_{H^{-s/2}(\Omega)}}+\|\Theta_{_{\Gamma}}u-\widetilde{\Theta}_{_{\Gamma}}u\|_{_{H^{-r_d/2}_{\mu}(\Gamma)}}\right),
\end{equation}
for some constant $C'>0$.\\
\end{theorem}

\begin{proof}
Clearly one has $\mathcal{E}_{\mu}(u-\tilde{u},\varphi)=0$ for all $\varphi\in\mathcal{W}_2(\Omega;\Gamma)$
(recall the definition of $\mathcal{E}_{\mu}(\cdot,\cdot)$ given in (\ref{3.1.01})). Moreover, since the (nonlocal) maps are linear, we see that
$$\mathcal{J}_{_{\Omega}}u-\widetilde{\mathcal{J}}_{_{\Omega}}\tilde{u}=(\mathcal{J}_{_{\Omega}}u-\widetilde{\mathcal{J}}_{_{\Omega}}u)+
\widetilde{\mathcal{J}}_{_{\Omega}}(u-\tilde{u});$$ the same property holds for $\Theta_{_{\Gamma}}$ and $\widetilde{\Theta}_{_{\Gamma}}$.
Using these facts, rearranging, applying Young's inequality, using the positivity of both $\widetilde{\mathcal{J}}_{_{\Omega}},\,
\widetilde{\Theta}_{_{\Gamma}}$, applying \cite[Lemma 4.2 and 4.3]{GES-MITR09} to the $L^2$-boundary terms,
taking into account (\ref{1.04}), and using (\ref{2.04b}) and (\ref{2.05}), by selecting $\varphi:=u-\tilde{u}$ one has\\[2ex]
$\min\{c_0,1\}\|u-\tilde{u}\|^2_{_{\mathcal{W}_2(\Omega;\Gamma)}}$\\
$$\leq\,-\displaystyle\int_{\Omega}
\left(\displaystyle\sum^N_{^{j=1}}\hat{a}_j(x)(u-\tilde{u})\partial_j(u-\tilde{u})+
\displaystyle\sum^N_{^{i=1}}\check{a}_i(x)\partial_i(u-\tilde{u})(u-\tilde{u})+\lambda(x)|u-\tilde{u}|^2\right)\,dx+$$
$$-\Lambda_{_{\Gamma}}(u-\tilde{u},u-\tilde{u})+c\|u-\tilde{u}\|^2_{_{2,\Gamma}}-\displaystyle\int_{\Gamma}\beta(x)|u-\tilde{u}|^2\,d\mu+\indent\indent\indent$$
$$-\left(\left\langle\mathcal{J}_{_{\Omega}}u-\widetilde{\mathcal{J}}_{_{\Omega}}u,u-\tilde{u}\right\rangle_{s/2}+
\left\langle\widetilde{\mathcal{J}}_{_{\Omega}}(u-\tilde{u}),u-\tilde{u}\right\rangle_{s/2}+
\left\langle\Theta_{_{\Gamma}}u-\widetilde{\Theta}_{_{\Gamma}}u,u-\tilde{u}\right\rangle_{r_d/2}\right)+$$
$$\indent\indent\indent-\left\langle\widetilde{\Theta}_{_{\Gamma}}(u-\tilde{u}),u-\tilde{u}\right\rangle_{r_d/2}+
\min\{c_0,1\}\left\|\textbf{u}-\tilde{\textbf{u}}\right\|^2_{_{\mathbb{X\!}^{\,2}(\Omega;\Gamma)}}$$
$$\leq\,\epsilon\left\|\nabla(u-\tilde{u})\right\|^2_{_{2,\Omega}}+
C^{\ast}_{\epsilon}(\delta^{\ast}+c+\min\{c_0,1\})\|u-\tilde{u}\|^2_{_{2,\Omega}}+\indent\indent\indent\indent$$
$$\indent+\|\mathcal{J}_{_{\Omega}}u-\widetilde{\mathcal{J}}_{_{\Omega}}u\|_{_{H^{-s/2}(\Omega)}}\|u-\tilde{u}\|_{_{H^{s/2}(\Omega)}}+
\|\Theta_{_{\Gamma}}u-\widetilde{\Theta}_{_{\Gamma}}u\|_{_{H^{-r_d/2}_{\mu}(\Gamma)}}\|u-\tilde{u}\|_{_{H^{r_d/2}_{\mu}(\Gamma)}}$$
$$\leq\,3\epsilon\left\|\nabla(u-\tilde{u})\right\|^2_{_{2,\Omega}}+
\{C^{\ast}_{\epsilon}(\delta^{\ast}+c+\min\{c_0,1\})+C'_{\epsilon}\}\|u-\tilde{u}\|^2_{_{2,\Omega}}+$$
$$\indent\indent\indent\indent\indent\indent+\|\mathcal{J}_{_{\Omega}}u-\widetilde{\mathcal{J}}_{_{\Omega}}u\|^2_{_{H^{-s/2}(\Omega)}}+
\|\Theta_{_{\Gamma}}u-\widetilde{\Theta}_{_{\Gamma}}u\|^2_{_{H^{-r_d/2}_{\mu}(\Gamma)}},$$
for all $\epsilon>0$, for some positive constants $c,\,C^{\ast}_{\epsilon},\,C'_{\epsilon}$, and where the constant $\delta^{\ast}$
is given by (\ref{3.2.10}). By selecting $\epsilon>0$ suitably, from the above calculation we deduce the validity of (\ref{4.1.13}).
Furthermore, if any of the conditions in Remark \ref{coercivity} occur, then one gets immediately from the above estimates that
(\ref{4.1.14}) is valid, completing the proof.
\end{proof}

\indent To conclude this section, we will establish $L^{\infty}$-type a priori estimates for elliptic and parabolic problems of type (\ref{1.05}) and
(\ref{1.06}), respectively, but now for the case where we have nonlocal interior and boundary data. To be more precise, consider the solvability of the
problems
\begin{equation}
\label{1.05B}\left\{
\begin{array}{lcl}
\mathcal{A}u\,=\,\mathcal{J}^{_{B}}_{_{\Omega}}f(x)\indent\,\textrm{in}\,\,\Omega,\\
\mathcal{B}u\,=\,\Theta^{_{B}}_{_{\Gamma}}g(x)\indent\,\textrm{on}\,\,\Gamma,\\
\end{array}
\right.
\end{equation}
and
\begin{equation}
\label{1.06B}\left\{
\begin{array}{lcl}
u_t-\mathcal{A}u\,=\,\mathcal{J}^{_{B}}_{_{\Omega}}\tilde{f}(t,x)\indent\,\textrm{in}\,\,(0,\infty)\times\Omega,\\
\,\,\,\,\mathcal{B}u\,\,\,\,\,\,\,\,=\,\Theta^{_{B}}_{_{\Gamma}}\tilde{g}(t,x)\indent\,\textrm{on}\,\,(0,\infty)\times\Gamma,\\
\,u(0,x)=u_0(x)\,\,\,\,\,\,\,\,\,\,\,\,\,\,\,\,\,\,\,\,\,\,\textrm{in}\,\,\Omega,
\end{array}
\right.
\end{equation}
where $\mathcal{J}^{_{B}}_{_{\Omega}}$ and $\Theta^{_{B}}_{_{\Omega}}$ denote the nonlocal operators defined by (\ref{Besov-Interior-Map}) and
(\ref{Besov-Boundary-Map}) for $\mathfrak{a}=\mathfrak{b}=1$, respectively,  and where in the denominator $N$ and $2d-N$ are replaced by $2N/\tilde{p}$ and $(4d-2N)/\tilde{q}$, respectively (for corresponding values of $\tilde{p},\,\tilde{q}\in[2,\infty)$). Then our main result regarding this situation is the following.

\begin{theorem}\label{Nonlocal-estimate}
Given $T>0$ a fixed constant, let $f\in H^{\mathfrak{s}/2}(\Omega)$, \,$g\in H^{r_d/2}_{\mu}(\Gamma)$,\, $\tilde{f}\in L^2(0,T;H^{\mathfrak{s}/2}(\Omega))$, and
$\tilde{g}\in L^2(0,T;H^{r_d/2}_{\mu}(\Gamma))$. Write $\eta_{\tilde{q}}:=1-(N-d)/\tilde{q}$. Then the following hold.
\begin{enumerate}
\item[(a)]\,\,\,Problem (\ref{1.05B}) is uniquely solvable over $\mathcal{W}_2(\Omega;\Gamma)$. Moreover, if $f\in\mathbb{B\!}^{\,\tilde{p}}_{\,\mathfrak{s}}(\Omega)$
for $\tilde{p}>N$, and $g\in\mathbb{B\!}^{\,\tilde{q}}_{_{\,\eta_{\tilde{q}}}}(\Gamma,\mu)$ for $\tilde{q}>2d(2+d-N)^{-1}$, then $u$ is globally bounded with
$$|\|\mathbf{u}\||_{_{\infty}}\,\leq\,C\left(\mathcal{M}^{\tilde{p}}_{\mathfrak{s}}(f,\Omega)+
\mathcal{N}^{\tilde{q}}_{\eta_{\tilde{q}}}(g,\Gamma,\mu)+\|u\|_{_{2,\Omega}}\right),$$
for some constant $C>0$, where we recall that $\mathcal{M}^{\tilde{p}}_{\mathfrak{s}}(f,\Omega)$ and $\mathcal{N}^{\tilde{q}}_{\eta_{\tilde{q}}}(g,\Gamma,\mu)$ are defined in
(\ref{nonlocal-interior}) and (\ref{nonlocal-boundary}), respectively. Furthermore, if any of the conditions in Remark \ref{R-embedding} holds, then one has that
$$|\|\mathbf{u}\||_{_{\infty}}\,\leq\,C\left(\mathcal{M}^{\tilde{p}}_{\mathfrak{s}}(f,\Omega)+
\mathcal{N}^{\tilde{q}}_{\eta_{\tilde{q}}}(g,\Gamma,\mu)\right).$$
\item[(b)]\,\,\,Given $\tilde{p},\,\tilde{q}\in[2,\infty)$ fulfilling
$$\displaystyle\frac{1}{\kappa_{\tilde{p}}}+\displaystyle\frac{N}{\tilde{p}}<1,\,\,\,\,\,\,\,\,\,\,\,\,\,\,\,\,\textrm{and}\,\,\,\,\,\,\,\,
\,\,\,\,\,\,\,\,\displaystyle\frac{1}{\kappa_{\tilde{q}}}+\displaystyle\frac{d}{\tilde{q}(2+d-N)}<\frac{1}{2},$$
if $\tilde{f}\in L^{\kappa_{\tilde{p}}}(0,T;\mathbb{B\!}^{\,\tilde{p}}_{\,\mathfrak{s}}(\Omega))$,\, $\tilde{g}\in L^{\kappa_{\tilde{q}}}
(0,T;\mathbb{B\!}^{\,\tilde{q}}_{_{\,\eta_{\tilde{q}}}}(\Gamma,\mu))$, $u_0\in L^{\infty}(\Omega)$, and if $u\in C([0,T];L^2(\Omega))\cap L^2((0,T), \mathcal{W}_2(\Omega;\Gamma))$ is a weak solution to problem \eqref{1.06B}, then
$$\|u\|_{_{L^{\infty}(0,T;L^{\infty}(\Omega))}}
\,\leq\,C\left(\|u_0\|_{_{\infty,\Omega}}+\left\{\displaystyle\int^T_0\left[\mathcal{M}^{\tilde{p}}_{\mathfrak{s}}(\tilde{f}(t),\Omega)\right]^
{\kappa_{\tilde{p}}}dt\right\}^{1/\kappa_{\tilde{p}}}+\left\{\displaystyle\int^T_0\left[\mathcal{N}^{\tilde{q}}_{\eta_{\tilde{q}}}(\tilde{g}(t),\Gamma,\mu)\right]^
{\kappa_{\tilde{q}}}dt\right\}^{1/\kappa_{\tilde{q}}}\right),$$
for some constant $C>0$. Additionally, all the results in subsection \ref{subsec04-03} can be adapted to this case.
\end{enumerate}
\end{theorem}

\begin{proof}
It suffices to verify this case for the parabolic problem (which is the hardest case). Furthermore, the only key result required to be
established is the validity of Theorem \ref{Thm-Partial-Bounded}; in particular, the derivation of the results in the first step of the proof for this
particular case. Select in this case
$$\hat{k}^2_0:=\left\{\displaystyle\int^T_0\left[\mathcal{M}^{\tilde{p}}_{\mathfrak{s}}(\tilde{f}(t),\Omega)\right]^
{\kappa_{\tilde{p}}}dt\right\}^{2/\kappa_{\tilde{p}}}+\left\{\displaystyle\int^T_0\left[\mathcal{N}^{\tilde{q}}_{\eta_{\tilde{q}}}(\tilde{g}(t),\Gamma,\mu)\right]^
{\kappa_{\tilde{q}}}dt\right\}^{2/\kappa_{\tilde{q}}},$$
where we recall that $\eta_{\tilde{q}}:=1-(N-d)/\tilde{q}$.
Then, for $k\geq\hat{k}_0$, after a verification of the behavior of the nonlocal integral representations for
$(\mathcal{J}^{_{B}}_{_{\Omega}}\tilde{f}(t,x))u_k$ and $(\Theta^{_{B}}_{_{\Gamma}}\tilde{g}(t,x))u_k$,
proceeding in the exact way as in the proof of the first step in the proof, we arrive at the following inequality:\\[2ex]
$\Upsilon^2_{\varsigma}(\psi u_k)\,\leq\,c\left\{L_{\psi'}\displaystyle\int^0_{-\varsigma}\|u_k(t)\|^2_{_{2,\Omega_k(t)}}\,dt\right.+$\\[2ex]
$+k^2\left(\displaystyle\int^0_{-\varsigma}\left(\displaystyle\int_{\Omega_k(t)}
\int_{\Omega_k(t)}\frac{|\tilde{f}(t,x)-\tilde{f}(t,y)|^2}{k^2|x-y|^{2\mathfrak{s}+2N/\tilde{p}}}\,dxdy+\right.\right.$\\
$$+\left.\displaystyle\int_{\Omega_k(t)}\int_{\Omega\setminus\Omega_k(t)}\frac{|\tilde{f}(t,x)-\tilde{f}(t,y)|^2}{k^2|x-y|^{2\mathfrak{s}+2N/\tilde{p}}}\,dxdy
+\displaystyle\int_{\Omega\setminus\Omega_k(t)}\int_{\Omega_k(t)}\frac{|\tilde{f}(t,x)-\tilde{f}(t,y)|^2}{k^2|x-y|^{2\mathfrak{s}+2N/\tilde{p}}}\,dxdy\right)dt+$$\\
$\,\,\,\,\,+\displaystyle\int^0_{-\varsigma}\left(
\displaystyle\int_{\Gamma_k(t)}\int_{\Gamma_k(t)}\frac{|\tilde{g}(t,x)-\tilde{g}(t,y)|^2}{k^2|x-y|^{2+(4d-2N)/\tilde{q}}}\,d\mu_xd\mu_y+\right.$\\
\begin{equation}\label{NL01}
+\left.\left.\left.\displaystyle\int_{\Gamma_k(t)}\int_{\Gamma\setminus\Gamma_k(t)}\frac{|\tilde{g}(t,x)-\tilde{g}(t,y)|^2}{k^2|x-y|^{2+(4d-2N)/\tilde{q}}}\,d\mu_xd\mu_y
+\displaystyle\int_{\Gamma\setminus\Gamma_k(t)}\int_{\Gamma_k(t)}\frac{|\tilde{g}(t,x)-\tilde{g}(t,y)|^2}{k^2|x-y|^{2+(4d-2N)/\tilde{q}}}\,d\mu_xd\mu_y\right)dt\right)\right\},
\end{equation}
for some constant $c>0$ (which can vary from line to line), where as before $L_{\psi'}:=\displaystyle\sup_{-\varsigma\leq t\leq0}|\psi'(t)|$.
Estimating the nonlocal-type data terms, we calculate and obtain that\\[2ex]
$\displaystyle\int^0_{-\varsigma}\left\{\displaystyle\int_{\Omega_k(t)}\int_{\Omega_k(t)}\frac{|\tilde{f}(t,x)-\tilde{f}(t,y)|^2}{k^2|x-y|^{2\mathfrak{s}+2N/\tilde{p}}}\,dxdy+
\displaystyle\int_{\Omega_k(t)}\int_{\Omega\setminus\Omega_k(t)}\frac{|\tilde{f}(t,x)-\tilde{f}(t,y)|^2}{k^2|x-y|^{2\mathfrak{s}+2N/\tilde{p}}}\,dxdy+\right.$\\
$$\,\,\,\,\,\,\,\,\,\,\,\,\,\,\,\,\,\,\,\,\,\,\,\,\,\,\,\,\,\,\,\,\,\,\,\,\,\,\,\,\,\,\,\,\,\,\,\,
\left.+\displaystyle\int_{\Omega\setminus\Omega_k(t)}\int_{\Omega^{\pm}_k(t)}\frac{|\tilde{f}(t,x)-\tilde{f}(t,y)|^2}{k^2|x-y|^{2\mathfrak{s}+2N/\tilde{p}}}\,dxdy\right\}dt$$
$$\leq\,\displaystyle\frac{3}{k^2}|\Omega|^{1/\breve{\tilde{p}}}\left\{\displaystyle\int^T_0\left[\mathcal{M}^{\tilde{p}}_{\mathfrak{s}}(\tilde{f}(t),\Omega)\right]^
{\kappa_{\tilde{p}}}dt\right\}^{1/\kappa_{\tilde{p}}}
\left(\displaystyle\int^0_{-\varsigma}|\Omega_k(t)|^{\breve{\kappa}_{\tilde{p}}/\breve{\tilde{p}}}dt\right)^{1/\breve{\kappa}_{\tilde{p}}}\,
\,\,\,\,\,\,\,\,\,\,\,\,\,\,$$
\begin{equation}\label{NL02}
\,\,\,\,\,\,\,\,\,\,\,\,\,\,\,\,\,\,\leq\,
3|\Omega|^{1/\breve{\tilde{p}}}\left(\displaystyle\int^0_{-\varsigma}|\Omega_k(t)|^{\breve{\kappa}_{\tilde{p}}
/\breve{\tilde{p}}}dt\right)^{1/\breve{\kappa}_{\tilde{p}}},
\end{equation}
and\\[2ex]
$\displaystyle\int^0_{-\varsigma}\left\{\displaystyle\int_{\Gamma_k(t)}\int_{\Gamma_k(t)}\frac{|\tilde{g}(t,x)-\tilde{g}(t,y)|^2}{k^2|x-y|^{2+(4d-2N)/\tilde{q}}}\,d\mu_xd\mu_y+
\displaystyle\int_{\Gamma_k(t)}\int_{\Gamma\setminus\Gamma_k(t)}\frac{|\tilde{g}(t,x)-\tilde{g}(t,y)|^2}{k^2|x-y|^{2+(4d-2N)/\tilde{q}}}\,d\mu_xd\mu_y+\right.$\\
$$\,\,\,\,\,\,\,\,\,\,\,\,\,\,\,\,\,\,\,\,\,\,\,\,\,\,\,\,\,\,\,\,\,\,\,\,\,\,\,\,\,\,\,\,\,\,\,\,
\left.+\displaystyle\int_{\Gamma\setminus\Gamma_k(t)}\int_{\Gamma_k(t)}\frac{|\tilde{g}(t,x)-\tilde{g}(t,y)|^2}{k^2|x-y|^{2+(4d-2N)/\tilde{q}}}\,d\mu_xd\mu_y\right\}d\xi,$$
$$\leq\,\displaystyle\frac{3}{k^2}\mu(\Gamma)^{1/\breve{\tilde{q}}}\left\{\displaystyle\int^T_0\left[\mathcal{N}^{\tilde{q}}_{\eta_{\tilde{q}}}(\tilde{g}(t),\Gamma,\mu)
\right]^{\kappa_{\tilde{q}}}dt\right\}^{1/\kappa_{\tilde{q}}}
\left(\displaystyle\int^0_{-\varsigma}\mu(\Gamma_k(t))^{\breve{\kappa}_{\tilde{q}}/\breve{\tilde{q}}}dt\right)^{1/\breve{\kappa}_{\tilde{q}}}
\,\,\,\,\,\,\,\,\,\,\,\,\,\,$$
\begin{equation}\label{NL03}
\,\,\,\,\,\,\,\,\,\,\,\,\,\,\,\,
\leq\,3\mu(\Gamma)^{1/\breve{\tilde{q}}}\left(\displaystyle\int^0_{-\varsigma}\mu(\Gamma_k(t))^{\breve{\kappa}_{\tilde{q}}
/\breve{\tilde{q}}}dt\right)^{1/\breve{\kappa}_{\tilde{q}}},
\end{equation}
where $\breve{\tau'}:=2\tau'(\tau'-2)^{-1}$ for $\tau'\in(2,\infty)$. Then letting
$$\theta_{1,1}:=\frac{2N+2\mathfrak{r}-N\mathfrak{r}}{N\mathfrak{r}},\,\,\,\,\,\,\,
r_{1,1}:=2(1+\theta_{1,1}),\,\,\,\,\,\,\,s_{1,1}:=(1+\theta_{1,1})\mathfrak{r},$$
$$\theta_{1,2}:=\frac{4\breve{\tilde{p}}+4\breve{\kappa}_{\tilde{p}}N-\breve{\tilde{p}}\breve{\kappa}_{\tilde{p}}N}{\breve{\tilde{p}}\breve{\kappa}_{\tilde{p}}N}
,\,\,\,\,\,\,\,r_{1,2}:=(1+\theta_{1,2})\breve{\kappa}_{\tilde{p}},\,\,\,\,\,\,\,s_{1,2}:=(1+\theta_{1,2})\breve{\tilde{p}},$$
$$\theta_{2,1}:=\frac{2d+2\mathfrak{s}-N\mathfrak{s}}{N\mathfrak{s}},\,\,\,\,\,\,\,
r_{2,1}:=2(1+\theta_{2,1}),\,\,\,\,\,\,\,s_{2,1}:=(1+\theta_{2,1})\mathfrak{s},$$
$$\theta_{2,2}:=\frac{4\breve{\tilde{q}}+4\breve{\kappa}_{\tilde{q}}d-\breve{\tilde{q}}\breve{\kappa}_{\tilde{q}}N}
{\breve{\tilde{q}}\breve{\kappa}_{\tilde{q}}N},\,\,\,\,\,\,\,
r_{2,2}:=(1+\theta_{2,2})\breve{\kappa}_{\tilde{q}},\,\,\,\,\,\,\,s_{2,2}:=(1+\theta_{2,2})\breve{\tilde{q}},$$
we get that\\[2ex]
$\Upsilon^2_{\varsigma}(\psi u_k)\,\leq\,cL_{\psi'}\displaystyle\int^0_{-\varsigma}\|u_k(t)\|^2_{_{2,\Omega_k(t)}}\,dt+$\\
\begin{equation}\label{NL04}
+ck^2\displaystyle\sum^{2}_{\xi=1}\left(\displaystyle\int^0_{-\varsigma}|\Omega_k(t)|
^{^{\frac{r_{1,\xi}}{s_{1,\xi}}}}\,dt\right)^{^{\frac{2(1+\theta_{1,\xi})}{r_{1,\xi}}}}
+ck^2\displaystyle\sum^{2}_{\zeta=1}\left(\displaystyle\int^0_{-\varsigma}\mu(\Gamma_k(t))
^{^{\frac{r_{2,\zeta}}{s_{2,\zeta}}}}\,dt\right)^{^{\frac{2(1+\theta_{2,\zeta})}{r_{2,\zeta}}}}.
\end{equation}
From here, we proceed exactly as in the proof of the first step in Theorem \ref{Thm-Partial-Bounded}, and the rest
of the proofs run unchanged. This is all what we need to verify, so we are done.
\end{proof}

\begin{remark}
    In Theorem \ref{Nonlocal-estimate}(b), observe that we are not guaranteeing the existence of a weak solution, as the procedure in subsection \ref{subsec04-02} may not be applicable for nonlocal data as in problem \eqref{1.06B}. The crucial change is that for this nonlocal version of the problem, the operator $\mathbf{A}$ defined in Definition \ref{a-normal-a}(d) needs to be extended to the space $H^{-\mathfrak{s}/2}(\Omega)\times H^{-r_d/2}(\Gamma)$. Since it is unknown whether the dual of the fractional Sobolev space possesses order continuous norm (see \cite[Appendix C]{AR-BA-HIE-NEU} for a precise definition), it is unknown whether $\mathbf{A}$ may generate a once integrated semigroup (Cf. \cite[Theorem 3.11.7]{AR-BA-HIE-NEU}), and thus Proposition \ref{reg-sol-class} may not hold in this case. In the case $\mathcal{A}$ and $\mathcal{B}$ are symmetric of Laplace-type structure, an alternative path has been recently considered in \cite{G-M-S-AVS24}.
\end{remark}

\subsection{Applications of nonlocal maps}\label{subsec07-04}

\indent The presence of nonlocal boundary value problems have experienced a substantial growth in interest and investigation, motivated by
many applications and advances in the field. For instance, nonlocal Besov-type operators as in subsection \ref{subsec07-01} are
important maps to showcase the presence of ``fractional" diffusion along $\Omega$ (in the case of the map $\mathcal{J}_{_{\Omega}}$
given by (\ref{Besov-Interior-Map})), and over $\Gamma$ (in the case of the map $\Theta_{_{\Gamma}}$ given by (\ref{Besov-Interior-Map})).
In fact, the inclusion of these kind of nonlocal maps provide a bridge to consider the realization of fractional-type boundary value problems
involving unbounded coefficients over rough domains, which is rather new as of the present time (refer to \cite{WAR15-1} in the case of bounded coefficients).
As a more concrete example, we recall the classes of ramified domains given in subsection \ref{subsec5-03}, which are useful in the simulation
of the diffusion of medical sprays in the bronchial tree. Then, since oxygen diffusion between the lungs and the circulatory system takes place
only in the last generation of the lungs tree, a reasonable diffusion model may need to involve nonlocal Robin boundary conditions containing
the Besov boundary operator $\Theta_{_{\Gamma}}$ given by (\ref{Besov-Interior-Map}). The Besov nonlocal maps are also of great value to establish a transition
between nonlocal diffusion equations and the corresponding nonlocal diffusion processes which could be of great value in applications to stable-like processes
(see for example \cite{B-B-C-K09,CHEN-KUM03}).\\
\indent The Dirichlet-to-Neumann map has been a very useful tool to deal with multiple inverse problems, and has also fundamental in the study of
some models related to electrical impedance tomography and conductivity (e.g. \cite{JONN99,JONS97,SILV-UHLM}). Recently in \cite{E-V-U-T20},
the authors developed a model relying in the Dirichlet-to-Neumann map to deal a low-frequency electromagnetic modeling which is computationally
accurate and efficient. Furthermore, in \cite{HUA-LU-LI07,LI-WANG11}, the Dirichlet-to-Neumann map has been a fundamental tool used to develop an
efficient numerical method for the modal analysis of two-dimensional photonic crystal waves (in the area of optics).\\
\indent In views of the above applications, we take the opportunity to comment on the conclusion of Theorem \ref{Nonlocal-estimate}, which establishes
the existence and regularity theory of problems (\ref{1.05B}) and (\ref{1.06B}). In these boundary value problems, the inhomogeneous interior and boundary
data are given in terms of the nonlocal Besov operators defined in subsection \ref{subsec07-01}, which correspond to interesting examples of
boundary value problems with ``fractional-type" data. These kind of model equations are rather new, especially in the time-dependent setting,
and provide the ground for further analysis and applications of problems with fractional data.\\

\noindent{\bf Acknowledgements.}\, We would like to thank Dr. Ives Achdou and Dr. Nicoletta Tchou for their help in providing us with suitable images of ramified domains.\\

\noindent The first author has been supported by: {\it the Gruppo
Nazionale per l'Analisi Matematica, la Probabilit\`a e le loro
Applicazioni (GNAMPA) of the Istituto Nazionale di Alta Matematica
(INdAM)}. The second author was supported by:
{\it The Puerto Rico Science, Technology and Research Trust}, under agreement number 2022-00014, and by: {\it University of Puerto Rico - R\'io Piedras FIPI grant 2024-2026}, award number 20FIP1570025.\\

\indent\\

\noindent{\bf Disclaimer.}  {\it This content is only the responsibility of the authors and does not necessarily represent the official views of The Puerto Rico Science, Technology and Research Trust}.\\

\end{document}